\documentclass[pdflatex,iicol,sn-basic]{sn-jnl}
\pdfoutput=1

\usepackage{lipsum}
\usepackage{amsmath}
\usepackage{cool}
\usepackage{mathtools}
\usepackage{dirtytalk}
\usepackage{cuted}
\usepackage{relsize}
\usepackage[font=small]{caption}
\usepackage{subcaption}

\usepackage{ifthen}
\usepackage[normalem]{ulem}

\newboolean{AnnotateChanges}
\setboolean{AnnotateChanges}{false} 

\ifthenelse{\boolean{AnnotateChanges}}{%
\definecolor{annotateColor}{RGB}{0,0,255}
}{%
\definecolor{annotateColor}{RGB}{0,0,0} 
}

\newcommand{\added}[2][]{%
    \ifthenelse{\boolean{AnnotateChanges}}{%
        \textcolor{annotateColor}{#2}\textbf{\textcolor{red}{#1}}}{#2}}
        
\newcommand{\deleted}[2][]{%
    \ifthenelse{\boolean{AnnotateChanges}}{%
        \textcolor{annotateColor}{\sout{#2}}\textbf{\textcolor{red}{#1}}}{}}
        
\newcommand{\replaced}[3][]{%
    \ifthenelse{\boolean{AnnotateChanges}}{%
        \textcolor{annotateColor}{#2\sout{#3}}\textbf{\textcolor{red}{#1}}}{#2}}

\newcommand{\varr}[1]{\mathlarger{\delta}{\mathsmaller{\tilde{#1}}}}

\algrenewcommand\algorithmicrequire{\textbf{Initialisation:}}

\usepackage{stfloats}

\jyear{2023}

\theoremstyle{thmstyleone}%
\theoremstyle{thmstyletwo}%

\theoremstyle{thmstylethree}%
\newtheorem{definition}{Definition}%
\newtheorem{lemma}{Lemma}
\newtheorem{weak}{Weak form}

\begin{document}

\title[\textcolor{white}{.}]{A Hilbertian projection method for constrained level set-based topology optimisation}

\author*[1]{\fnm{Zachary J.} \sur{Wegert}}\email{zach.wegert@hdr.qut.edu.au}
\author[1]{\fnm{Anthony P.} \sur{Roberts}}\email{ap.roberts@qut.edu.au}
\author*[1]{\fnm{Vivien J.} \sur{Challis}}\email{vivien.challis@qut.edu.au}

\affil[1]{\orgdiv{School of Mathematical Sciences}, \orgname{Queensland University of Technology}, \orgaddress{\street{2 George St}, \city{Brisbane}, \postcode{4000}, \state{QLD}, \country{Australia}}}

\abstract{We present an extension of the projection method proposed by Challis et al. (Int J Solids Struct 45(14--15):4130--4146, 2008) for constrained level set-based topology optimisation that harnesses the Hilbertian velocity extension-regularisation framework. Our Hilbertian projection method chooses a normal velocity for the level set function as a linear combination of (1) an orthogonal projection operator applied to the extended optimisation objective shape sensitivity and (2) a weighted sum of orthogonal basis functions for the extended constraint shape sensitivities. This combination aims for the best possible first-order improvement of the optimisation objective in addition to first-order improvement of the constraints. Our formulation utilising basis orthogonalisation naturally handles linearly dependent constraint shape sensitivities. Furthermore, use of the Hilbertian extension-regularisation framework ensures that the resulting normal velocity is extended away from the boundary and enriched with additional regularity. Our approach is generally applicable to any topology optimisation problem to be solved in the level set framework. We consider several benchmark constrained microstructure optimisation problems and demonstrate that our method is effective with little-to-no parameter tuning. We also find that our method performs well when compared to a Hilbertian sequential linear programming method.
}


\maketitle

\section{Introduction}\label{sec: intro}
The field of topology optimisation has enjoyed rapid growth owing to improved computing power, new optimisation techniques, and application to a wide range of design problems \citep{TopOptMonograph, DeatonGrandhi2013, TopOptReviewSigmund}. Classical computational methods for topology optimisation include density-based methods \citep{Bendsoe89, Rozvanyetal1992} in which the design variables are material densities of elements/nodes in a mesh, and level set-based methods \citep{10.1016/S0045-7825(02)00559-5_2003,10.1016/j.jcp.2003.09.032_2004} in which the boundary of the shape is implicitly tracked as the zero level set of a higher dimensional level set function. Conventional level set methods rely on the Hamilton-Jacobi evolution equation to update the design according to a normal velocity field defined on the boundary \cite[e.g.,][]{10.1016/S0045-7825(02)00559-5_2003,10.1016/j.jcp.2003.09.032_2004}. An important aspect of these methods is extending the normal velocity away from the boundary. To this end, the Hilbertian velocity extension-regularisation framework, which is well known in the context of level set methods (see discussion by \cite{10.1016/bs.hna.2020.10.004_978-0-444-64305-6_2021}), can be used to generate a velocity field that guarantees a descent direction and has additional regularity (smoothness) over the whole computational domain. 

Topology optimisation problems often include multiple constraints. In density-based topology optimisation \citep{Bendsoe89, Rozvanyetal1992} the application of constraints is usually straightforward and handled by the optimisation algorithm (e.g., Method of Moving Asymptotes \citep{MMASvanberg}). 
In the context of conventional level set-based methods, applying constraints is more complicated.
The \textit{augmented Lagrangian method} \citep{978-0-387-30303-1_2006,10.1007/978-0-387-74759-0_517_978-0-387-74759-0_2009} is a classical approach for constrained optimisation problems. It converts a constrained optimisation problem into a sequence of unconstrained problems that are a combination of the classical Lagrangian method and quadratic penalty method. In a level set framework, applying the method is straightforward: the shape sensitivity of the augmented Lagrangian is used to inform the normal velocity for the Hamilton-Jacobi equation \cite[e.g.,][]{10.1016/j.cma.2014.01.010_2014,10.1007/s00158-016-1453-y_2016,10.3390/app11125578_2021}. However, the difficulty associated with tuning the accompanying parameters is problem dependent and scales with the number of constraints \citep{10.1016/bs.hna.2020.10.004_978-0-444-64305-6_2021}. The level-set \textit{sequential linear programming method} (SLP) \citep{10.1179/1743284715Y.0000000022_2015,10.1007/s00158-014-1174-z_2015} involves linearising the optimisation problem into a number of sub-problems that are then solved using a linear programming method (e.g., the simplex method \citep{10.1016/j.compstruc.2020.106265_2020}). For level set-based topology optimisation, applying SLP is fairly straightforward except for implementing appropriate trust region constraints \citep{10.1007/s00158-014-1174-z_2015}. 
The \textit{projection method} \citep{10.1016/j.cma.2003.10.008_2004,10.1016/j.ijsolstr.2008.02.025_2008} projects the objective shape sensitivity onto a space that will leave the constraints unchanged and combines this with constraint shape sensitivities. This approach has been used to successfully design material microstructures subject to isotropy constraints \citep{10.1016/j.ijsolstr.2008.02.025_2008} but has not been widely adopted in the literature. However, similar methods have more recently been proposed in the literature by \cite{10.3934/dcdsb.2019249} and \cite{10.1051/cocv/2020015_2020} for level set topology optimisation. The projection method and these two recent works are examples of general null-space gradient methods \cite[e.g.,][]{978-0-387-30303-1_2006}.

It is natural to consider methods of constrained optimisation that take advantage of the Hilbertian extension-regularisation framework. For example, \cite{10.1016/bs.hna.2020.10.004_978-0-444-64305-6_2021} recently presented an SLP method in the Hilbertian framework. In this paper we revisit the projection method from \citet{10.1016/j.ijsolstr.2008.02.025_2008} and combine it with the Hilbertian extension-regularisation procedure. Our method constructs an orthogonal basis that spans the set of extended constraint shape sensitivities and an orthogonal projection operator that projects onto the set perpendicular to the extended constraint shape sensitivities. We then define the normal velocity for the level set function as a linear combination of the orthogonal projection operator applied to the extended objective function shape sensitivity and a weighted sum of basis functions for the extended constraint shape sensitivities. This normal velocity is naturally extended onto the bounding domain and endowed with additional regularity due to the Hilbertian extension-regularisation. While our method is similar to other recently proposed approaches   \citep{10.3934/dcdsb.2019249,10.1051/cocv/2020015_2020}, our formulation utilising an orthogonal basis provides significant benefits.

To demonstrate our presented \emph{Hilbertian projection method} we consider several linear elastic microstructure optimisation (i.e., inverse homogenisation) problems with multiple constraints. The constraints naturally arise under the enforcement of symmetries for the effective material properties, such as isotropy. Irrespective of optimisation method, microstructure optimisation has been used successfully for a range of design problems including linear elastic materials with extremal properties \cite[e.g.,][]{10.1016/S0022-5096(99)00043-5_2000,10.1016/j.mechmat.2013.09.018_2014}, multifunctional composites \cite[e.g.,][]{10.1016/j.ijsolstr.2008.02.025_2008}, auxetic materials \cite[e.g.,][]{10.1016/j.cad.2016.09.009_2017}, piezoelectric materials \cite[e.g.,][]{10.1016/S0045-7825(98)80103-5_1998,10.1016/j.ijsolstr.2022.111666_2022}, and multi-material composites \cite[e.g.,][]{10.1080/03052150903443780_2010,10.1007/s00158-017-1688-2_2017}. In this work we consider maximising the bulk modulus with and without isotropy constraints and the design of auxetic and multi-phase materials. We compare our Hilbertian projection method with a Hilbertian sequential linear programming (SLP) method \cite[Sec. 5.3.2,][]{10.1016/bs.hna.2020.10.004_978-0-444-64305-6_2021} and show that the Hilbertian projection method is able to successfully handle these optimisation problems with little-to-no parameter tuning.

The remainder of the paper is as follows. In Section \ref{sec: math background} we discuss the mathematical background for the level set method, Hilbertian extension-regularisation procedure, and linear elastic microstructure optimisation. In Section \ref{sec: HPM} we formulate the Hilbertian projection method and compare our formulation to the null space method presented by \citet{10.1051/cocv/2020015_2020}. In Section \ref{sec: numerical impl} we discuss our numerical implementation. In Section \ref{sec: eg problems} we present and discuss the example optimisation problems and results. Finally, in Section \ref{sec: concl} we present our concluding remarks.

\section{Mathematical background}\label{sec: math background}
In this section we give a brief introduction to the level set method for topology optimisation and shape derivatives. We then discuss the Hilbertian extension-regularisation framework. 
We conclude by describing linear elastic microstructure optimisation for single- and multi-phase materials.

\subsection{The level set method}\label{subsec: bg - LS method}
Level set methods track the boundary of a domain $\Omega$ inside a bounding domain $D\subset\mathbb{R}^d$ implicitly via the zero level set of a function $\phi:D\rightarrow\mathbb{R}$ \citep{978-0-521-57202-6_1996,978-0-387-22746-7_2006}. For a domain $\Omega$ inside a bounding domain $D$, the level set function $\phi$ is typically defined as 
\begin{equation}
\begin{cases}\phi(\boldsymbol{x})<0&\text{if } \boldsymbol{x} \in \Omega, \\ \phi(\boldsymbol{x})=0 &\text{if } \boldsymbol{x} \in \partial \Omega, \\\phi(\boldsymbol{x})>0 &\text{if } \boldsymbol{x} \in D \backslash \bar{\Omega}.\end{cases}
\end{equation}
Using this definition and assuming that the interface may evolve in time, a material derivative of $\phi$ on $\partial\Omega$ gives 
\begin{equation}
    \pderiv{\phi}{t}(t,\boldsymbol{x}) + v(t,\boldsymbol{x})\lvert\nabla\phi(t,\boldsymbol{x})\rvert = 0,
\end{equation}
where $v$ is the normal velocity of the interface. In practice, the above is solved over the whole bounding domain $D$ instead of only on the interface $\partial\Omega$ by extending the velocity $v$ away from the boundary. Assuming that the time interval $(0,T)$ is small so that the velocity does not vary in time gives the Hamilton-Jacobi evolution equation \citep{978-0-521-57202-6_1996,978-0-387-22746-7_2006,10.1016/bs.hna.2020.10.004_978-0-444-64305-6_2021}:
\begin{align}\label{eqn: HJ}
    \begin{cases}
    \pderiv{\phi}{t}(t,\boldsymbol{x}) + v(\boldsymbol{x})\lvert\nabla\phi(t,\boldsymbol{x})\rvert = 0,\\
    \phi(0,\boldsymbol{x})=\phi_0(\boldsymbol{x}),\\
    \boldsymbol{x}\in D,~t\in(0,T),
    \end{cases}
\end{align}
where $\phi_0(\boldsymbol{x})$ is the initial condition for $\phi$ at $t=0$.

It is often useful to reinitialise the level set function as the \textit{signed distance function} $d_\Omega$. This ensures that the level set function is neither too steep nor too flat near the boundary of $\Omega$ \citep{978-0-387-22746-7_2006}. The signed distance function may be defined as \citep{10.1016/bs.hna.2020.10.004_978-0-444-64305-6_2021}:
\begin{equation}
d_{\Omega}(\boldsymbol{x})=\begin{cases}
-d(\boldsymbol{x}, \partial \Omega)&\text{if }\boldsymbol{x} \in \Omega, \\
0&\text{if } \boldsymbol{x} \in \partial \Omega, \\
d(\boldsymbol{x}, \partial \Omega)&\text{if } \boldsymbol{x} \in D \backslash \bar{\Omega},
\end{cases}
\end{equation}
where $d(\boldsymbol{x}, \partial \Omega):=\min _{\boldsymbol{p} \in \partial \Omega}\lvert\boldsymbol{x}-\boldsymbol{p}\rvert$ is the minimum Euclidean distance from $\boldsymbol{x}$ to the boundary $\partial \Omega$. Several methods are available for constructing the signed distance function and the reader is referred to \cite{978-0-387-22746-7_2006} and \cite{10.1016/bs.hna.2020.10.004_978-0-444-64305-6_2021} and the references therein for a detailed discussion. In this work we use the following \textit{reinitialisation equation} \citep{10.1006/jcph.1999.6345_1999,978-0-387-22746-7_2006} to reinitialise a pre-existing level set function $\phi_0(\boldsymbol{x})$ as the signed distance function:
\begin{align}\label{eqn: reinit}
\begin{cases}
    \pderiv{\phi}{t}(t,\boldsymbol{x}) + S(\phi_0(\boldsymbol{x}))\left(\lvert\nabla\phi(t,\boldsymbol{x})\rvert-1\right) = 0,\\
    \phi(0,\boldsymbol{x})=\phi_0(\boldsymbol{x}),\\
    \boldsymbol{x}\in D, t>0.
\end{cases}
\end{align}
Here $S$ is the sign function and Equation~\ref{eqn: reinit} is solved until close to steady state. Similar numerical schemes may be used to solve for Hamilton-Jacobi evolution and reinitialisation (Eqs.~\ref{eqn: HJ} and \ref{eqn: reinit}).

\subsection{Shape derivatives}\label{subsec: bg - LS shape deriv}
To find a normal velocity $v$ that reduces some functional $J(\Omega)$ via solution of the Hamilton-Jacobi equation (Eq.~\ref{eqn: HJ}) we use the notion of shape derivatives. We recall the following useful results from \cite{10.1016/j.jcp.2003.09.032_2004,10.1016/bs.hna.2020.10.004_978-0-444-64305-6_2021}. 

Suppose that we consider smooth variations of the domain $\Omega$ of the form $\Omega_{\boldsymbol{\theta}} =(\boldsymbol{I}+\boldsymbol{\theta})(\Omega)$, where $\boldsymbol{\theta} \in W^{1,\infty}(\mathbb{R}^d,\mathbb{R}^d)$. Then the following definition and lemma follow:
\begin{definition}[\cite{10.1016/j.jcp.2003.09.032_2004}]
The shape derivative of $J(\Omega)$ at $\Omega$ is defined as the Fr\'echet derivative in $W^{1, \infty}(\mathbb{R}^d, \mathbb{R}^d)$ at $\boldsymbol{\theta}$ of the application $\boldsymbol{\theta} \rightarrow J(\Omega_{\boldsymbol{\theta}})$, i.e.,
\begin{equation}\label{eqn: shape deriv defin}
J(\Omega_{\boldsymbol{\theta}})(\Omega))=J(\Omega)+J^{\prime}(\Omega)(\boldsymbol{\theta})+\mathrm{o}(\boldsymbol{\theta})  
\end{equation}
with $\lim _{\boldsymbol{\theta} \rightarrow 0} \frac{\lvert\mathrm{o}(\boldsymbol{\theta})\rvert}{\|\boldsymbol{\theta}\|}=0,$ where the shape derivative $J^{\prime}(\Omega)$ is a continuous linear form on $W^{1, \infty}(\mathbb{R}^d, \mathbb{R}^d)$.
\end{definition}
\begin{lemma}[\cite{10.1016/j.jcp.2003.09.032_2004}]\label{lemma: omega int scalar}
Let $\Omega$ be a smooth bounded open set and $f \in W^{1,1}(\mathbb{R}^d)$. Define
\begin{equation}
J(\Omega)=\int_{\Omega} f~\mathrm{d} \Omega .
\end{equation}
Then $J$ is differentiable at $\Omega$ and
\begin{equation}
J^{\prime}(\Omega)(\boldsymbol{\theta})=\int_{\Omega} \operatorname{div}(\boldsymbol{\theta} f)~ \mathrm{d} \Omega=\int_{\partial \Omega} f~\boldsymbol{\theta} \cdot \boldsymbol{n} ~ \mathrm{d} \Gamma
\end{equation}
for any $\boldsymbol{\theta} \in W^{1, \infty}(\mathbb{R}^d, \mathbb{R}^d)$.
\end{lemma}

C\'ea's formal method \citep{10.1051/m2an/1986200303711_1986} can be applied to find the shape derivative of a functional $J$ that depends on fields that satisfy specified state equations \cite[e.g.,][]{10.1016/j.jcp.2003.09.032_2004,10.1016/bs.hna.2020.10.004_978-0-444-64305-6_2021}. The method relies on defining a Lagrangian functional $\mathcal{L}$ that satisfies the two following properties:
\begin{enumerate}
    \item The state equations are generated by stationarity of  $\mathcal{L}$ under variations of the fields.
    \item $\mathcal{L}$ is equal to the functional of interest  $J$ at the solution to the state equations.
\end{enumerate}
Once these properties are satisfied the shape derivative of the functional of interest can be found using Lemma \ref{lemma: omega int scalar} \citep{10.1016/j.jcp.2003.09.032_2004}. 


\subsection{Hilbertian extension- regularisation}\label{subsec: bg - Hilb. ext.-reg.}
To infer a descent direction from $J^\prime(\Omega)$ we utilise the Hilbertian extension-regularisation method as discussed by \cite{10.1016/bs.hna.2020.10.004_978-0-444-64305-6_2021}. This involves solving an identification problem over a Hilbert space $H$ on $D$ with inner product $\langle\cdot,\cdot\rangle_H$: \textit{Find $g_\Omega\in H$ such that }
\begin{equation}
    \langle g_\Omega,w\rangle_H=-J^{\prime}(\Omega)(w\boldsymbol{n})~\forall w\in H.
\end{equation}
For an unconstrained optimisation problem the resulting field $g_\Omega$ is the extended shape sensitivity that is used to evolve the interface with $\boldsymbol{\theta}=\tau g_\Omega\boldsymbol{n}$ where $\tau>0$ is sufficiently small. 

The Hilbertian extension-regularisation method provides two important benefits: it naturally extends the shape sensitivity from $\partial\Omega$ onto the bounding domain $D$, and ensures a descent direction for $J(\Omega)$ with additional regularity (i.e., $H$ as opposed to $L^2(\partial\Omega)$) \citep{10.1016/bs.hna.2020.10.004_978-0-444-64305-6_2021}. As discussed by \cite{10.1016/bs.hna.2020.10.004_978-0-444-64305-6_2021}, this may be viewed as an analog to the sensitivity filtering used in density-based topology optimisation algorithms. 

A common choice for the Hilbert space $H$ is $H^1(D)$ with the inner product 
\begin{equation}\label{eqn: H1 inner product}
    \langle u,v\rangle_{H}=\beta^2\int_D\nabla u\cdot\nabla v~\mathrm{d}\Omega+\int_Duv~\mathrm{d}\Omega,
\end{equation}
where $\beta$ is the so-called regularisation length scale \cite[e.g.,][]{10.1007/s00158-016-1453-y_2016,10.1007/s40324-018-00185-4_2019,10.1007/s10957-021-01928-6_2021}. For microstructure optimisation we use the periodic Sobolev space $H=H^1_{\text{per}}(D)$ and use the inner product defined in Equation \ref{eqn: H1 inner product}.

\subsection{Linear elastic microstructure optimisation}\label{subsec: bg - LE Micro}
In this section we briefly discuss computational homogenisation and topology optimisation in the context of periodic microstructure design.

The state equations for linear elastic homogenisation over a domain $\Omega$ contained in a representative volume element (RVE) $D\subset\mathbb{R}^d$ under an applied strain field $\bar{\varepsilon}_{i j}$ are \cite[e.g.,][]{YvonnetCompHomogenization}
\begin{align}
    -\sigma_{ij,i}&= 0~\text{in }\Omega,\label{eqn:le1}\\
    \sigma_{ij}n_j&= 0~\text{on }\partial\Omega,\\
    \sigma_{i j}&=C_{i j k l} \varepsilon_{k l},\\
    \varepsilon_{ij}&=\frac{1}{2}\left(u_{i,j}+u_{j,i}\right), \\
    \frac{1}{\operatorname{Vol}(D)}\int_{\Omega}{\varepsilon_{kl}}~\mathrm{d}\Omega&=\bar{\varepsilon}_{kl},
    \label{eqn:lelast}
\end{align}
where $\sigma_{i j}$ is the stress tensor, $\varepsilon_{i j}=\varepsilon_{i j}(\boldsymbol{u})$ is the $D$-periodic strain field with displacement $\boldsymbol{u}$, and $C_{i j k l}$ is the spatially dependent elasticity tensor. Note that in the above we use summation notation for indices and comma notation for derivatives. 

To compute the homogenised stiffness tensor $\bar{C}_{ijkl}$ of a periodic material, the above state equations are solved over $\Omega$ for three ($d=2$) or six ($d=3$) different combinations of macroscopic strain fields. These macroscopic strain fields are applied by decomposing the strain into the constant macroscopic strain field and fluctuation strain field as $\varepsilon_{i j}=\bar{\varepsilon}_{i j}+\tilde{\varepsilon}_{i j}$. The macroscopic strain fields are then given by the unique components of $\bar{\varepsilon}_{i j}^{(k l)}=\frac{1}{2}\left(\delta_{i k} \delta_{j l}+\delta_{i l} \delta_{j k}\right)$ in $k$ and $l$. For example, in two dimensions the unique macroscopic strains are: $\bar{\varepsilon}_{i j}^{(11)}$, $\bar{\varepsilon}_{i j}^{(22)}$, $\bar{\varepsilon}_{i j}^{(12)}$. The notation $\tilde{\varepsilon}_{i j}^{(k l)}$ is used to denote the strain field fluctuation arising from the applied strain field $\bar{\varepsilon}_{i j}^{(k l)}$.

In practice, Equations \ref{eqn:le1}-\ref{eqn:lelast} are solved using a finite element method and the weak formulation given here:
\begin{weak}\label{wf: LE hom}
\textit{For each unique constant macroscopic strain field $\bar{\varepsilon}_{i j}^{(k l)}$, find $\tilde{\boldsymbol{u}}^{(kl)}\in H^1_{\textrm{per}}(\Omega)^d$ such that
\begin{equation}
    \begin{aligned}
    &\int_{\Omega} C_{pqrs} \varepsilon_{rs}(\tilde{\boldsymbol{u}}^{(k l)}) \varepsilon_{pq}(\boldsymbol{v}) ~\mathrm{d} \Omega\\& \hspace{0.5cm} = -\int_{\Omega} C_{pqrs} \bar{\varepsilon}_{rs}^{(kl)} \varepsilon_{pq}(\boldsymbol{v}) ~\mathrm{d} \Omega~~\forall\, \boldsymbol{v}\in H^1_{\textrm{per}}(\Omega)^d\label{eqn: le weak form}
    \end{aligned}
\end{equation}
where $\varepsilon_{ij}(\boldsymbol{v})=\frac{1}{2}\left(v_{i,j}+v_{j,i}\right).$}
\end{weak}

\subsubsection{Single-phase problems}\label{subsubsec: LE - single phase theory}
For single-phase problems (one solid and a void phase), once the solution $\tilde{\boldsymbol{u}}^{(ij)}$ to Weak Form \ref{wf: LE hom} has been found for each unique macroscopic strain $\bar{\varepsilon}_{pq}^{(ij)}$, the resulting homogenised stiffness tensor may be computed via \citep{YvonnetCompHomogenization}
\begin{equation}
    \bar{C}_{ijkl}(\Omega) = \int_\Omega C_{pqrs}({\varepsilon}_{pq}(\tilde{\boldsymbol{u}}^{(ij)})+\bar\varepsilon^{(ij)}_{pq})\bar\varepsilon_{rs}^{(kl)}~\mathrm{d}\Omega
    \label{eqn: le hom},
\end{equation}
assuming that $\operatorname{Vol}(D)=1$.

To evaluate Weak Form \ref{wf: LE hom} and the homogenised stiffness tensor above, we utilise the \textit{ersatz material approximation}. This method, which is classical in the literature \cite[e.g.,][]{10.1016/j.jcp.2003.09.032_2004}, fills the void phase with a soft material so that the state equations can be resolved without a body-fitted mesh. To this end, for small $\varepsilon_{\mathrm{void}}$ we take
\begin{equation}\label{eqn: sharp ersatz}
    C_{ijkl}(\boldsymbol{x}) = \begin{cases}
        C_{ijkl},& \boldsymbol{x}\in \Omega\\
        \varepsilon_{\mathrm{void}} C_{ijkl},& \boldsymbol{x}\in D\setminus\Omega
    \end{cases}
\end{equation}
and relax integration to be over $D$. We can provide a smooth approximation to Equation \ref{eqn: sharp ersatz} using a smoothed Heaviside function $H_{\eta}$
\begin{equation}
    H_\eta(\phi)=\begin{cases}
    0&\text{if }\phi<-\eta,\\
    \frac{1}{2}+\frac{\phi}{2\eta}+\frac{1}{2\pi}\sin\left(\frac{\pi\phi}{\eta}\right)&\text{if }\lvert\phi\rvert\leq\eta,\\
    1&\text{if }\phi>\eta,
    \end{cases}
\end{equation}
where $\eta$ is half the length of the small transition region of $H_\eta(\phi)$ between 0 and 1. Equation \ref{eqn: sharp ersatz} can then be replaced with
\begin{equation}\label{eqn: smooth ersatz}
    C_{ijkl}(\phi) = C_{ijkl}(1-H_\eta(\phi)) + \varepsilon_{\mathrm{void}} C_{ijkl} H_\eta(\phi).
\end{equation}
It is important to note that the ersatz material approximation is consistent \citep{10.1016/bs.hna.2020.10.004_978-0-444-64305-6_2021}. That is, as $\varepsilon_{\mathrm{void}}\rightarrow0$, the approximation becomes exact.

We conclude this section by stating the shape derivative of the homogenised stiffness tensor: 
\begin{lemma}
The shape derivative of Equation \ref{eqn: le hom} is given by
\begin{equation}
\begin{aligned}
    &\bar{C}_{ijkl}^{\prime}(\Omega)(\boldsymbol{\theta})= \int_{\partial\Omega}C_{pqrs}({\varepsilon}_{pq}(\tilde{\boldsymbol{u}}^{(ij)})+\bar{\varepsilon}_{pq}^{(ij)})\\&\hspace{2cm}\times({\varepsilon}_{rs}(\tilde{\boldsymbol{u}}^{(kl)})+\bar{\varepsilon}_{rs}^{(kl)})~\boldsymbol{\theta}\cdot\boldsymbol{n}~\mathrm{d}{\Gamma}.
\end{aligned}
\end{equation}
\end{lemma}
\begin{proof}
See Appendix \ref{secA1}.
\end{proof}

\subsubsection{Multi-phase problems}\label{subsubsec: LE - multi phase theory}
For multi-phase problems, we utilise the colour level set method in which up to $2^M$ phases can be represented by the sign of $M$ level set functions \citep{10.1016/j.cma.2003.10.008_2004,10.1051/cocv/2013076_2014}. For example, in the case of four phases $\Omega_1,\Omega_2, \Omega_3$ and $\Omega_4$ with two level set functions $\phi_1$ and $\phi_2$, we have
\begin{equation}
    \begin{cases}
        \phi_1<0~\&~\phi_2>0,&x\in\Omega_1\\
        \phi_1>0~\&~\phi_2<0,&x\in\Omega_2\\
        \phi_1<0~\&~\phi_2<0,&x\in\Omega_3\\
        \phi_1>0~\&~\phi_2>0,&x\in\Omega_4
    \end{cases}.
\end{equation}
We further denote the domains associated with each level set function $\phi_1$ and $\phi_2$ as $\mathcal{D}_1$ and $\mathcal{D}_2$, respectively. 
\begin{figure}[t]
    \centering
    \begin{subfigure}{0.94\linewidth}
		\includegraphics[width=\textwidth]{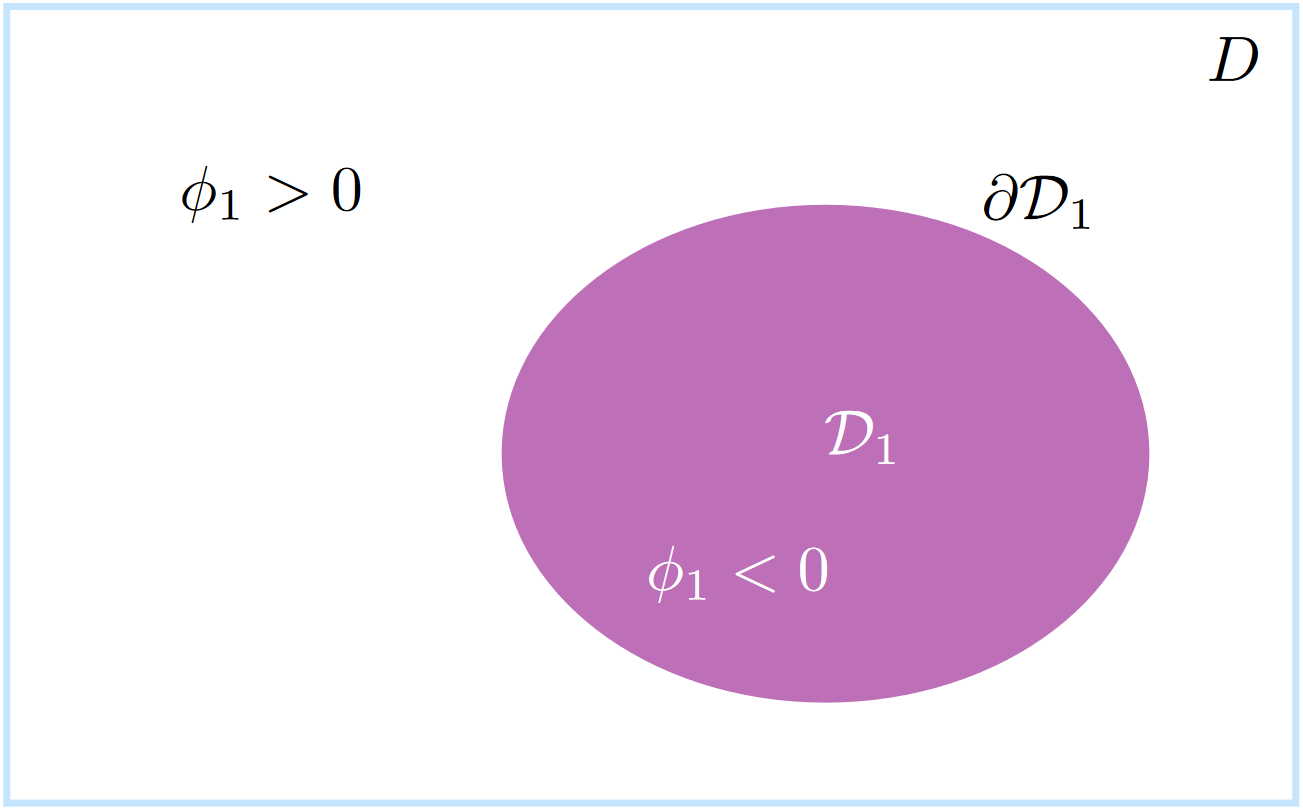}
        \caption{}
        \label{fig:Fig 1a}
    \end{subfigure}
    \begin{subfigure}{0.94\linewidth}
        \includegraphics[width=\textwidth]{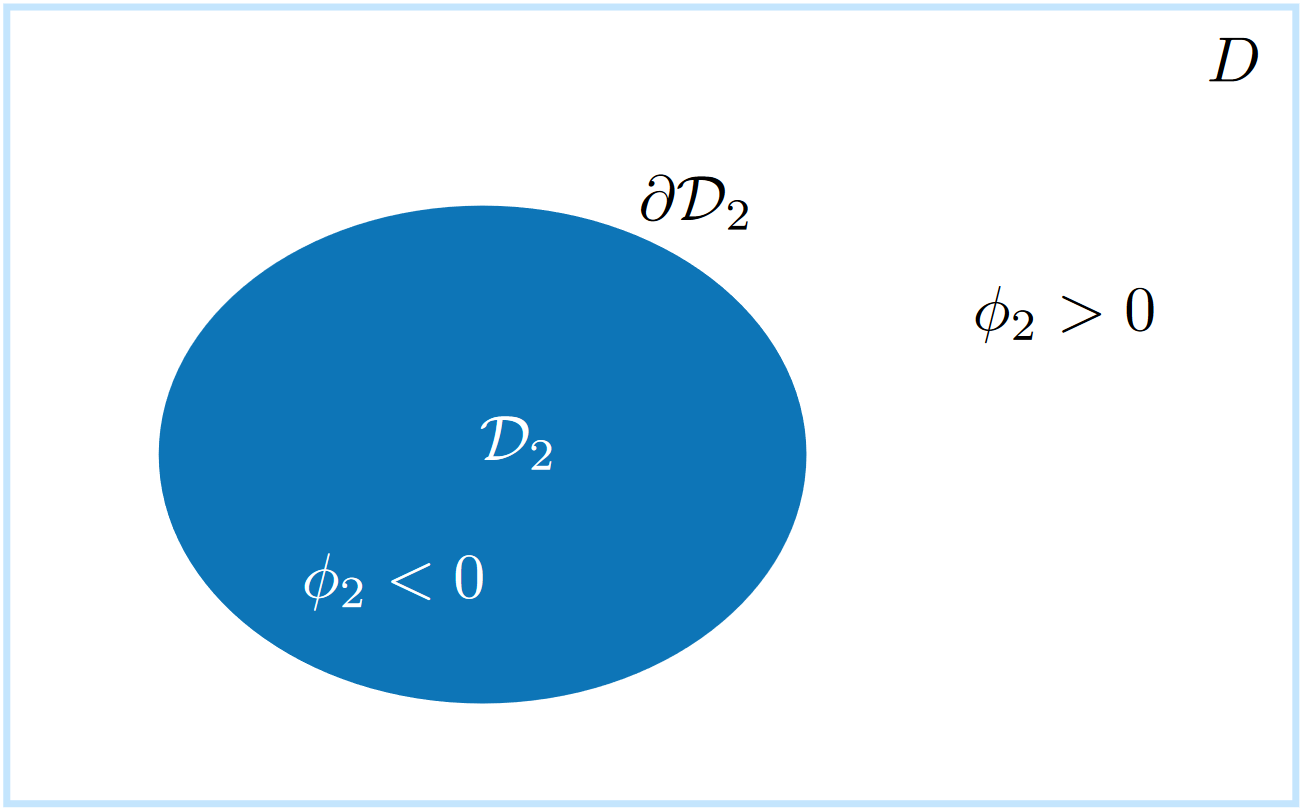}
        \caption{}
        \label{fig:Fig 1b}
    \end{subfigure}
    \begin{subfigure}{0.94\linewidth}
        \includegraphics[width=\textwidth]{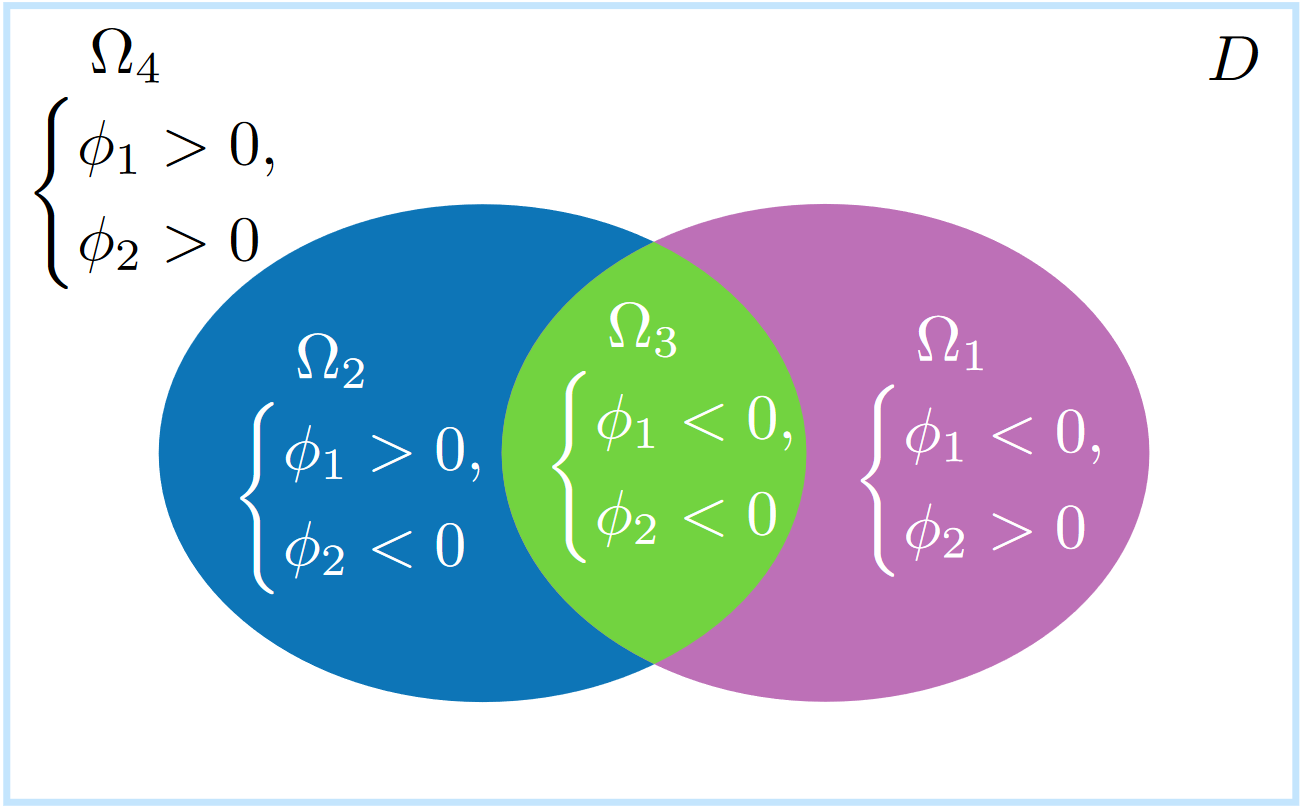}
        \caption{}
        \label{fig:Fig 1c}
    \end{subfigure}
    \caption{An illustration of colour level sets with two level set functions and four phases.
    (a) and (b) show the domain $\mathcal{D}_1$ and $\mathcal{D}_2$ represented via the level set function $\phi_1$ and $\phi_2$ respectively. (c) shows the colour representation of $\Omega_1,$ $\Omega_2$, $\Omega_3$, and $\Omega_4$ for different signs of $\phi_1$ and $\phi_2$.}
    \label{fig:Fig 1}
\end{figure}
Figure \ref{fig:Fig 1} shows an illustration of this case. 

In a similar way to Equation \ref{eqn: smooth ersatz} we can interpolate the value of the stiffness tensor between these domains via
\begin{equation}\label{eqn: multiphase interp}
    \begin{aligned}
        C_{ijkl}(d_{\mathcal{D}_1},d_{\mathcal{D}_2})&\\&\hspace{-2cm}=C_{ijkl,1}(1-H_{\eta}(d_{\mathcal{D}_1}))H_{\eta}(d_{\mathcal{D}_2}) \\&\hspace{-2cm}~+ C_{ijkl,2} H_{\eta}(d_{\mathcal{D}_1})(1-H_{\eta}(d_{\mathcal{D}_2})) \\&\hspace{-2cm}~+ C_{ijkl,3} (1-H_{\eta}(d_{\mathcal{D}_1}))(1-H_{\eta}(d_{\mathcal{D}_2})) \\&\hspace{-2cm}~+ C_{ijkl,4} H_{\eta}(d_{\mathcal{D}_1})H_{\eta}(d_{\mathcal{D}_2}),
    \end{aligned}
\end{equation}
where $C_{ijkl,\alpha}$ is the elasticity tensor for the phase occupying $\Omega_{\alpha}$. In this multi-phase case we have replaced the level set functions $\phi_1$ and $\phi_2$ with $d_{\mathcal{D}_1}$ and $d_{\mathcal{D}_2}$ denoting their respective signed distance functions. This change facilitates shape differentiation. Unlike the situation discussed by \cite{10.1007/s00158-016-1453-y_2016}, periodicity of $D$ ensures that this replacement is valid provided the level set functions are reinitialised often. Additional care should then be taken regarding the calculation of certain quantities and their shape derivatives. Namely, the homogenised stiffness tensor (Eq. \ref{eqn: le hom}) becomes
\begin{equation}
\begin{aligned}
    &\bar{C}_{ijkl}(\mathcal{D}_1,\mathcal{D}_2)&\\\hspace{-1.9cm}&= \int_D C_{pqrs}(d_{\mathcal{D}_1},d_{\mathcal{D}_2})({\varepsilon}_{pq}(\tilde{\boldsymbol{u}}^{(ij)})+\bar\varepsilon^{(ij)}_{pq})\bar\varepsilon_{rs}^{(kl)}~\mathrm{d}\Omega.
\end{aligned}
    \label{eqn: le hom multiphase}
\end{equation}
The volume of $\Omega_1$ is given by
\begin{equation}\label{eqn: volume 1}
    \operatorname{Vol}_{\Omega_1}(\mathcal{D}_1,\mathcal{D}_2)=\int_D (1-H_{\eta}(d_{\mathcal{D}_1}))H_{\eta}(d_{\mathcal{D}_2})~\mathrm{d}\Omega.
\end{equation}
Similar expressions are used for $\Omega_2$, $\Omega_3$ and $\Omega_4$. 

Integration over the whole cell $D$ in Equations \ref{eqn: le hom multiphase} and \ref{eqn: volume 1} with dependence on the signed distance functions differs from the single-phase case where integration is over the domain $\Omega$ occupied by the solid phase (see Eq.~\ref{eqn: le hom}). 
Shape differentiability of the signed distance function and a coarea formula can be used in this multi-phase case to derive the shape derivative. We utilise the \say{approximate} formula discussed by \cite{10.1051/cocv/2013076_2014}. This assumes that the Heaviside smoothing parameter $\eta$ is small and that the principal curvatures of $\partial\Omega$ vanish.
\begin{lemma}
    The approximate shape derivatives of Equations \ref{eqn: le hom multiphase} and \ref{eqn: volume 1} under variation of the domain $\mathcal{D}_1$ by $\boldsymbol{\theta}_1$ are
\begin{equation}
\begin{aligned}
    &\bar{C}_{ijkl}^{\prime}(\mathcal{D}_1,\mathcal{D}_2)(\boldsymbol{\theta}_1)\\&\quad\approx-\int_{\partial\mathcal{D}_1}\frac{\partial C_{pqrs}}{\partial g}({\varepsilon}_{pq}(\tilde{\boldsymbol{u}}^{(ij)})+\bar{\varepsilon}_{pq}^{(ij)})\\&\quad\quad\quad\quad\quad
    \times({\varepsilon}_{rs}(\tilde{\boldsymbol{u}}^{(kl)})+\bar{\varepsilon}_{rs}^{(kl)})~\boldsymbol{\theta}_1\cdot\boldsymbol{n}~\mathrm{d}{\Gamma},
\end{aligned}
\end{equation}
where $g=H_\eta(d_{\mathcal{D}_1})$, and 
\begin{equation}
\begin{aligned}
    \operatorname{Vol}^{\prime}_{\Omega_1}(\mathcal{D}_1,\mathcal{D}_2)(\boldsymbol{\theta}_1)\approx\int_{\partial\mathcal{D}_1} H_{\eta}(d_{\mathcal{D}_2})~\boldsymbol{\theta}_1\cdot\boldsymbol{n}~\mathrm{d}\Gamma.
\end{aligned}
\end{equation}
Analogous expressions follow for $\boldsymbol{\theta}_2$ and $\Omega_2$, $\Omega_3$ and $\Omega_4$.  
\end{lemma}
\begin{proof}
See Appendix \ref{secA2}.
\end{proof}

We note that comparisons between the \say{true} formula, \say{Jacobian-free} formula (zero principal curvatures), and \say{approximate} formula have been discussed for compliance elsewhere in the literature \citep{10.1051/cocv/2013076_2014,10.1007/s00158-016-1453-y_2016}. It suffices to mention that \cite{10.1007/s00158-016-1453-y_2016} found that the \say{approximate} formula does not capture the distortion that arises due to the ray integration and approximation of the principal curvatures in the \say{true} formula. 

\section{Hilbertian projection method}\label{sec: HPM}
The Hilbertian framework yields a descent direction $\boldsymbol{\theta}=\tau g_\Omega\boldsymbol{n}$ for unconstrained optimisation problems. However, for constrained optimisation problems such as
\begin{equation}\label{eqn:optim prob}
          \begin{aligned}
            \underset{\Omega\in\mathcal{U}_{\mathrm{ad}}}{\min}~
              &J(\Omega)\\
            \text{s.t.}~~ & C_i(\Omega)=0,~i=1,\dots,N,\\
            &a(\boldsymbol{u},\boldsymbol{v})=l(\boldsymbol{v}),~\forall \boldsymbol{v}\in V,
          \end{aligned}
\end{equation}
the choice of $\boldsymbol{\theta}$ is more difficult.

In the literature a variety of optimisation methods deal with this problem but few of these take advantage of the Hilbertian framework. \cite{10.1016/bs.hna.2020.10.004_978-0-444-64305-6_2021} recently presented a sequential linear programming (SLP) method in the Hilbertian framework. The projection method uses orthogonal projections to evolve the design in a direction that aims for best possible improvement of the objective functional while improving the constraint functionals \citep{10.1016/j.ijsolstr.2008.02.025_2008}. In the following we present a Hilbertian extension of the projection method for constrained topology optimisation. 

\subsection{Preliminaries}
We proceed by first solving the following set of scalar Hilbertian extension-regularisation problems over $H$ for an objective functional $J(\Omega)$ and constraint functionals $C_i(\Omega)$:

\textit{Find $g_\Omega\in H$ and $\mu_{\Omega i}\in H$ such that 
\begin{align}
    \langle g_\Omega,v\rangle_H&=-J^{\prime}(\Omega)(v\boldsymbol{n}),~\forall v\in H,\text{ and}\label{eqn: hilbertian obj}\\
    \langle\mu_{\Omega i},v\rangle_H&=-C_i^{\prime}(\Omega)(v\boldsymbol{n}),~\forall v\in H,  \label{eqn: hilbertian consts}
\end{align}
for all $i=1,\dots,N,$ with inner product $\langle\cdot,\cdot\rangle_H$ and norm $\lVert\cdot\rVert_H=\sqrt{\langle\cdot,\cdot\rangle_H}$.}
 
Next we use Gram-Schmidt orthogonalisation to remove linearly dependent constraints from the set $\{\mu_{\Omega i}\}_{i=1}^{N}$ to obtain the set $\{\mu_{\Omega p}\}_{p=1}^{\bar{N}}$, where $\bar{N}\leq N$. We use $\{\bar{\mu}_{\Omega p}\}$ to denote the corresponding orthogonal basis  that spans the set $C\subset H$ of extended constraint shape sensitivities. The basis $\{\bar{\mu}_{\Omega p}\}$ can be used to construct an orthogonal projection operator $P_{C^\perp}$ that projects the shape sensitivity $g_\Omega$ onto the set $C^\perp$ perpendicular to the set of extended constraint shape sensitivities. We define this operator as
\begin{equation}\label{eqn: proj operator}
    P_{C^\perp}{g_\Omega} = g_\Omega - \sum_{p=1}^{\bar{N}}\frac{\langle\bar{\mu}_{\Omega p},g_\Omega\rangle_H}{\lVert\bar{\mu}_{\Omega p}\rVert_H^2}\bar{\mu}_{\Omega p}.
\end{equation}
Then, by construction, evolving the level set function using normal velocity $P_{C^\perp}{g_\Omega}$ would to first order improve the objective functional $J(\Omega)$ while leaving the constraint functionals $C_i(\Omega)$ unchanged. On the other hand, the set of basis functions $\{\bar{\mu}_{\Omega p}\}$ describes directions that to first order improve the constraint functionals $C_i(\Omega)$.

\subsection{Formulation}
The Hilbertian projection method can then be formulated as follows: For some rate parameter $\lambda\in\mathbb{R}$, suppose we choose $v_\Omega\in H$ in the deformation field $\boldsymbol{\theta}=\tau v_\Omega \boldsymbol{n}$ so that $J(\Omega)$ and $C_i(\Omega)$ decrease via \citep{10.1016/j.ijsolstr.2008.02.025_2008}
\begin{equation}\label{eqn: proj assumption}
        \left\{\begin{aligned}
    &J^\prime(\Omega)(v_\Omega \boldsymbol{n}) = \text{ min possible},\\
    &C_i^\prime(\Omega)(v_\Omega \boldsymbol{n}) = -\lambda C_i.
\end{aligned}\right.
\end{equation}
It is important to note that we purposefully pose the former requirement as \say{min possible} so that the objective functional may increase when required to improve constraints.
\added[{[R2-3]}]{Furthermore, linearly dependent constraints in the optimisation problem need to be consistent to ensure that the second line of Equation~\ref{eqn: proj assumption} is well posed. Specifying constraints that have linearly dependent shape sensitivities but for which the directions of improvement are in contradiction would violate this requirement.}

In the Hilbertian framework, Equation \ref{eqn: proj assumption} may be rewritten as 
\begin{equation}\label{eqn: hilb. proj assumption}
    \left\{\begin{aligned}
            &\langle g_\Omega,v_\Omega\rangle_H = \text{ max possible},\\
            &\langle \mu_{\Omega i},v_\Omega\rangle_H = \lambda C_i.
        \end{aligned}\right.
\end{equation}
Note that the change in sign comes from the application of Equations \ref{eqn: hilbertian obj} and \ref{eqn: hilbertian consts}. We choose the following linear combination for $v_\Omega$:
\begin{equation}\label{eqn: proj eqn}
    v_{\Omega}=\sqrt{1-\sum_{p=1}^{\bar{N}} \alpha_p^2} \frac{P_{C^{\perp}} g_{\Omega}}{\|P_{C^{\perp}} g_{\Omega}\|_H}+\sum_{p=1}^{\bar{N}} \alpha_p \frac{\bar{\mu}_{\Omega p}}{\|\bar{\mu}_{\Omega p}\|_H},
\end{equation}
where $\alpha_p\in\mathbb{R}$ are \replaced[{[R2-3]}]{determined}{arbitrary at this stage. To first order the first term of Equation \ref{eqn: proj eqn} improves the objective while leaving the constraints unchanged due to orthogonality, while the second term improves the constraints.
The square root term is included to facilitate a balance between improving the objective and constraints. We can then determine $\alpha_p$} using $\langle \mu_{\Omega p},v_\Omega\rangle_H = \lambda C_p$ for $p=1,\dots,\bar{N}$. This generates a lower-triangular linear system of the form 
\begin{equation}
    \lambda C_p = \sum_{l=1}^{p-1} \alpha_l\frac{\langle\bar{\mu}_{\Omega l}, \mu_{\Omega p}\rangle_H}{\|\bar{\mu}_{\Omega l}\|_H} + \alpha_p\|\bar{\mu}_{\Omega p}\|_H,
\end{equation}
which can easily be solved via forward substitution \citep{10.1016/j.ijsolstr.2008.02.025_2008}.

\added[{[R2-3]}]{To first order the first term of Equation \ref{eqn: proj eqn} improves the objective while leaving the constraints unchanged due to orthogonality, while the second term improves the constraints with extended shape sensitivities that contribute to the basis $\{\bar{\mu}_{\Omega p}\}$. In the numerical examples we have observed satisfaction of all constraints at convergence of the optimisation algorithm, including those that have linearly dependent shape sensitivities. The square root term of Equation \ref{eqn: proj eqn} is included to facilitate a balance between improving the objective and constraints. }

\subsection{Parameters}\label{sec: params}
As discussed by \cite{10.1016/j.ijsolstr.2008.02.025_2008}, the rate parameter $\lambda$ should be chosen to ensure {$1-\sum_{p=1}^{\bar{N}} \alpha_p^2 \geq 0$} and $\sum_{p=1}^{\bar{N}} \alpha_p^2 \geq \alpha_{\textrm{min}}^2$. The new parameter $\alpha_{\textrm{min}}$ then controls the balance between improving the objective or constraints. For example, $\alpha_{\mathrm{min}}=1$ ignores the objective function in Equation \ref{eqn:optim prob} and instead solves a constraint satisfaction problem. As a result, the method only has a single parameter $\alpha_{\textrm{min}}$ while $\lambda$ is dictated by the inequalities above. In general we find that $\alpha_{\textrm{min}}$ does not require fine tuning and unless otherwise stated we choose $\alpha^2_{\mathrm{min}}=0.1$.\\

\subsection{Comparison to null space methods}\label{sec: compar null}

Our formulation is similar to null space methods recently developed by \cite{10.3934/dcdsb.2019249} and \cite{10.1051/cocv/2020015_2020}, both of which present a similar formulation. In the following we discuss some differences between our Hilbertian projection method and the null space method proposed by \cite{10.1051/cocv/2020015_2020}.

Most notably, our formulation makes use of an orthogonal basis for the set of extended constraint shape sensitivities. This avoids a possibly expensive matrix inversion that appears in the algorithm presented by \citep{10.1051/cocv/2020015_2020}. Our use of the orthogonal basis also avoids reliance on the ‘Linear Independence Constraint Qualification’ (LICQ) condition \cite[Sec. 2.1,][]{10.1051/cocv/2020015_2020}. Our method therefore naturally handles linearly dependent constraint shape sensitivities. Such dependencies often appear in microstructure optimisation problems when symmetries are imposed on the effective material properties (e.g., Secs.~\ref{subsec: eg problems - Bmod iso} and \ref{subsec: eg problems - multi-phase results} below). Multi-phase level set-based topology optimisation via the colour level set method \citep{10.1016/j.cma.2003.10.008_2004,10.1051/cocv/2013076_2014} can also give rise to such linear dependency (e.g., Sec.~\ref{subsec: eg problems - multi-phase results} below). The ability of the Hilbertian projection method to naturally handle these situations gives the user more freedom in how topology optimisation problems are posed and avoids additional special treatment of linearly dependent constraint sensitivities.

\replaced[{[R2-1 \& R2-2]}]{For equality constrained problems when LICQ is satisfied, both the null space method \citep{10.1051/cocv/2020015_2020} and our Hilbertian project method give equivalent directions for improvement of the objective and violated constraints, up to coefficients $\alpha_p$ and constant $\lambda$. H}{Both the Hilbertian projection method and the aforementioned null space method (Feppon et al, 2020) result in exponential improvement of violated constraints, h}owever\added{,} the method of attaining this improvement is quite different. In this work, the second term of Equation \ref{eqn: proj eqn} is a linear combination of the orthogonal basis of the set of extended constraint shape sensitivities. The coefficients $\alpha_p$ are chosen to improve constraints exponentially, as per the second requirement in Equation \ref{eqn: hilb. proj assumption}. The null space instead uses a Gauss-Newton direction for ensuring exponential decay of violated constraints \cite[Lemma 2.5 \& Prop. 2.6,][]{10.1051/cocv/2020015_2020}. This again relies on LICQ and possibly expensive matrix inversion discussed above.

\deleted[{[R2-4]}]{Another difference presents in the choice of method parameter/s. The approach described by 
Feppon et al (2020) uses two parameters that dictate the size of the contributions from first and second term in their descent direction. In contrast, our formulation is a single-parameter method where $\alpha_{\mathrm{min}}$ dictates the balance between the first and second term in the descent direction (Eqn. \ref{eqn: proj eqn}). A benefit of the latter is the clear interpretation of $\alpha_\mathrm{min}$ discussed in Section \ref{sec: params}.}

Finally, unlike the null space methods proposed by \cite{10.3934/dcdsb.2019249} and \cite{10.1051/cocv/2020015_2020}, the Hilbertian projection method as formulated above is unable to handle inequality constraints. Such an extension could be considered in future using slack variables or by adopting the procedure used by either \cite{10.3934/dcdsb.2019249} or \cite{10.1051/cocv/2020015_2020}. Interestingly, \added{in terms of the total length covered to reach the optimum,} \cite{10.1051/cocv/2020015_2020} \replaced{found}{find} that their dual quadratic programming method for handling inequality constraints \replaced{yielded}{yields} equivalent performance when compared to the method of slack variables. \added{However, using slack variables introduces additional computational cost and possibly further parameter tuning \citep{10.1051/cocv/2020015_2020}. For our implementation this additional cost should be small owing to the use of orthogonalisation. As such the slack variable method would be an appropriate first recourse for implementing inequality constraints within the Hilbertian projection method. This would be similar to the approach taken by \cite{doi:10.1080/01630560008816971}.}
\deleted[{[R2-6]}]{This combined with the small cost of using slack variables in our orthogonalisation framework suggests that the slack variable method would be an appropriate first recourse for implementing inequality constraints within the Hilbertian projection method.}

\section{Numerical implementation}\label{sec: numerical impl}
In the following we describe the numerical implementation of our topology optimisation algorithm. We first discuss the resolution of state equations and Hamilton-Jacobi type equations followed by an overview of the optimisation algorithm. We finish with a brief discussion of the Hilbertian SLP method which is compared to our presented Hilbertian projection method.

\subsection{Resolving state and Hamilton-Jacobi-type equations}
To solve the state equations and the Hilbertian extension-regularisation problems we use the finite element package \textit{Gridap} \citep{Badia2020,Verdugo2022} in the programming language \textit{Julia}. In particular, we discretise the periodic domain $D\subset\mathbb{R}^d$ into $n^d$ linear quadrilateral ($d$=2) or hexahedral ($d$=3) elements with element width $\Delta x$ and discretise the level set function at the nodes of the triangulation. To reduce computational cost when solving the state equations, we remove any elements that are completely void phase and leave a strip of ersatz material near the phase interface. The resulting linear systems for the state equations and Hilbertian extension-regularisation problems are then solved using a direct method in 2D or a GPU-based Jacobi pre-conditioned conjugate gradient method in 3D.

For the Hamilton-Jacobi evolution equation and signed distance reinitialisation equation (Eqs. \ref{eqn: HJ} and \ref{eqn: reinit}) we use standard first order Godunov upwind finite difference schemes \citep{10.1006/jcph.1999.6345_1999,10.1016/j.jcp.2003.09.032_2004,978-0-387-22746-7_2006} that have been implemented on GPUs using \textit{CUDA.jl} \citep{besard2018juliagpu}. For the Hamilton-Jacobi evolution equation we use $\lfloor n/10\rfloor$ or $\lfloor n/3\rfloor$ number of time steps in two dimensions or three dimensions, respectively. We are more conservative in three dimensions because we have less elements along each axial direction. For the reinitialisation equation we iterate until reaching a stopping condition
\begin{equation}
    \lVert\phi^{q}-\phi^{q-1}\rVert_\infty < 5\times10^{-5},
\end{equation}
where $q$ is the iteration number. In addition, for the sign function $S$ we use the common approximation 
\begin{equation}\label{eqn: approx sign}
    S(\phi)=\frac{\phi}{\sqrt{\phi^2+\lvert\nabla\phi\rvert^2\Delta x^2}},
\end{equation}
that applies on a Cartesian grid with square elements of side length $\Delta x$ \citep{978-0-387-22746-7_2006}. 

\subsection{Algorithm overview}\label{sec: algo overview}
In Algorithm \ref{alg:hpm} we present our optimisation algorithm that is based on the theory discussed in Sections \ref{sec: math background} and \ref{sec: HPM}. Algorithm \ref{alg:hpm} is similar to Algorithm 5 presented by \cite{10.1016/bs.hna.2020.10.004_978-0-444-64305-6_2021}, with the addition of a line search method for determining the Courant-Friedrichs-Lewy (CFL) coefficient \added{$\gamma$} for solving the Hamilton-Jacobi evolution equation \cite[e.g.,][]{10.1007/s10957-021-01928-6_2021} \added[{[R2-5]}]{with time step \citep{978-0-387-22746-7_2006}}
\begin{equation}
    \added{\Delta t = \frac{\gamma\Delta x}{\lVert{v_\Omega}\rVert_{\infty}}.}
\end{equation}
\added{Note that we omit the indices on $\gamma$ that appear in Algorithm \ref{alg:hpm} for sake of clarity.}
\deleted{.}In general, the line search method helps to remove oscillations in the optimisation history and improve convergence. For the stopping criterion we require that the current objective value compared to the previous five is stationary and that the constraints are satisfied within specified tolerances.

The Hilbertian projection method is implemented using the package \emph{DoubleFloats.jl} that gives a machine epsilon of roughly $5\times 10^{-32}$. This prevents accumulation of round-off error when generating the projection operator that can affect the optimisation history. All other computations are completed in standard double precision.

{\begin{algorithm}[t]
\smaller
\caption{Optimisation Algorithm}\label{alg:hpm}
\begin{algorithmic}[1]
\Require Initialise $\Omega^0$ inside a computational domain $D$ with mesh $\mathcal{T}$ and a level set function $\phi^0$. 
\State Find the initial solution $\boldsymbol{u}^{(ij)}$ to the homogenisation problem for each unique $\bar{\varepsilon}^{(ij)}$.
\For{$q = 1,\dots,q_\mathrm{max}$}
\State  \parbox[t]{190pt}{Calculate the shape sensitivity of the objective $J(\Omega^{q-1})$ and constraints $C_i(\Omega^{q-1})$.}
\State \parbox[t]{190pt}{Solve scalar Hilbertian extension-regularisation problems for the objective and constraints with length scale $\beta$.}
\State \parbox[t]{190pt}{Apply the Hilbertian projection method to find $v_\Omega$.}
\For{$\replaced{k}{m} = 1,\dots,\replaced{k}{m}_\mathrm{max}$}
    \State \parbox[t]{180pt}{Solve H-J evolution equation with CFL coefficient $\gamma^{q-1,k}$ to find new domain $\bar{\Omega}^{k}$ with associated level set function $\bar{\phi}^{k}$.}
    \State Solve reinitialisation equation.
    \State \parbox[t]{180pt}{Solve state equations for linear elastic homogenisation and calculate the new objective $J_\mathrm{new}$.}
    \If{\parbox[t]{165pt}{$J_\mathrm{new} < J(\Omega^{q-1}) + \xi\lvert J(\Omega^{q-1})\rvert$ or $\gamma^{q-1,k}<\gamma_{\mathrm{min}}$ \textbf{then}}}
        \State Increase the CFL coefficient:
        $$\gamma^{q,k} = \min(\delta_{\mathrm{inc}}\gamma^{q-1,k},\gamma_{\mathrm{max}}).$$
        \State \parbox[t]{165pt}{Accept the new iteration with $\Omega^{q}=\bar{\Omega}^{k}$ and $\phi^{q}=\bar{\phi}^{k}.$}
        \State Break inner loop.
    \Else
        \State \parbox[t]{165pt}{Decrease the CFL coefficient:
        $$\gamma^{q-1,k+1} = \max(\delta_{\mathrm{dec}}\gamma^{q-1,k},\gamma_{\mathrm{min}}).$$}
        \State \parbox[t]{165pt}{Reject the new iteration and continue loop.}
    \EndIf
\EndFor
\If{\parbox[t]{185pt}{$\lvert J(\Omega^{q})-J(\Omega^{q-j})\rvert\leq \epsilon_1\lvert J(\Omega^{q})\rvert, \forall j = 1,\dots,j_{\mathrm{max}}$ and $\lvert C_i(\Omega^{q})\rvert <\epsilon_2$ $\forall i$ \textbf{then}}}
\State \textbf{return} $\Omega^{q}$ and $\phi^{q}$.
\EndIf
\EndFor

\end{algorithmic}
\end{algorithm}}

Table \ref{tab:parameters} gives the parameter values used for all the optimisation examples unless otherwise stated in Section \ref{sec: eg problems}. 

\begin{table}[t]
\centering
\caption{Parameter values used for optimisation examples. Note: Any variation in the parameters from these values will be specifically stated in the relevant example problem section.}
\label{tab:parameters}
\setlength{\tabcolsep}{0.5em} 
\begin{tabular}{l|l|l}
Parameter                         & Value        & Type\\ \hline
$q_{\mathrm{max}}$                                       & $1000$        & Max iterations\\
$\replaced{k}{m}_{\mathrm{max}}$                                       & $10$        & Max line search iterations\\
$\Delta x$                                       & $1/n$        & Mesh spacing\\
$\varepsilon_{\mathrm{void}}$                                    & $0.001$ & Ersatz material coeff.\\
$\eta$                                           & $1.5\Delta x$ & Heaviside smoothing\\
$\beta$                                          & 4$\Delta x$  & Hilbertian ext.-reg.\\
$\alpha_{\mathrm{min}}^2$                        & 0.1          & Hilbertian proj. meth.\\
$\lambda$                                        & 0.5          & Hilbertian proj. meth.\\
$\gamma_{\mathrm{min}}$ & 0.001   & H-J equation\\
$\gamma_{\mathrm{max}}$ & 0.1  & H-J equation\\
$\gamma_{\mathrm{reinit}}$                       & 0.1          & Reinit. equation\\
$\xi$                                            & 0.005        & Line search\\
$\delta_{\mathrm{inc}}$, $\delta_{\mathrm{dec}}$ & 1.1, 0.7     & Line search\\
$\epsilon_1$, $\epsilon_2$                       & 0.01 , 0.0001 & Stopping criteria\\
$j_{\mathrm{max}}$ & 5 & Stopping criteria
\end{tabular}
\end{table}

\subsection{Comparison to sequential linear programming (SLP)}
We compare the Hilbertian projection method to the Hilbertian SLP method presented by \cite{10.1016/bs.hna.2020.10.004_978-0-444-64305-6_2021} (Sec. 5.3.2). We make two adjustments to the method. Firstly, to match our formulation we replace the inequality constraints with equality constraints. In addition, we change the trust region constraints for the constraint functionals to be
\begin{equation}
    \lvert\lambda_i\rvert\leq\frac{\Delta x}{2\|\boldsymbol{\theta}_i\|_H},
\end{equation}
for $i=1,\dots,N$. For several two-dimensional problems we find that this choice promotes convergence and better optimisation results. However, as we will discuss later, choosing trust region constraints is not straightforward for our example optimisation problems.

We implement Hilbertian SLP by adjusting line 5 of Algorithm \ref{alg:hpm} accordingly. To solve the resulting linearised optimisation problem we use the Julia packages \textit{JuMP.jl} \citep{DunningHuchetteLubin2017} and \textit{Ipopt.jl} \citep{SLP2006}.

\section{Example problems}\label{sec: eg problems}
In the following we give the optimisation results for several example problems that have been solved with both the Hilbertian projection method and Hilbertian SLP method.

\subsection{Example 1: Maximum bulk modulus}\label{subsec: eg problems - Bmod}
In this example we consider a bounding domain $D=[0,1]^d$ that contains a solid phase and void phase. The solid phase is constructed from an isotropic medium with $E=1$ and $\nu=0.3$. Subject to a volume constraint $\operatorname{Vol}(\Omega)=1/2$, we maximise the effective bulk modulus $\bar{\kappa}(\Omega)$ of the material. The bulk modulus is a measure of stiffness to volumetric strain given by
\begin{equation}
    \bar{\kappa} = \frac{1}{4}(\bar{C}_{1111}+\bar{C}_{2222}+2\bar{C}_{1122})
\end{equation}
in two dimensions, or 
\begin{equation}
    \begin{aligned}
        \bar{\kappa} = \frac{1}{9}(&\bar{C}_{1111}+\bar{C}_{2222}+\bar{C}_{3333}\\&+2(\bar{C}_{1122}+\bar{C}_{1133}+\bar{C}_{2233}))
    \end{aligned}
\end{equation}
in three dimensions. In other words, we seek to solve the optimisation problem:
\begin{equation}
    \begin{aligned}
            \underset{\Omega\in\mathcal{U}_{\mathrm{ad}}}{\min}~
              &-\bar{\kappa}(\Omega)\\
            \text{s.t.}~~ & \mathrm{Vol}(\Omega)=1/2,\\
            &a(\boldsymbol{u},\boldsymbol{v})=l(\boldsymbol{v}),~\forall \boldsymbol{v}\in V.
          \end{aligned}
\end{equation}
The last line represents the satisfaction of the state equations. 

In two dimensions we use a periodic starting structure with four equally spaced holes. For three dimensions the initial boundary between void and solid material is given by a Schwarz P minimal surface.  It is well-known that in two dimensions hole nucleation is not possible under Hamilton-Jacobi evolution \citep[e.g.,][]{10.1016/j.jcp.2003.09.032_2004}. For this reason we initialise the two-dimensional optimisation problems with more holes than required. Topological derivatives could be incorporated to rectify this, but this is outside of the scope of the current paper.  We also note that different starting structures could be used provided they are periodic and have non-zero stiffness.

\begin{figure*}[p]
    \centering
    \begin{minipage}{.2\linewidth}
    \begin{subfigure}{\linewidth}
        \includegraphics[width=\textwidth]{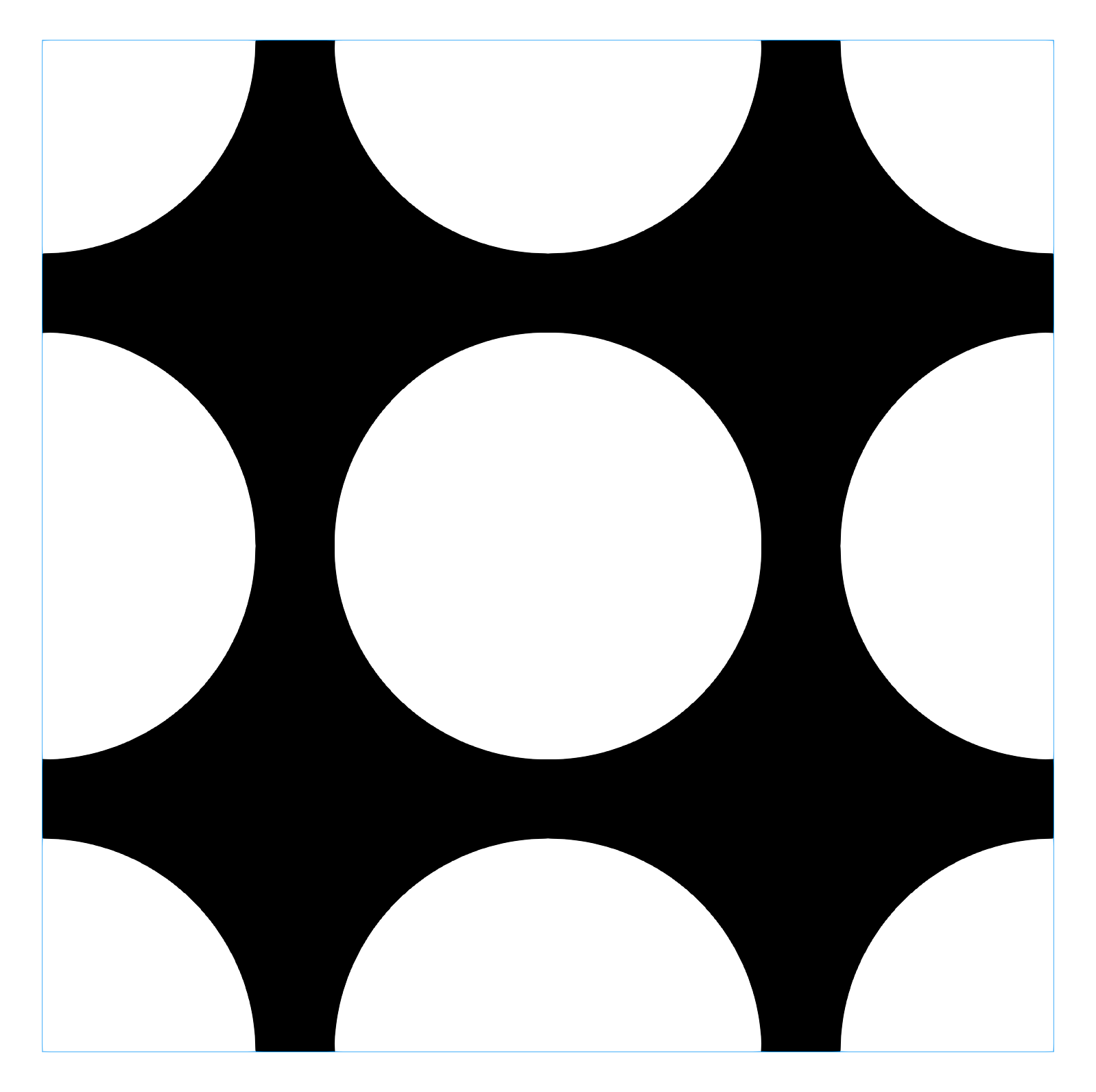}
        \caption{}
        \label{fig:Fig 2a}
    \end{subfigure}
    \end{minipage}
    \Huge{$\mathbf{\rightarrow}$}
    \begin{minipage}{.72\linewidth}
    \centering
    \begin{subfigure}{0.33\linewidth}
        \includegraphics[width=\textwidth]{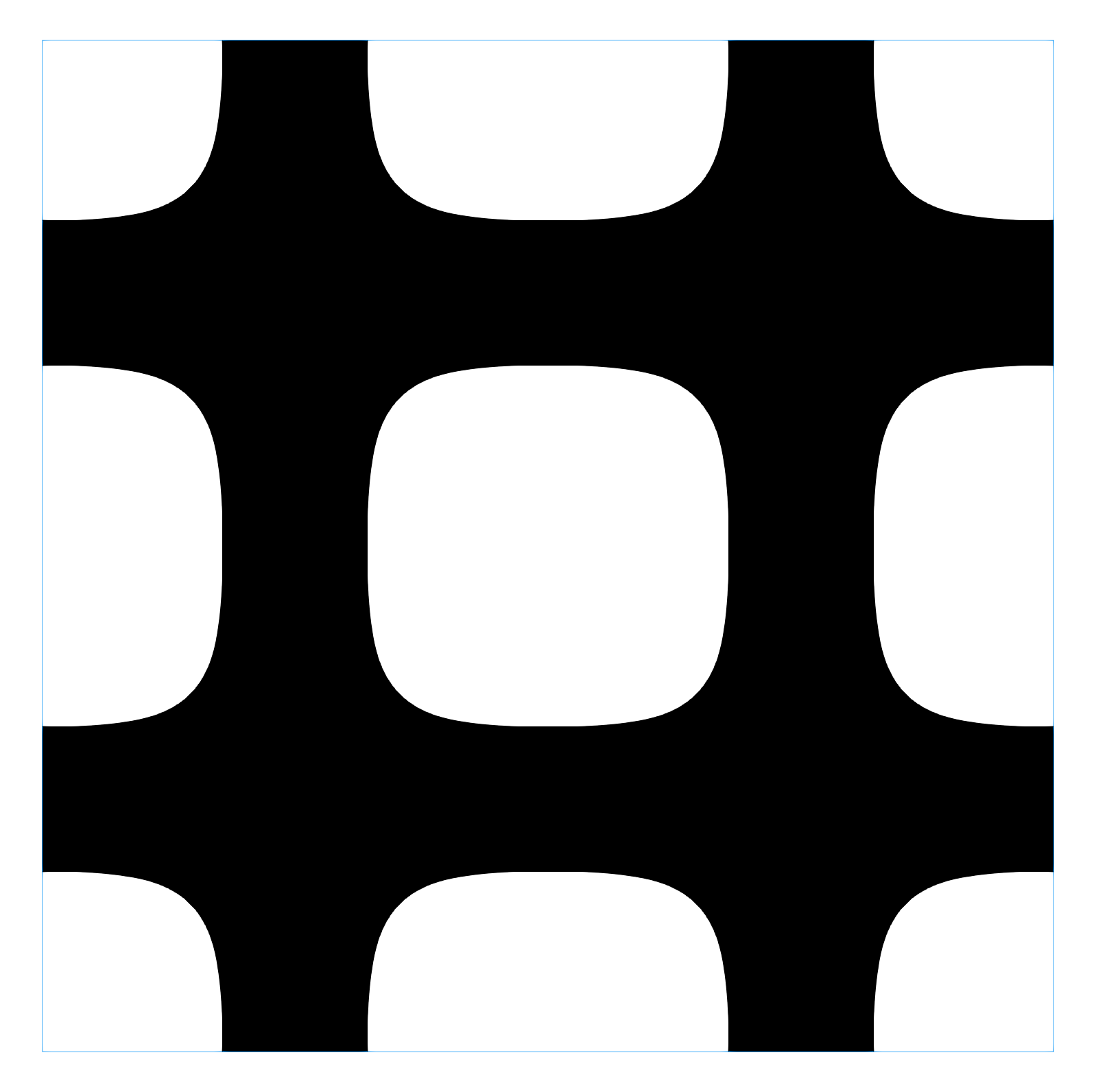}
        \caption{}
        \label{fig:Fig 2b}
    \end{subfigure}\hspace{2em}
    \begin{subfigure}{0.33\linewidth}
        \includegraphics[width=\textwidth]{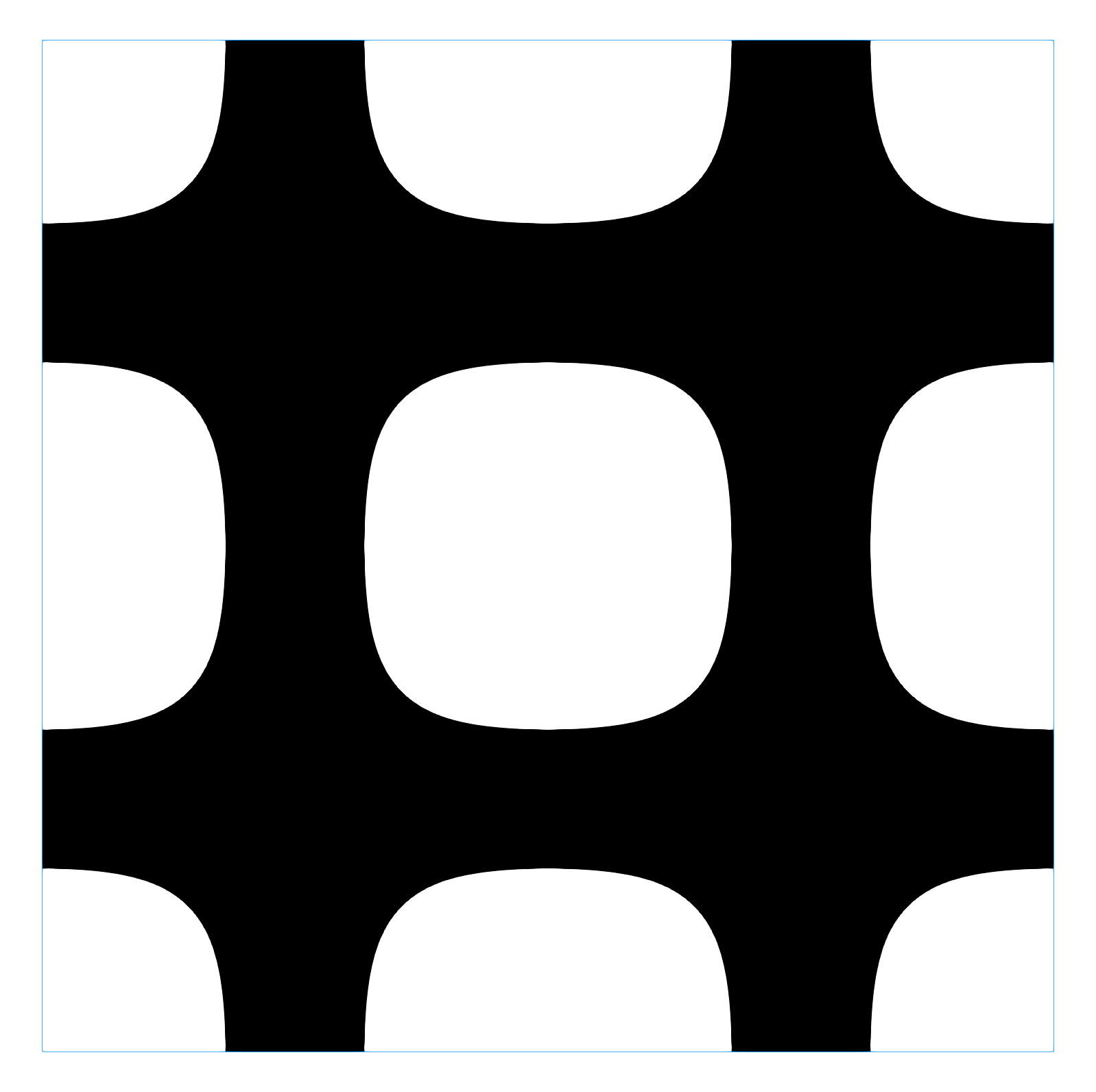}
        \caption{}
        \label{fig:Fig 2c}
    \end{subfigure}
    \begin{subfigure}{0.48\linewidth}
        \includegraphics[width=\textwidth]{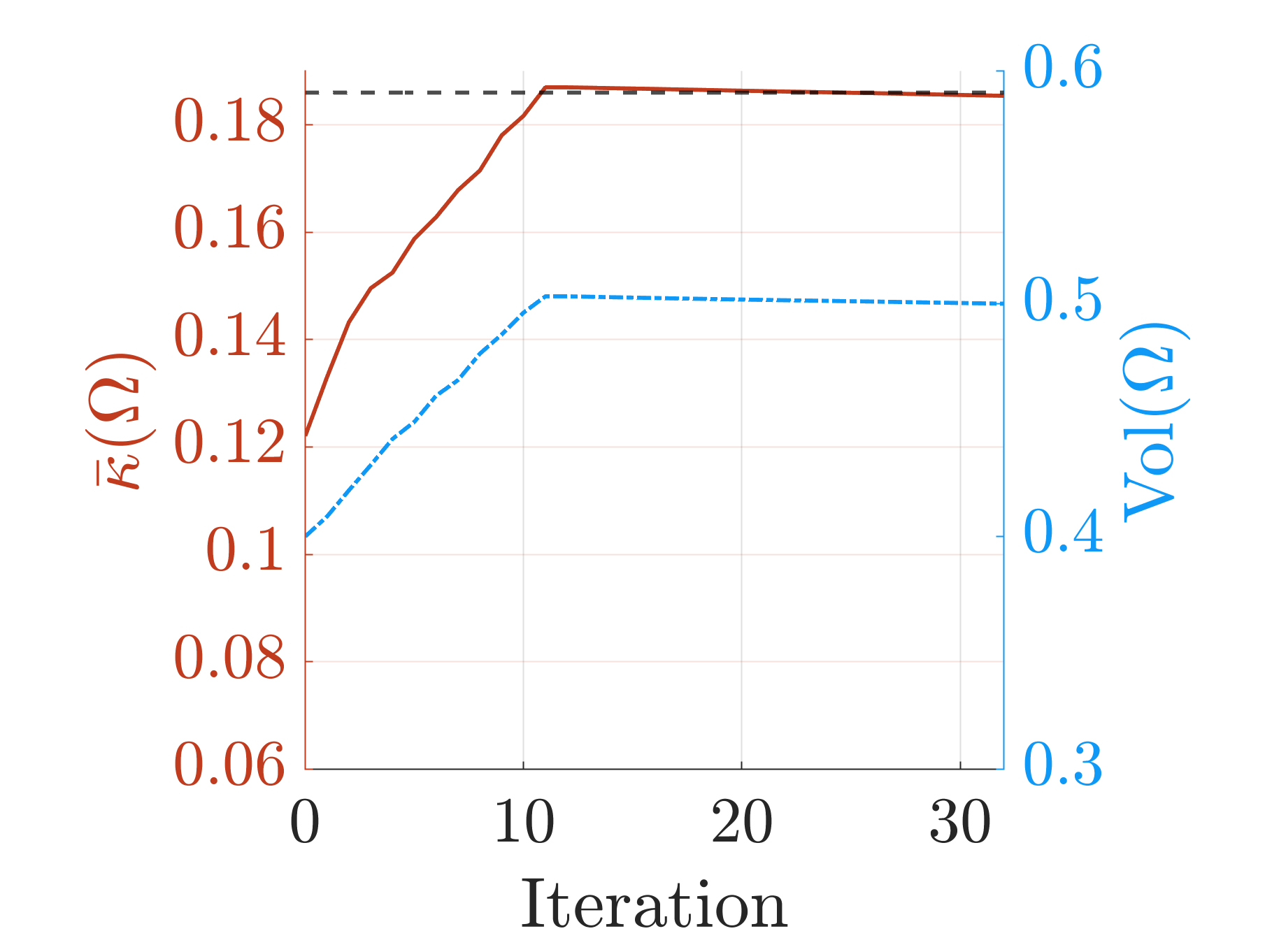}
        \caption{}
        \label{fig:Fig 2d}
    \end{subfigure}
    \begin{subfigure}{0.48\linewidth}
        \includegraphics[width=\textwidth]{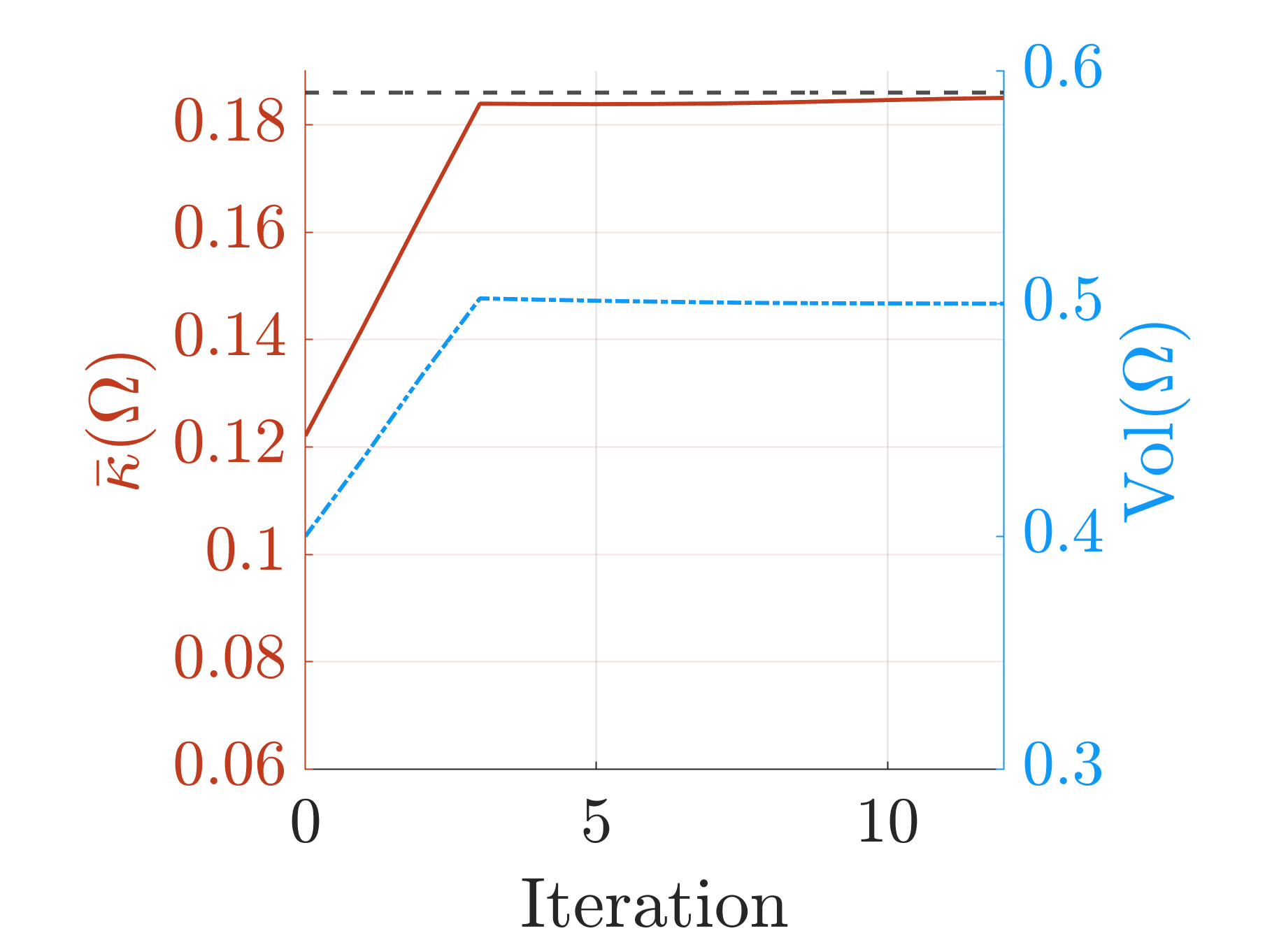}
        \caption{}
        \label{fig:Fig 2e}
    \end{subfigure}
    \end{minipage}
    \caption{Two-dimensional optimisation results for Example 1: maximum bulk modulus. For the starting structure (a), (b) and (c) show the final structures for the Hilbertian projection method and SLP respectively, while (d) and (e) show the respective iteration histories. The Hashin-Shtrikman upper bound for the bulk modulus is given by the dashed black line.}
    \label{fig:Fig 2}

    \begin{minipage}{.2\linewidth}
    \begin{subfigure}{1.1\linewidth}
        \includegraphics[width=\textwidth]{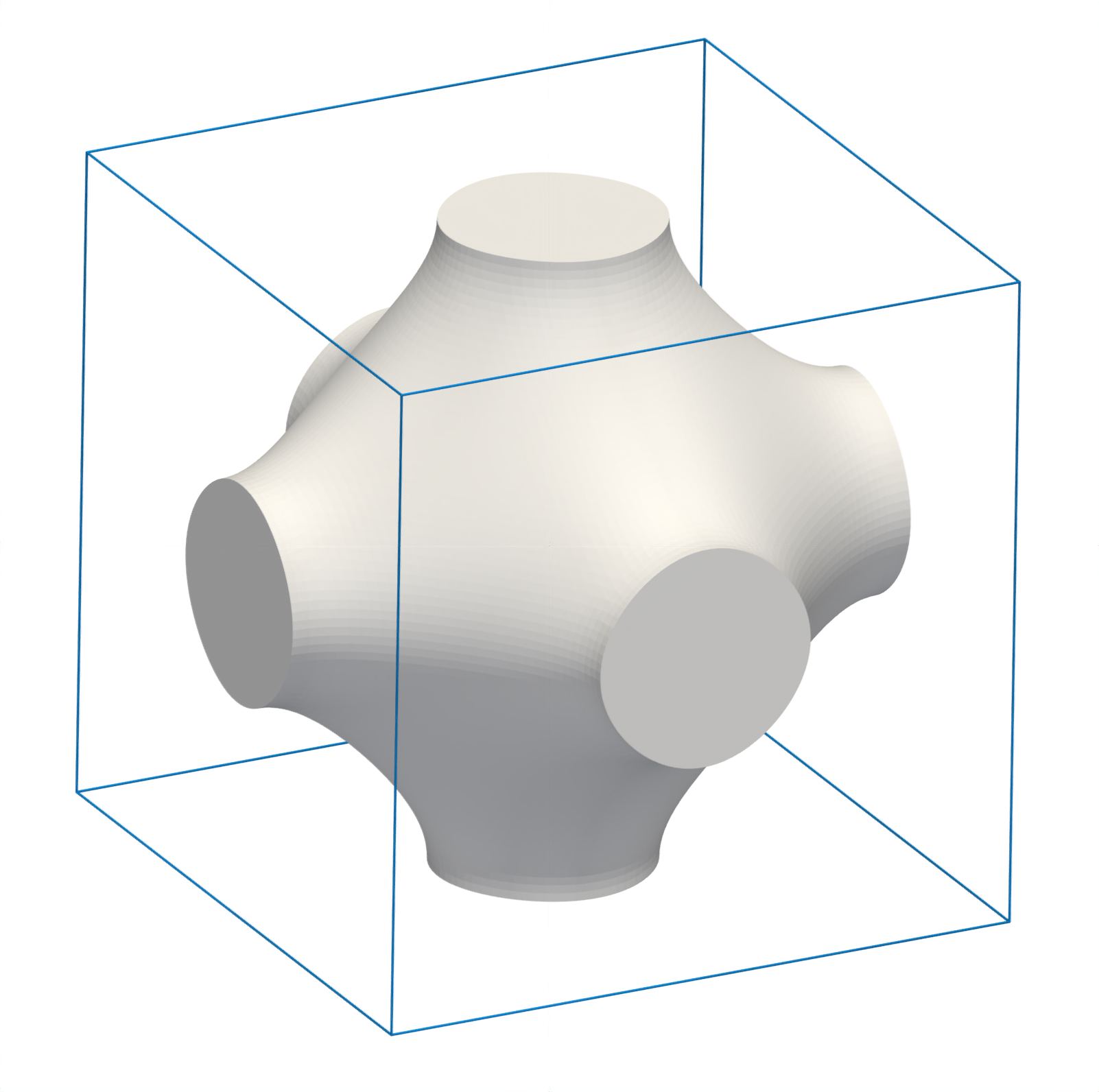}
        \caption{}
        \label{fig:Fig 3a}
    \end{subfigure}
    \end{minipage}
    \Huge{$\mathbf{\rightarrow}$}
    \begin{minipage}{.72\linewidth}
    \centering
    \begin{subfigure}{0.33\linewidth}
        \includegraphics[width=\textwidth]{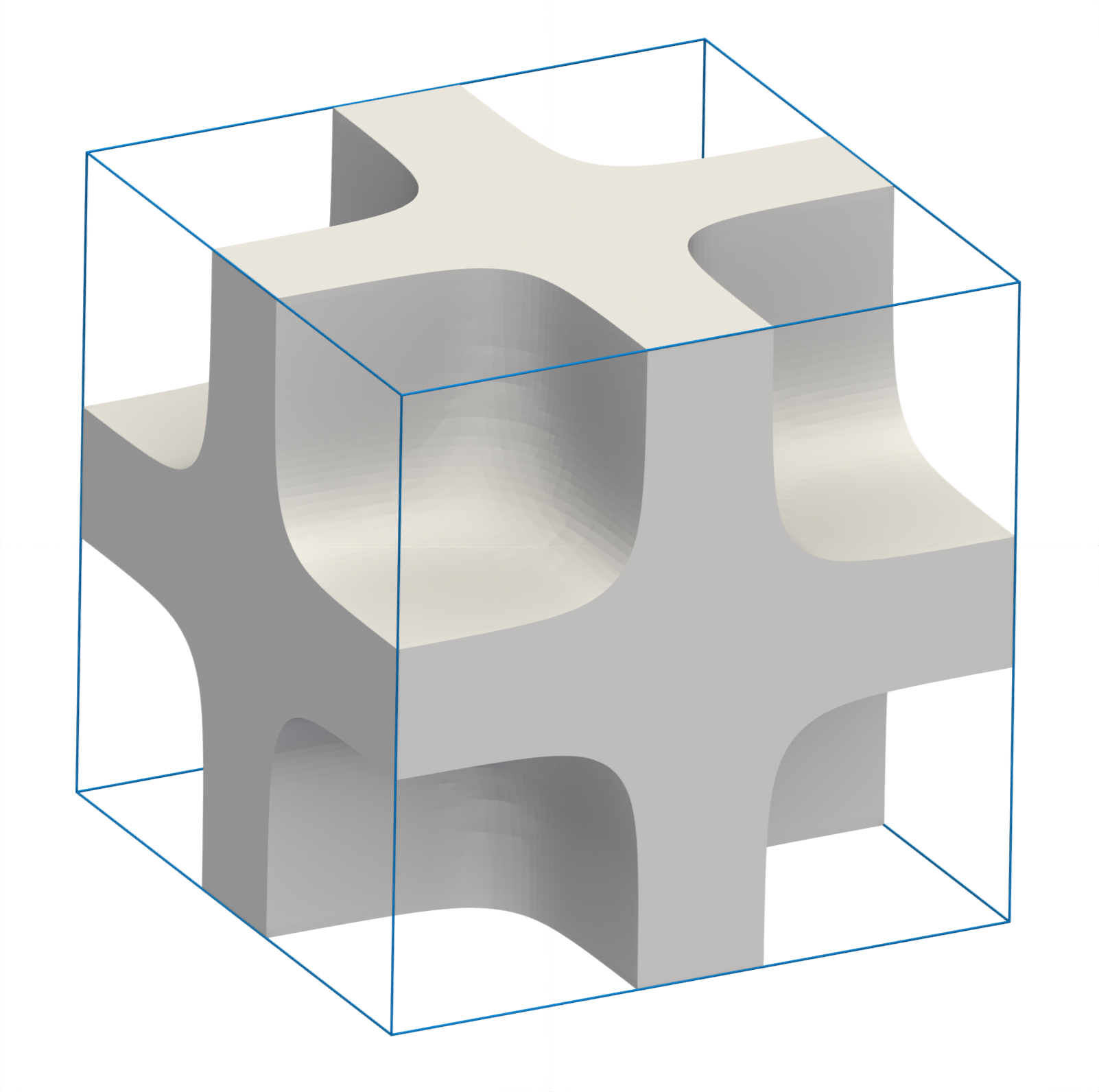}
        \caption{}
        \label{fig:Fig 3b}
    \end{subfigure}\hspace{2em}
    \begin{subfigure}{0.33\linewidth}
        \includegraphics[width=\textwidth]{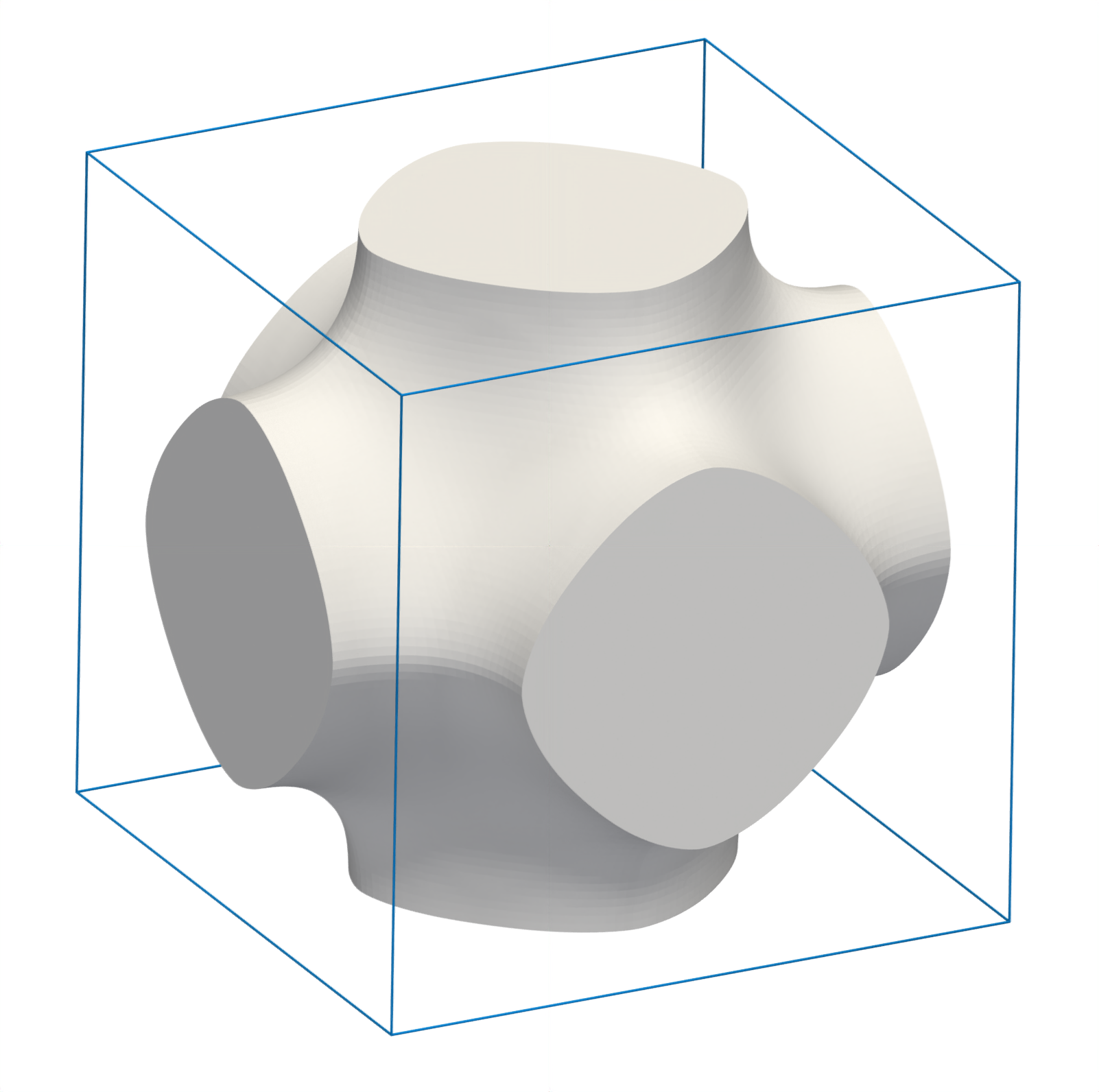}
        \caption{}
        \label{fig:Fig 3c}
    \end{subfigure}
    \begin{subfigure}{0.48\linewidth}
        \includegraphics[width=\textwidth]{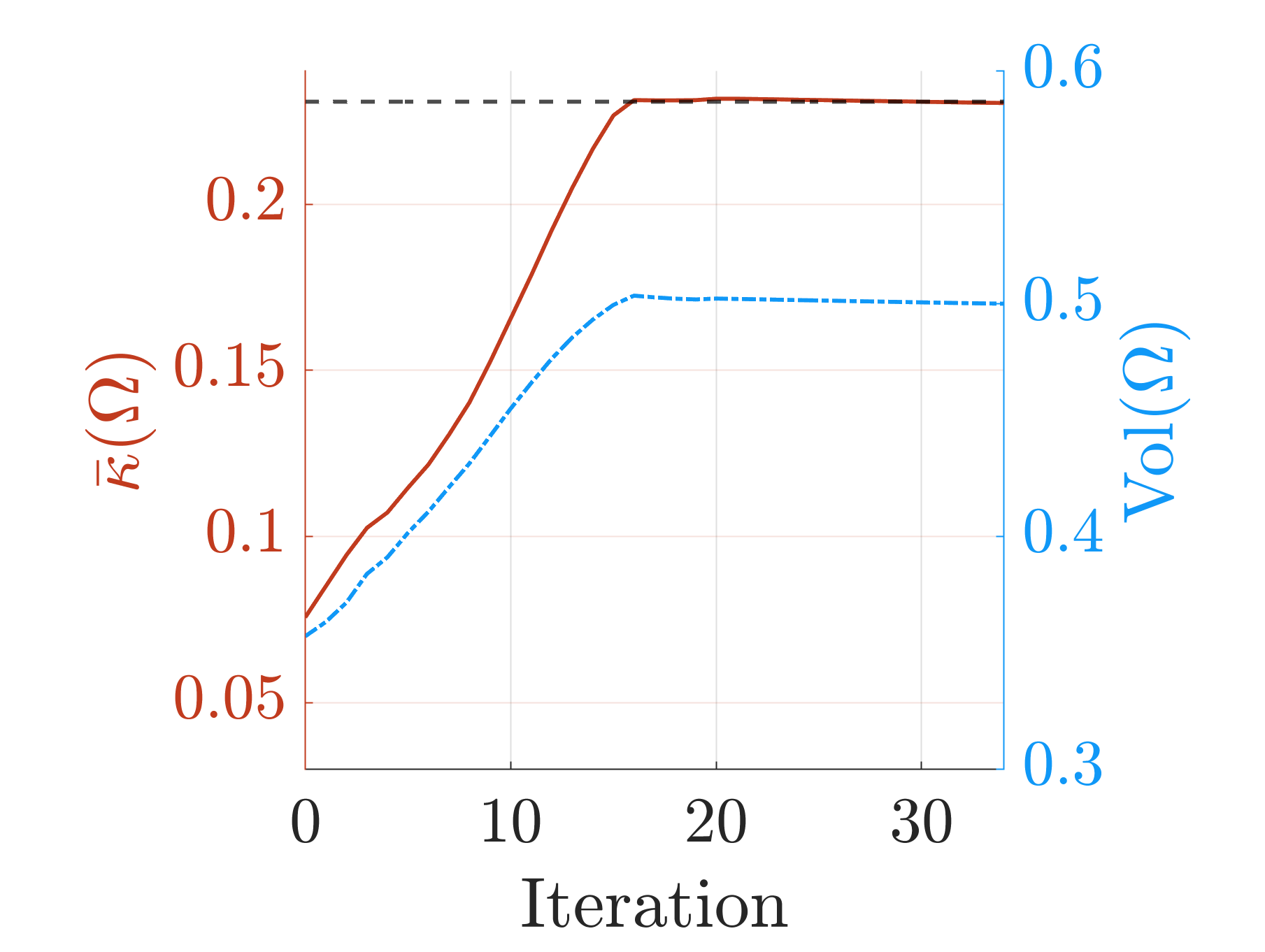}
        \caption{}
        \label{fig:Fig 3d}
    \end{subfigure}
    \begin{subfigure}{0.48\linewidth}
        \includegraphics[width=\textwidth]{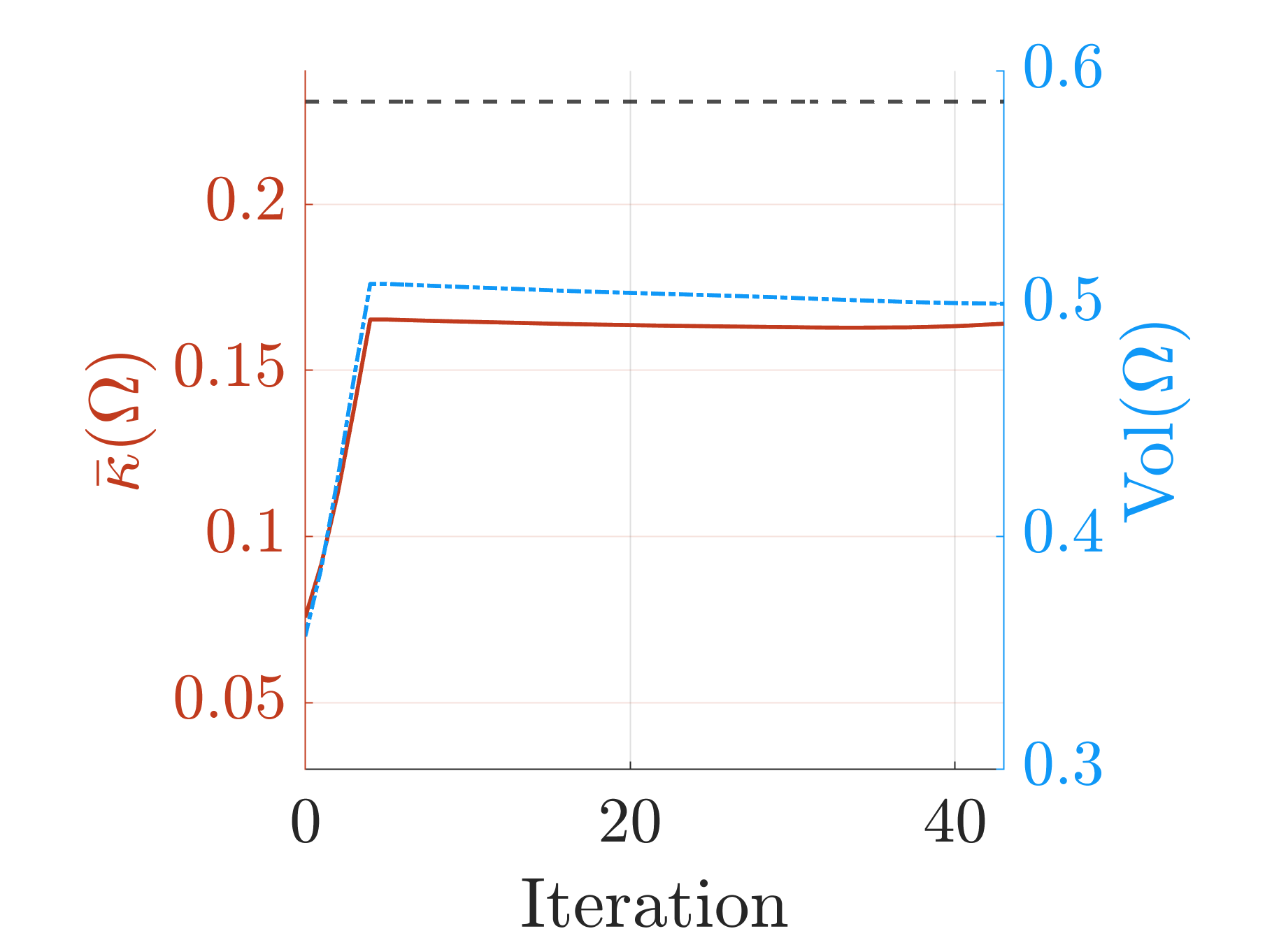}
        \caption{}
        \label{fig:Fig 3e}
    \end{subfigure}
    \end{minipage}
    \caption{Three-dimensional optimisation results for Example 1: maximum bulk modulus. For the starting structure (a), (b) and (c) show the final structures for the Hilbertian projection method and SLP respectively, while (d) and (e) show the respective iteration histories. The Hashin-Shtrikman upper bound for the bulk modulus is given by the dashed black line.}
    \label{fig:Fig 3}
\end{figure*}

Figures \ref{fig:Fig 2} and \ref{fig:Fig 3} show the starting structures and optimisation results for two and three dimensions, respectively for both the Hilbertian projection method and SLP. In addition, we compare the objective value with the Hashin-Shtrikman (HS) upper bound \citep{HASHIN1963127}. Table \ref{tab:summary eg. 1} shows a summary of the results.
\begin{table}[h]
    \centering
    \caption{Summary of optimisation results for Example 1: maximum bulk modulus.}
    \label{tab:summary eg. 1}
    \setlength{\tabcolsep}{0.5em} 
    \begin{tabular}{c|c|c|c|c|c|c}
       Method & $d$ & Fig. & $\bar{\kappa}$ & HS bound & $\operatorname{Vol}$ & Iters. \\\hline
        Proj. & 2 &\ref{fig:Fig 2b}/\ref{fig:Fig 2d} & 0.1854  &  0.1860 &   0.5001		&	32\\
        SLP & 2 &\ref{fig:Fig 2c}/\ref{fig:Fig 2e} & 0.1850  &  0.1860  & 0.5001		&	12\\
        Proj. & 3 &\ref{fig:Fig 3b}/\ref{fig:Fig 3d} & 0.2304  &  0.2308  &  0.5000		&	34\\
        SLP & 3 &\ref{fig:Fig 3c}/\ref{fig:Fig 3e} & 0.1640  &  0.2308  & 0.5000		&	43
    \end{tabular}
\end{table}

In two dimensions both the Hilbertian projection method (Fig. \ref{fig:Fig 2b}/\ref{fig:Fig 2d}) and Hilbertian SLP (Fig. \ref{fig:Fig 2c}/\ref{fig:Fig 2e}) perform well, converging in 32 and 12 iterations respectively at 99.71\% and 99.51\% of the HS bound. In addition, the resulting structures are geometrically similar and match classical results in the literature \cite[Sec. 2.10.3,][]{TopOptMonograph}. It is worth noting that no topological changes occur in this example.

In three dimensions the Hilbertian projection method converges to 99.86\% of the HS bound while Hilbertian SLP converges to 71.07\% of the bound. Over the course of the iteration history, topological changes occur for the Hilbertian projection method, while SLP fails to evolve from the initial topology. The resulting structure from the Hilbertian projection method matches other classical results \cite[Sec. 2.10.3,][]{TopOptMonograph}.

\subsection{Example 2: Maximum bulk modulus with isotropy}\label{subsec: eg problems - Bmod iso}
In Example 2 we consider the same problem setup as Example 1 with the addition of macroscopic isotropy constraints. This ensures that the resulting homogenised stiffness tensor is invariant under rotation. 

In two dimensions the optimisation problem is given by
\begin{equation}
    \begin{aligned}
            \underset{\Omega\in\mathcal{U}_{\mathrm{ad}}}{\min}~
              &-\bar{\kappa}(\Omega)\\
            \text{s.t.}~~ & \mathrm{Vol}(\Omega)=1/2,\\
            & C_i(\Omega)=0,~i=1,\dots,6,\\
            &a(\boldsymbol{u},\boldsymbol{v})=l(\boldsymbol{v}),~\forall \boldsymbol{v}\in V.
          \end{aligned}
\end{equation}
The constraints $C_i(\Omega)$ are given by
    \begin{align}
        &\sqrt{4\bar{\kappa}^2+8\bar{\mu}^2}C_1=\bar{C}_{1111}-\bar{\kappa}-\bar{\mu},\\
        &\sqrt{4\bar{\kappa}^2+8\bar{\mu}^2}C_2=\bar{C}_{2222}-\bar{\kappa}-\bar{\mu},\\
        &\sqrt{4\bar{\kappa}^2+8\bar{\mu}^2}C_3=\sqrt{2}(\bar{C}_{1122}-\bar{\kappa}+\bar{\mu}),\\
        &\sqrt{4\bar{\kappa}^2+8\bar{\mu}^2}C_4=2\bar{C}_{1112},\\
        &\sqrt{4\bar{\kappa}^2+8\bar{\mu}^2}C_5=2\bar{C}_{2212},\\
        &\sqrt{4\bar{\kappa}^2+8\bar{\mu}^2}C_6=2(\bar{C}_{1212}-\bar{\mu}),
    \end{align}
where $\bar{\mu}$ is the isotropic shear modulus given by
\begin{equation}
    \bar{\mu} = \frac{1}{8}(\bar{C}_{1111} + \bar{C}_{2222}) - \frac{1}{4}\bar{C}_{1122} + \frac{1}{2}\bar{C}_{1212}.
\end{equation}
Analogous expressions appear in three dimensions and the number of isotropy constraints increases from 6 to 21 \citep{ChallisThesis}. It should be noted that the term $\sqrt{4\bar{\kappa}^2+8\bar{\mu}^2}$ appears as a normalisation coefficient and is considered constant for the purpose of the shape differentiation. Furthermore, owing to the symmetry of the isotropy constraints the extended constraint shape sensitivities have a nullity of two. Our method handles these with no special treatment.

To visualise the behaviour of the isotropy constraints in the iteration history, we define the effective anisotropy $\bar{\mathcal{A}}$ to be the sum of squares of the violation of these constraints. That is,
\begin{equation}
    \bar{\mathcal{A}}(\Omega)=\sqrt{\sum C_i(\Omega)^2}.
\end{equation}

For the two-dimensional problem we implement both the full set of isotropy constraints as well as the single constraint $\bar{A}=0$. 
We find that the SLP method struggles with the full set of constraints. For this reason we instead use $\bar{\mathcal{A}}=0$ as the isotropy constraint for the SLP method. For the Hilbertian projection method we find that the full set of constraints is more effective. This matches previous literature \cite[e.g.,][]{10.1016/j.ijsolstr.2008.02.025_2008,ChallisThesis}. 

\begin{figure*}[p]
    \centering
    \begin{subfigure}{0.25\linewidth}
        \includegraphics[width=\textwidth]{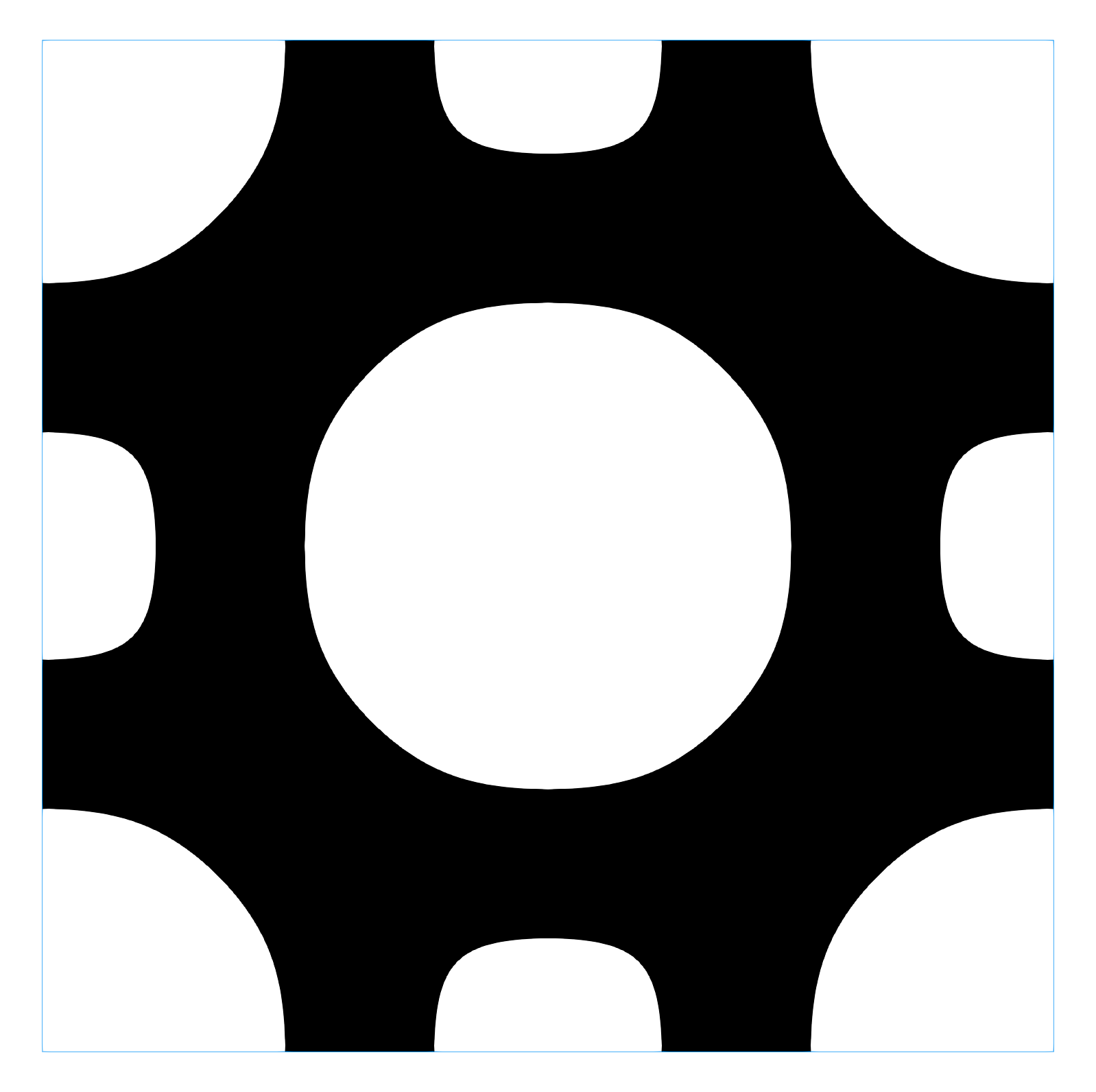}
        \caption{}
        \label{fig:Fig 4a}
    \end{subfigure}\hspace{3.3em}
    \begin{subfigure}{0.25\linewidth}
        \includegraphics[width=\textwidth]{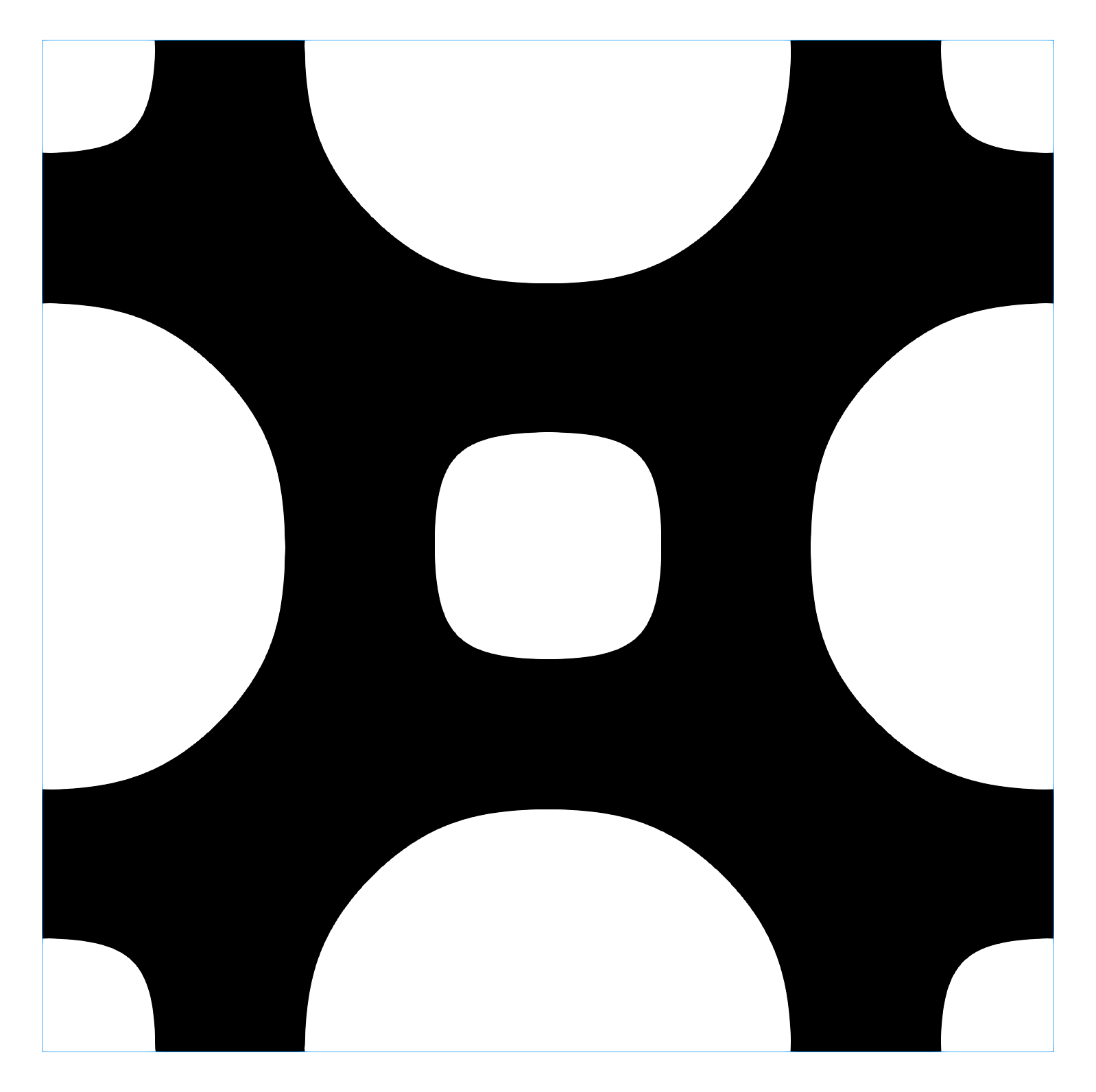}
        \caption{}
        \label{fig:Fig 4b}
    \end{subfigure}\hspace{3.3em}
    \begin{subfigure}{0.25\linewidth}
        \includegraphics[width=\textwidth]{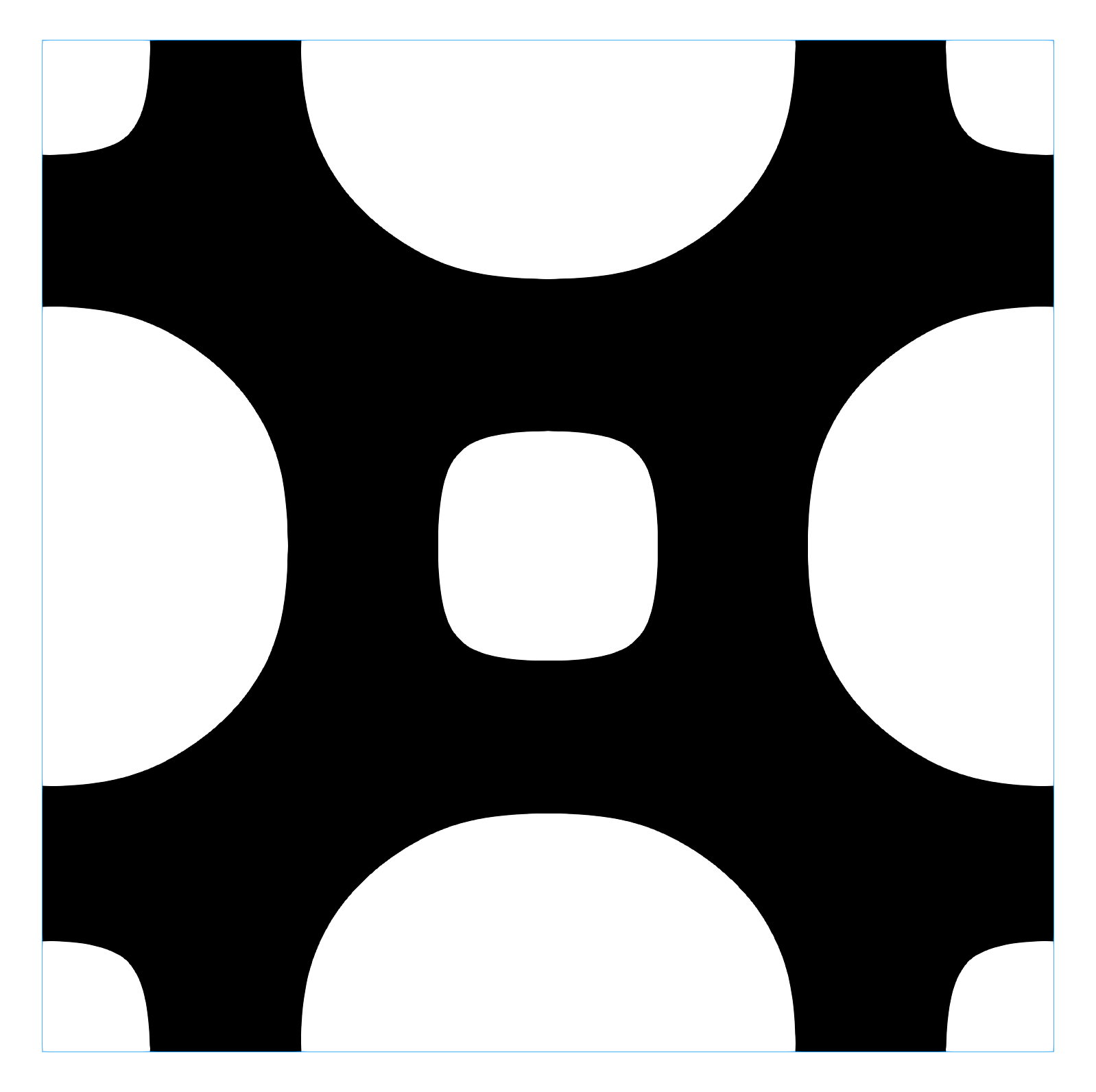}
        \caption{}
        \label{fig:Fig 4c}
    \end{subfigure}
    \begin{subfigure}{0.32\linewidth}
        \includegraphics[width=\textwidth]{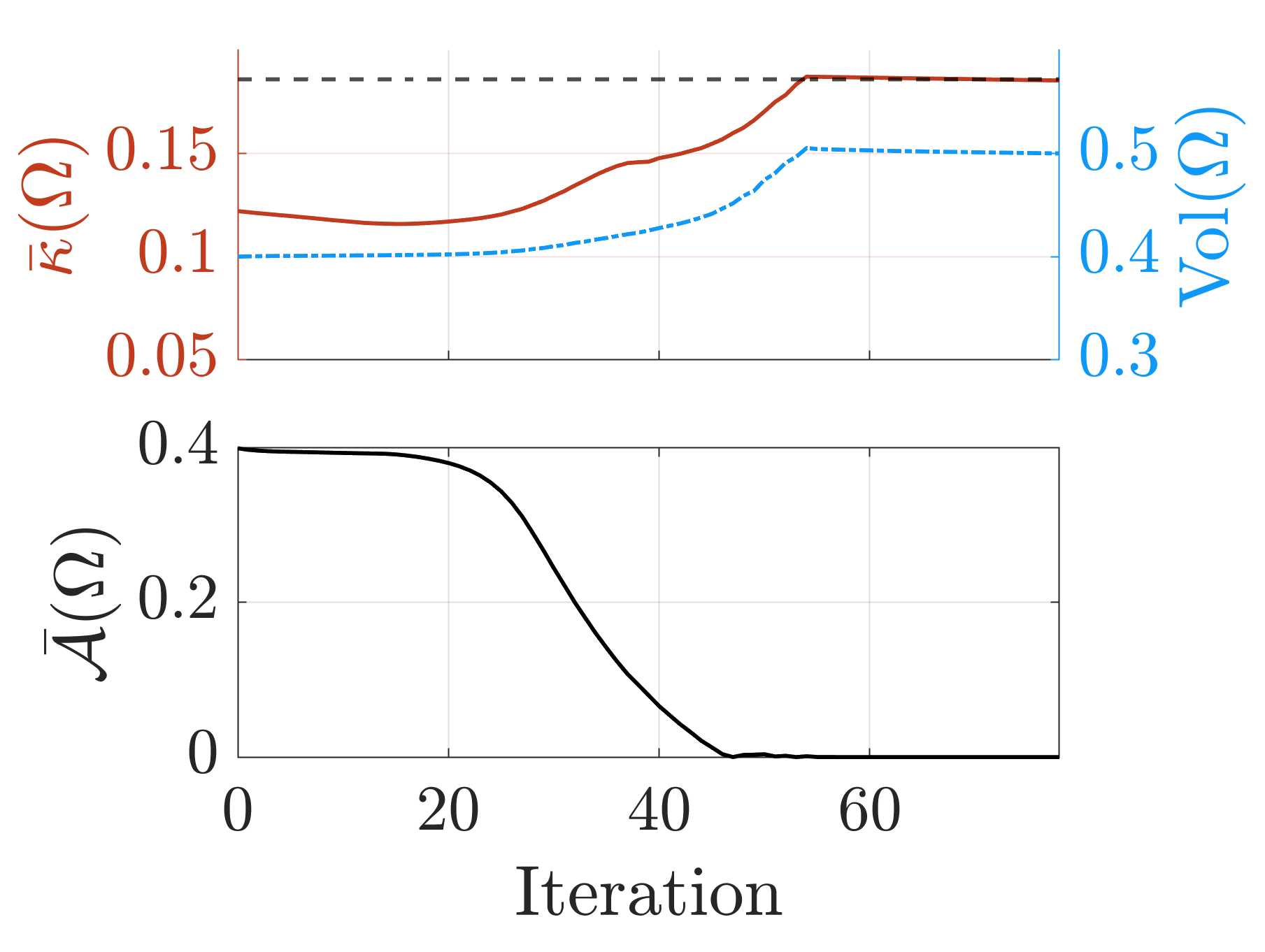}
        \caption{}
        \label{fig:Fig 4d}
    \end{subfigure}
    \begin{subfigure}{0.32\linewidth}
        \includegraphics[width=\textwidth]{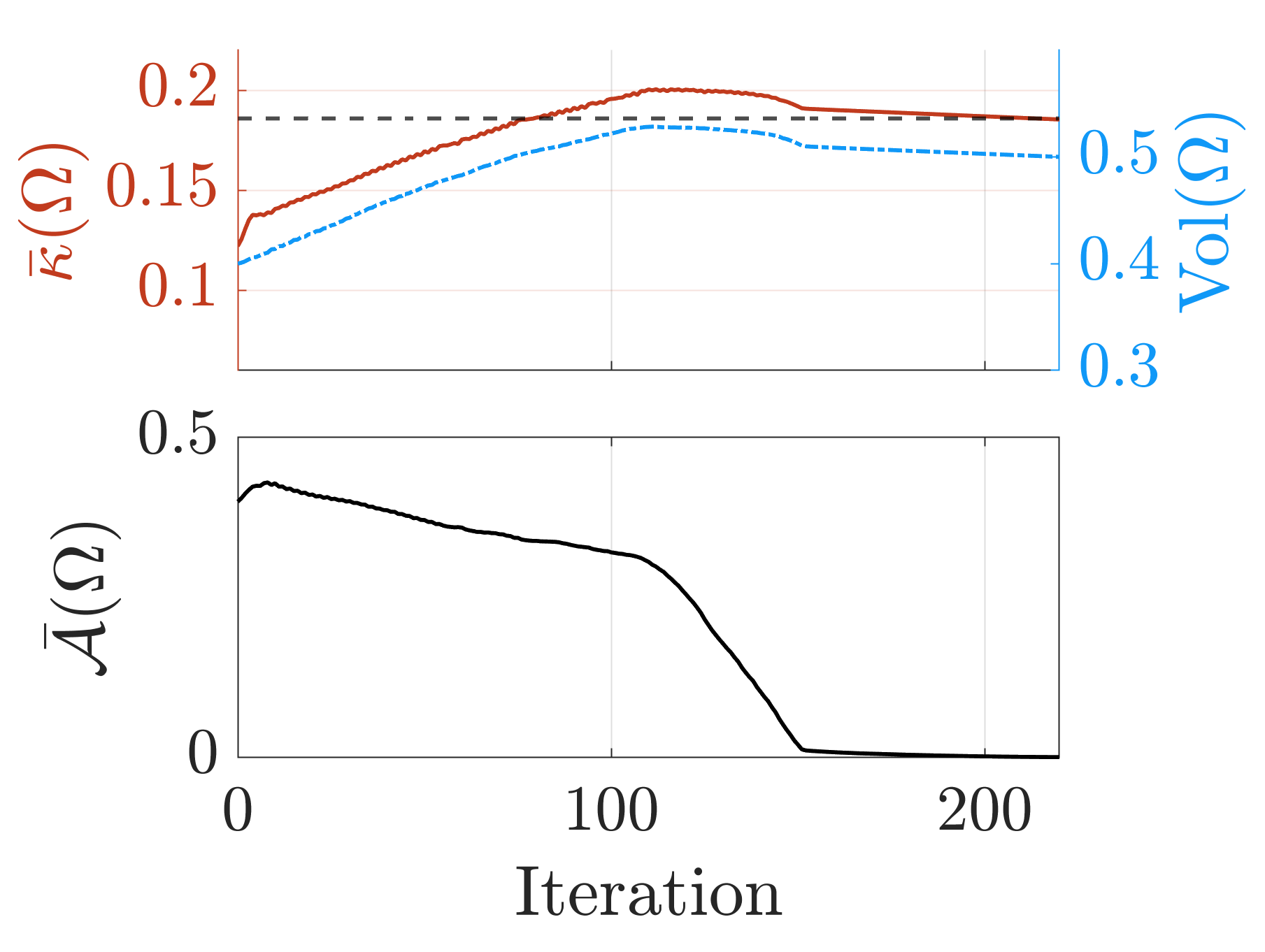}
        \caption{}
        \label{fig:Fig 4e}
    \end{subfigure}
    \begin{subfigure}{0.32\linewidth}
        \includegraphics[width=\textwidth]{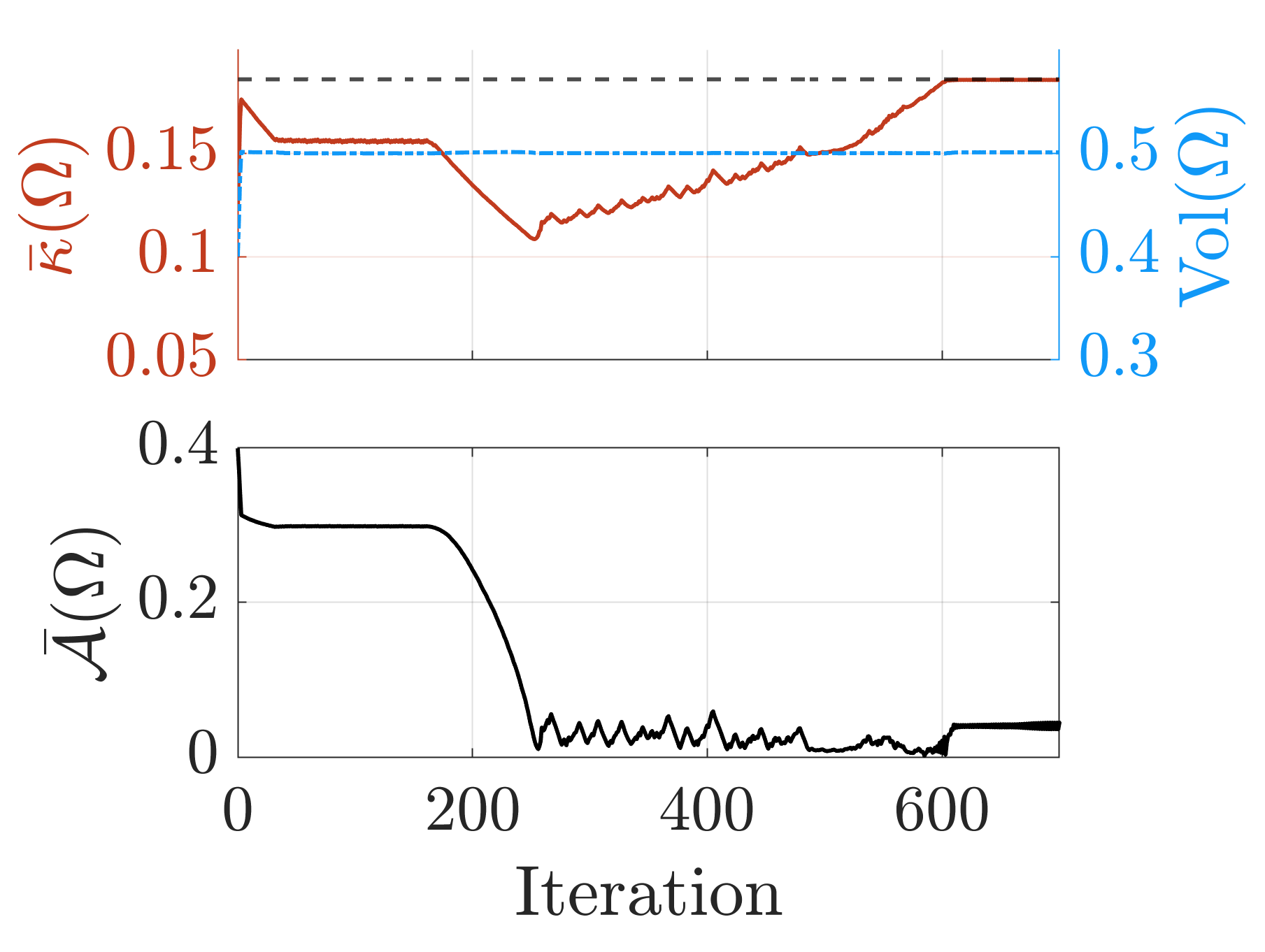}
        \caption{}
        \label{fig:Fig 4f}
    \end{subfigure}
    \caption{Two-dimensional optimisation results for Example 2: maximum bulk modulus with isotropy. (a), (b), and (c) show the final structures for the Hilbertian projection method with the full set of constraints, single anisotropy constraint, and SLP with the single anisotropy constraint respectively, while (d), (e), and (f) show the respective iteration histories. The Hashin-Shtrikman upper bound for the bulk modulus is given by the dashed black line. Due to very little change between iteration 600 and 1000, the upper bound on the $x$-axis in (f) has been reduced to 700.}
    \label{fig:Fig 4}

    \centering
    \begin{minipage}{.72\linewidth}
    \centering
    \begin{subfigure}{0.33\linewidth}
        \includegraphics[width=\textwidth]{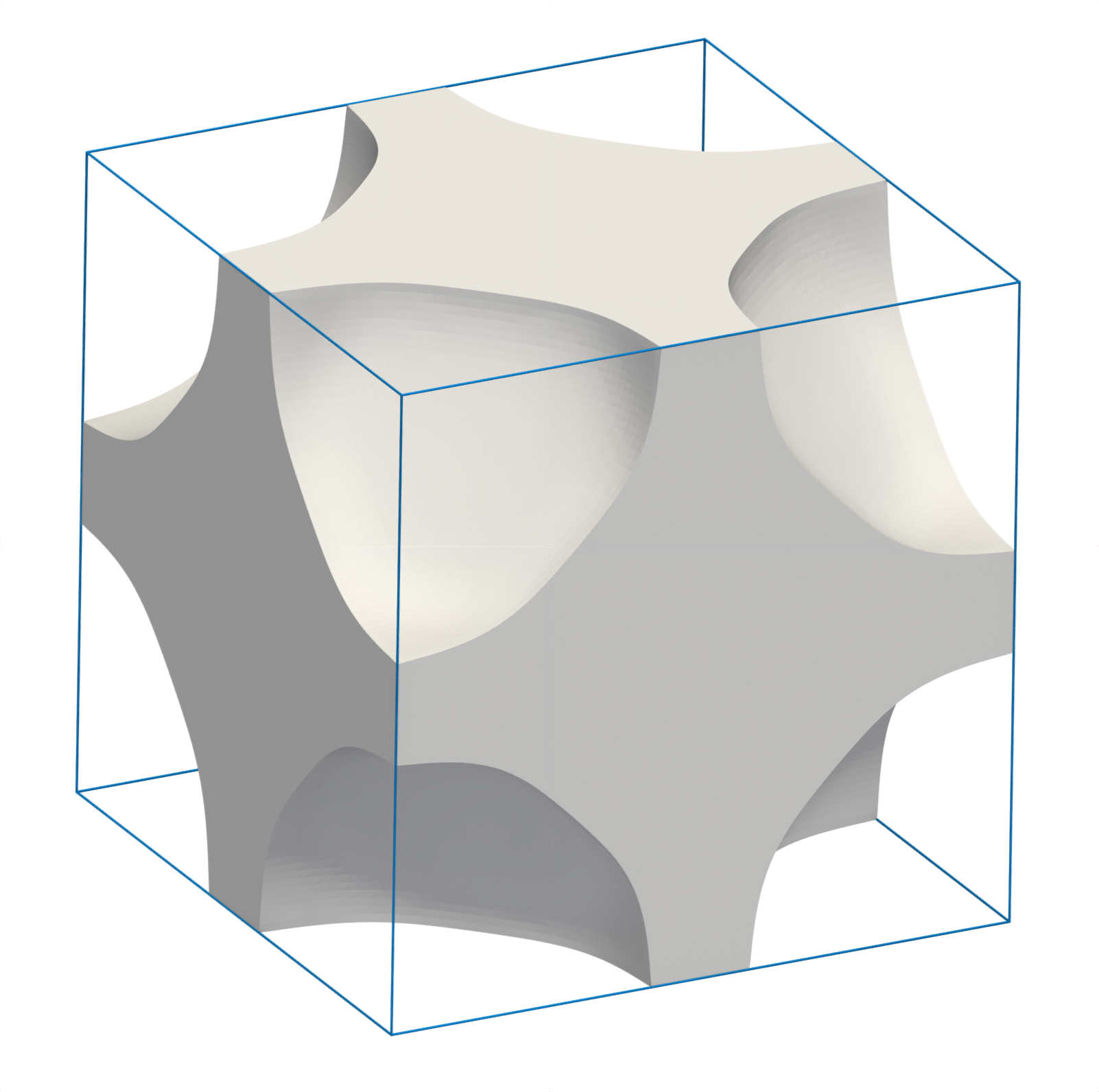}
        \caption{}
        \label{fig:Fig 5a}
    \end{subfigure}\hspace{2em}
    \begin{subfigure}{0.33\linewidth}
        \includegraphics[width=\textwidth]{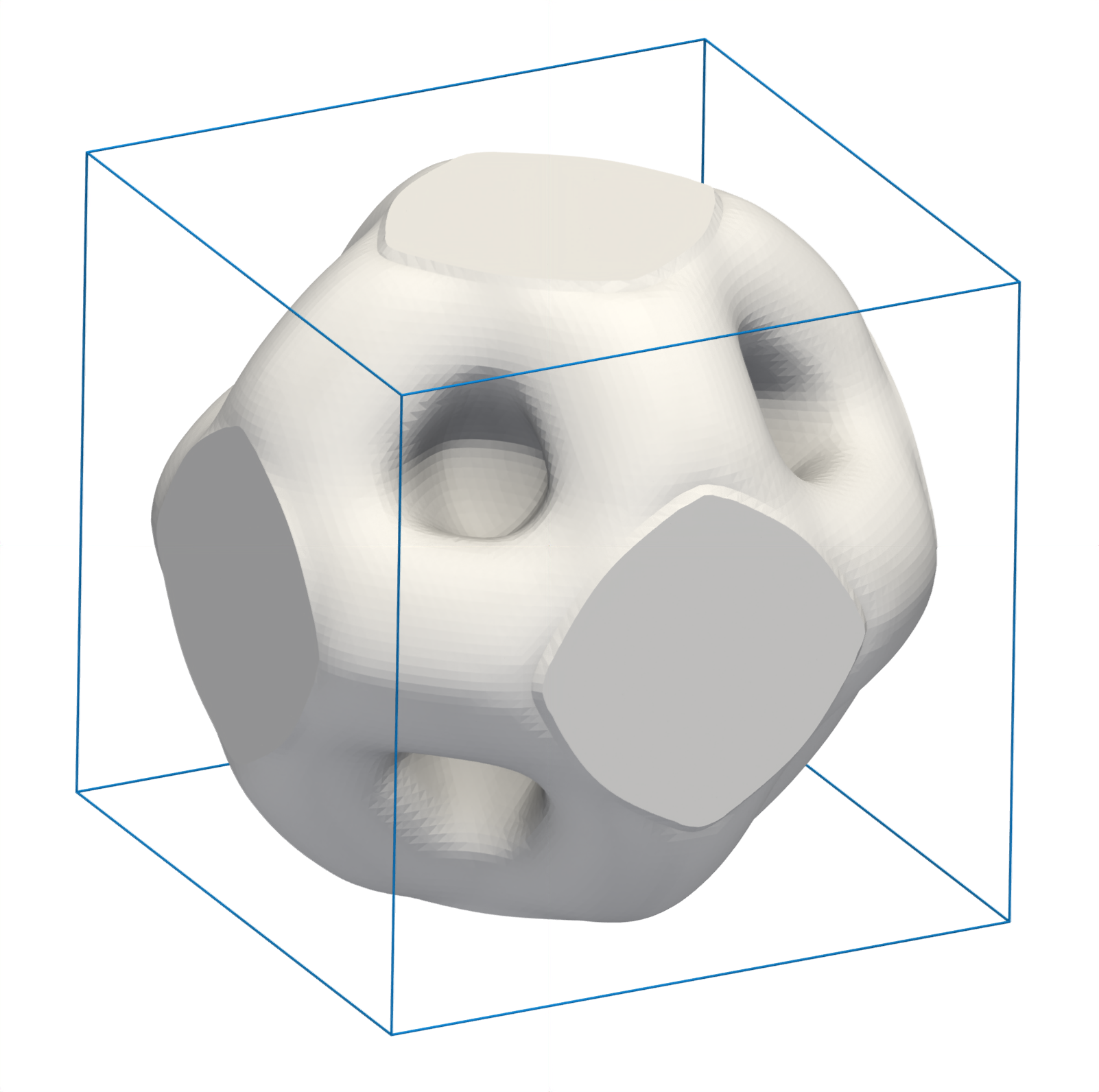}
        \caption{}
        \label{fig:Fig 5b}
    \end{subfigure}
    \begin{subfigure}{0.48\linewidth}
        \includegraphics[width=\textwidth]{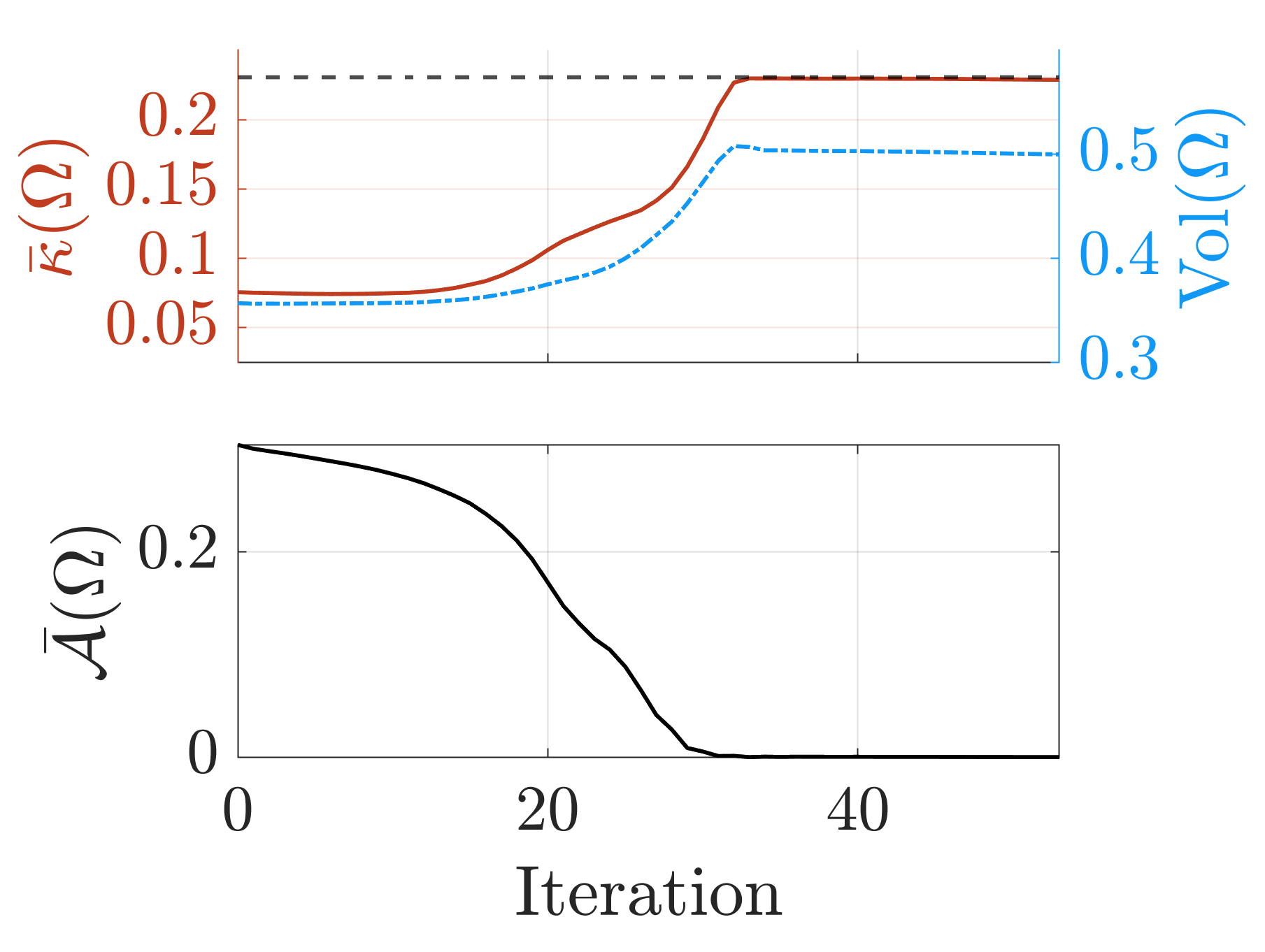}
        \caption{}
        \label{fig:Fig 5c}
    \end{subfigure}
    \begin{subfigure}{0.48\linewidth}
        \includegraphics[width=\textwidth]{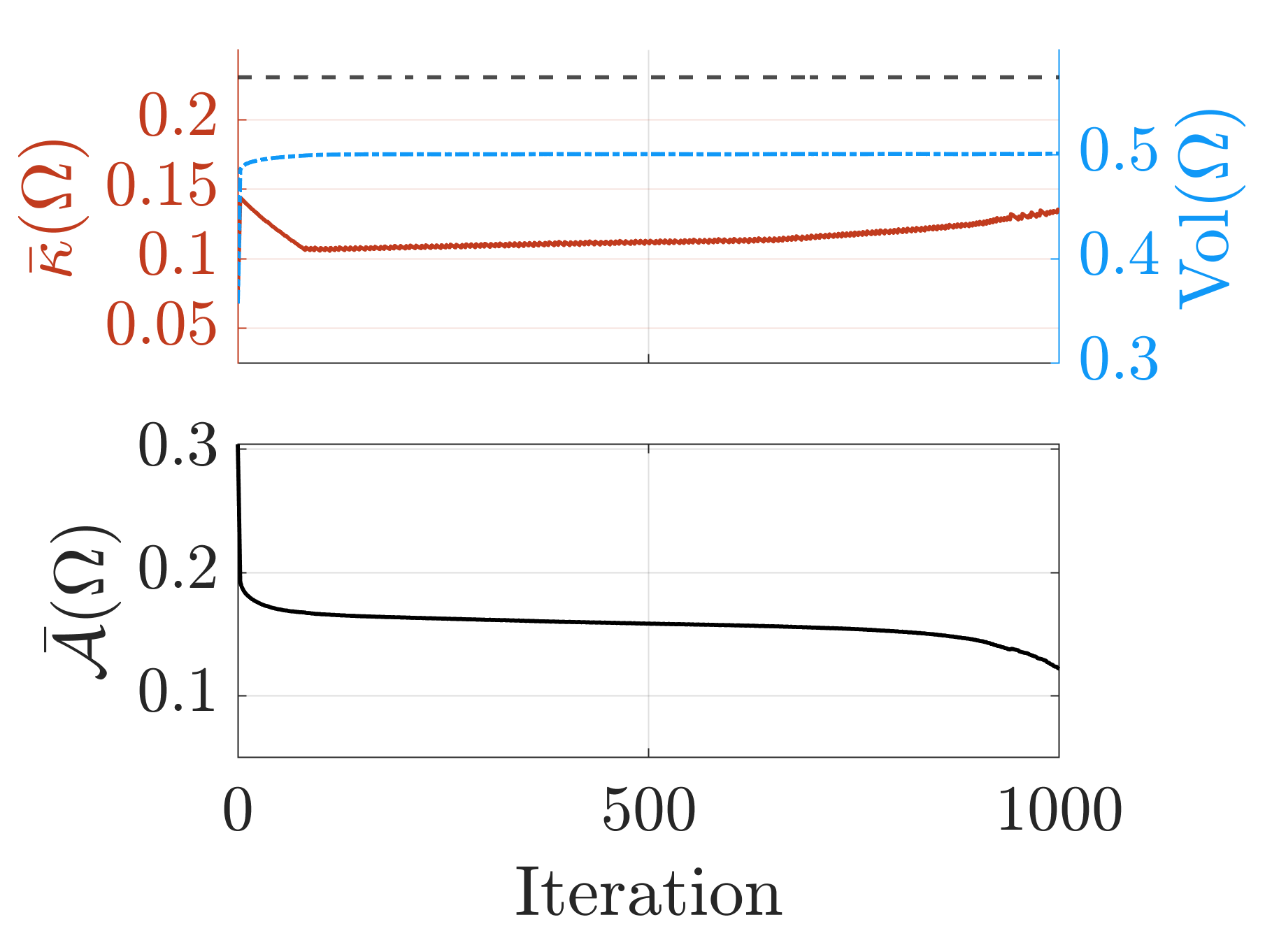}
        \caption{}
        \label{fig:Fig 5d}
    \end{subfigure}
    \end{minipage}
    \caption{Three-dimensional optimisation results for Example 2: maximum bulk modulus with isotropy. (a) and (b) show the final structures for the Hilbertian projection method and SLP respectively, while (c) and (d) show the respective iteration histories. The Hashin-Shtrikman upper bound for the bulk modulus is given by the dashed black line.}
    \label{fig:Fig 5}
\end{figure*}

Figure \ref{fig:Fig 4} shows the two-dimensional results for the Hilbertian projection method with the full set of constraints and single anisotropy measure constraint, and SLP with the anisotropy measure constraint. We omit the starting structure as this is the same as previously (Fig.~\ref{fig:Fig 2a}). Figure \ref{fig:Fig 5} shows the three-dimensional results for the Hilbertian projection method and SLP.  Table \ref{tab:summary eg. 2} shows a summary of the results for this example.
\begin{table*}[t]
    \centering
    \caption{Summary of optimisation results for Example 2: maximum bulk modulus with isotropy. The asterisk denotes failure to converge.}
    \label{tab:summary eg. 2}
    \setlength{\tabcolsep}{0.5em} 
    \begin{tabular}{c|c|c|c|c|c|c|c|c}
     Method & $d$ & Iso. Const. Type & Fig. & $\bar{\kappa}$ & HS bound & $\bar{\mathcal{A}}$ & $\operatorname{Vol}$ & Iters. \\\hline
        Proj. & 2 & Full set & \ref{fig:Fig 4a}/\ref{fig:Fig 4d} & 0.1854  &  0.1860 & 0.0001 &   0.5000 &  78\\
        Proj. & 2 & Aniso. Meas. & \ref{fig:Fig 4b}/\ref{fig:Fig 4e} & 0.1854  &  0.1860 & 0.0001 &   0.5000  &  220\\
        SLP & 2 & Aniso. Meas. & \ref{fig:Fig 4c}/\ref{fig:Fig 4f} & 0.1856  &  0.1860 & 0.0363 &  0.5010 &  1000*\\
        Proj. & 3 & Full set & \ref{fig:Fig 5a}/\ref{fig:Fig 5c} & 0.2287  &  0.2308 & 0.0001 &   0.5000 &  53\\
        SLP & 3 & Aniso. Meas. & \ref{fig:Fig 5b}/\ref{fig:Fig 5d} & 0.1295 &   0.2308 &  0.1348  &  0.5007  &  1000*
    \end{tabular}
\end{table*}

In two dimensions the Hilbertian projection method performs well for the full set of constraints and the single anisotropy measure constraint. The number of iterations is markedly lower when using the full set of constraints (78 vs 220). This is unsurprising because including the individual symmetry constraints enables the projection method to improve each constraint separately \citep{10.1016/j.ijsolstr.2008.02.025_2008,ChallisThesis}. The final optimisation values with the full set or single constraint are very similar, being 99.68\% and 99.69\% of the HS bound, while the final structures are geometrically identical under a periodic shift of half the cell edge length along each coordinate direction. The resulting structures match other classical results \cite[Sec. 2.10.3,][]{TopOptMonograph}. In contrast, the SLP method does not manage to reduce the measure of anisotropy to zero and reaches the maximum number of iterations. However, the final structure obtained with SLP (Fig. \ref{fig:Fig 4c}) is very close to those obtained with the Hilbertian projection method (Fig. \ref{fig:Fig 4a} and \ref{fig:Fig 4b}). We suspect that SLP fails to converge due to the difficulty associated with choosing the trust region constraints.

For the three-dimensional results, we find that the Hilbertian projection method with the full set of constraints converges in 53 iterations to 99.12\% of the HS bound. The final structure again matches known results from the literature \cite[Sec. 2.10.3,][]{TopOptMonograph}. For SLP, the optimisation algorithm fails to converge in 1000 iterations and the final structure is clearly not optimal. 

These results demonstrate that the Hilbertian projection method is able to handle a large number of constraints (22 in 3-dimensions) without requiring any parameter tuning of $\alpha_{\mathrm{min}}$ or $\lambda$. In addition, the optimisation histories for the Hilbertian projection method with the full set of isotropy constraints (Fig. \ref{fig:Fig 4d} and \ref{fig:Fig 5c}) are smooth and converge fairly quickly.

\subsection{Example 3: Auxetic materials}\label{subsec: eg problems - auxetic}
In this example we consider two-dimensional minimum volume auxetic materials with a Poisson's ratio of $-0.5$. For the problem set up, we consider a bounding domain $D=[0,1]^2$ that contains a solid phase and void phase. As previously, the solid phase is constructed from an isotropic material with $E=1$ and $\nu=0.3$.  

To obtain an effective Poisson's ratio $\bar{\nu}=-0.5$, we require that $\bar{C}_{1111}=\bar{C}_{2222}$ and prescribe a value to $\bar{C}_{1111}$ and $\bar{C}_{1122}$ so that the effective Poisson's ratio $\bar{\nu}$ given by
\begin{equation}
    \bar{\nu}=\frac{\bar{C}_{1122}}{\bar{C}_{1111}}
\end{equation}
gives the required $\bar{\nu}=-0.5$.

It may be noted that this is similar to the approach taken by \cite{10.1016/j.cad.2016.09.009_2017}. However, they instead minimise a weighted sum of $\bar{C}_{ijkl}$ and prescribed value subject to a volume constraint.

For the purpose of this example we choose $\bar{C}_{1111}=0.1$, which results in $\bar{C}_{1122}=-0.05$. The resulting optimisation problem is then
\begin{equation}
    \begin{aligned}
            \underset{\Omega\in\mathcal{U}_{\mathrm{ad}}}{\min}~
              &\mathrm{Vol}(\Omega)\\
            \text{s.t.}~~ & \bar{C}_{1111}(\Omega)=0.1,\\
            &\bar{C}_{2222}(\Omega)=0.1,\\
            &\bar{C}_{1122}(\Omega)=-0.05,\\
            &\bar{C}_{1112}(\Omega)=0,\\
            &\bar{C}_{2212}(\Omega)=0,\\
            &a(\boldsymbol{u},\boldsymbol{v})=l(\boldsymbol{v}),~\forall \boldsymbol{v}\in V.
          \end{aligned}
\end{equation}
For this example we use $\gamma_{\mathrm{max}}=0.05$ and $\alpha_{\mathrm{min}}^2=0.5$. This means that the optimiser favors improvement of the constraints rather than the objective and does not move too quickly to avoid disconnecting in the first few iterations.

\begin{figure*}[t]
    \centering
    \begin{subfigure}{0.25\linewidth}
        \includegraphics[width=\textwidth]{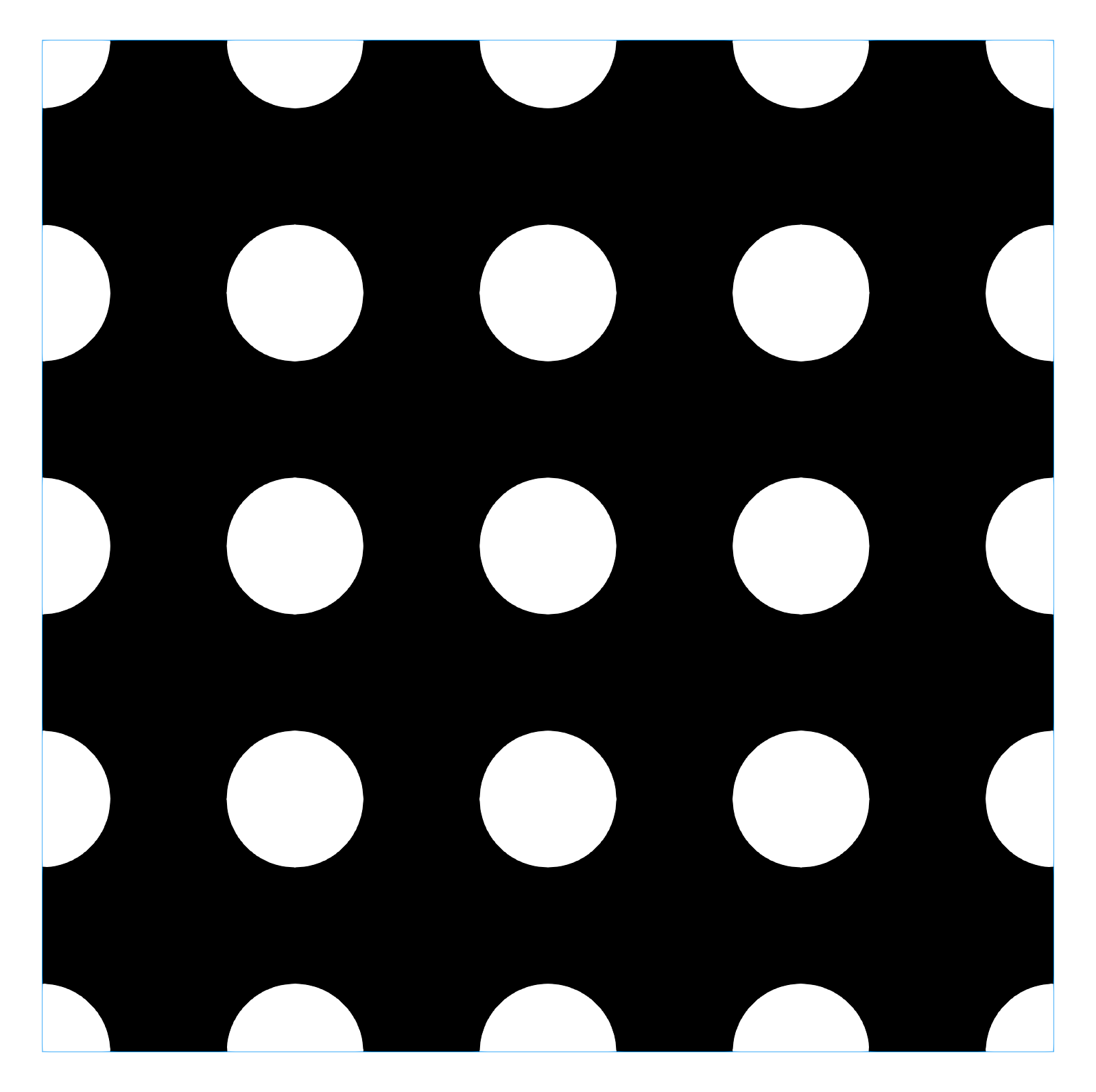}
        \caption{}
        \label{fig:Fig 6a}
    \end{subfigure}
    \begin{subfigure}{0.05\linewidth}
        \centering
        \Huge{$\mathbf{\rightarrow}$}
        \vspace{1.5cm}
    \end{subfigure}
    \begin{subfigure}{0.25\linewidth}
        \includegraphics[width=\textwidth]{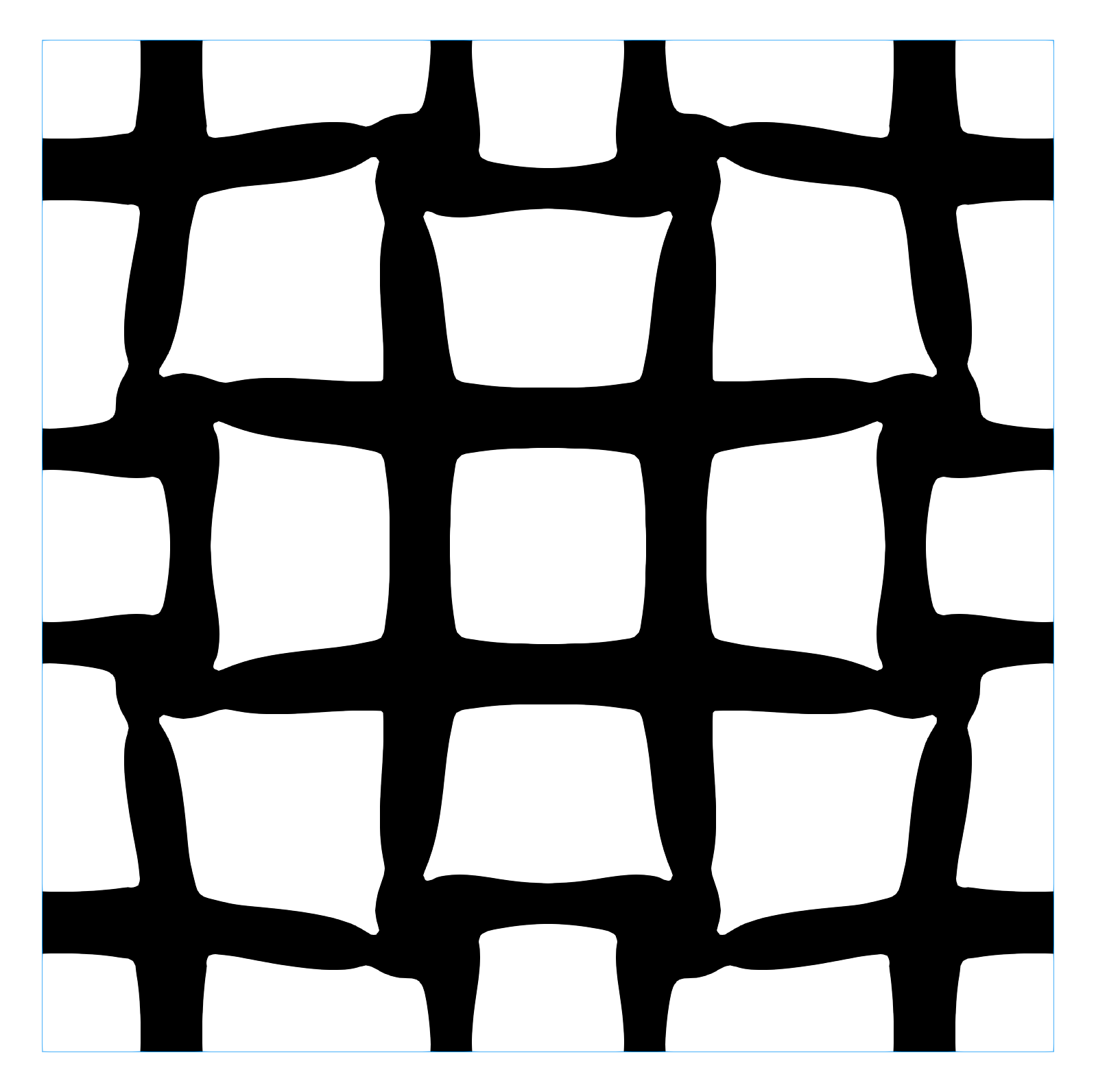}
        \caption{}
        \label{fig:Fig 6b}
    \end{subfigure}
    \begin{subfigure}{0.35\linewidth}
        \includegraphics[width=\textwidth]{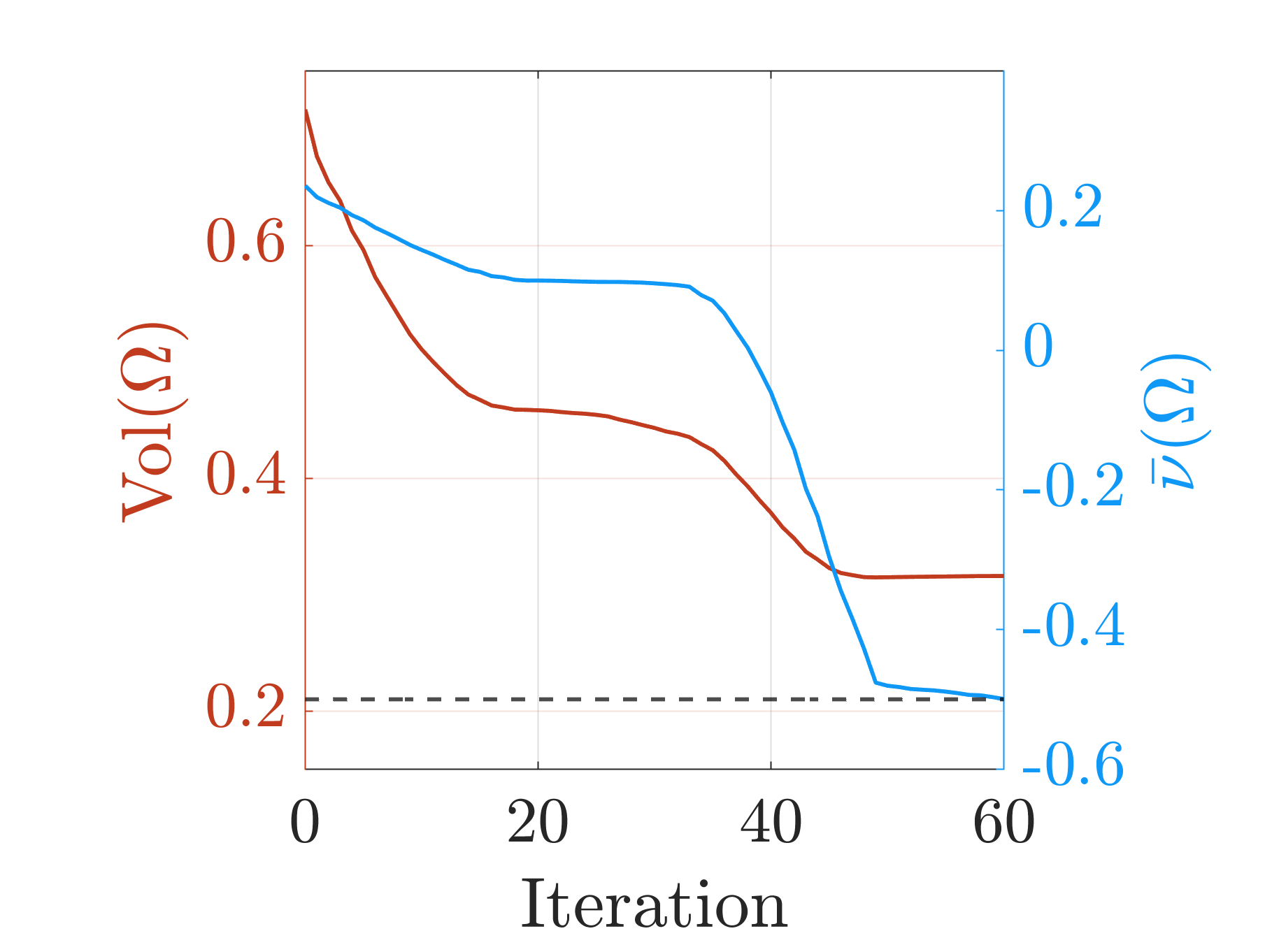}
        \caption{}
        \label{fig:Fig 6c}
    \end{subfigure}
    \caption{Optimisation results for Example 3: auxetic materials. For the starting structure (a), (b) shows the final structure while (c) shows the iteration history for the Hilbertian projection method. The desired Poisson's ratio of -0.5 is given by the dashed black line.}
    \label{fig:Fig 6}
\end{figure*}

Figure \ref{fig:Fig 6} shows the starting structure (Fig. \ref{fig:Fig 6a}) and optimisation results (Fig. \ref{fig:Fig 6b} and \ref{fig:Fig 6c}) for this problem using the Hilbertian projection method. We use a periodic starting structure with sixteen equally spaced holes. The method converges in 61 iterations with a final volume of 0.3159 and Poisson's ratio of -0.4998. In contrast, SLP fails because the optimiser prioritises the objective leading to a disconnected solid phase and the algorithm is unable to recover.

\subsection{Example 4: Multi-phase materials}\label{subsec: eg problems - multi-phase results}
For our final example we consider two-dimensional multi-phase maximum bulk modulus problems with and without isotropy constraints. We consider a bounding domain $D=[0,1]^2$ that contains two solid phases and void. The solid phase contained in $\Omega_2$ is an isotropic material with $E=1$ and $\nu=0.3$, while $\Omega_3$ contains isotropic material with $E=1/2$ and $\nu=0.3$. The void phase is contained in $\Omega_1$ and $\Omega_4$. As previously, most void phase is removed from the mesh (see Sec. \ref{sec: numerical impl}) while any material close to the interface is specified as weak material with $E=10^{-3}$ and $\nu=0.3$. 

We consider two optimisation problems. The first is maximum bulk modulus subject to volume constraints on $\Omega_2$ and $\Omega_3$ given by
\begin{equation}\label{eqn: multiphase example no-iso}
    \begin{aligned}
            \underset{\mathcal{D}_1,\mathcal{D}_2\in\mathcal{U}_{\mathrm{ad}}}{\min}~
              &-\bar{\kappa}(\mathcal{D}_1,\mathcal{D}_2)\\
            \text{s.t.}~~ &\mathrm{Vol}_{\Omega_2}(\mathcal{D}_1,\mathcal{D}_2)=1/4,\\
            &\mathrm{Vol}_{\Omega_3}(\mathcal{D}_1,\mathcal{D}_2)=1/4,\\
            &a(\boldsymbol{u},\boldsymbol{v})=l(\boldsymbol{v}),~\forall \boldsymbol{v}\in V.
          \end{aligned}
\end{equation}
The second is maximum bulk modulus subject to macroscopic isotropy constraints and volume constraints on $\Omega_2$ and $\Omega_3$ given by
\begin{equation}\label{eqn: multiphase example iso}
    \begin{aligned}
            \underset{\mathcal{D}_1,\mathcal{D}_2\in\mathcal{U}_{\mathrm{ad}}}{\min}~
              &-\bar{\kappa}(\mathcal{D}_1,\mathcal{D}_2)\\
            \text{s.t.}~~ &\mathrm{Vol}_{\Omega_2}(\mathcal{D}_1,\mathcal{D}_2)=1/4,\\
            &\mathrm{Vol}_{\Omega_3}(\mathcal{D}_1,\mathcal{D}_2)=1/4,\\
            &C_i(\mathcal{D}_1,\mathcal{D}_2)=0,~i=1,\dots,6,\\
            &a(\boldsymbol{u},\boldsymbol{v})=l(\boldsymbol{v}),~\forall \boldsymbol{v}\in V,
          \end{aligned}
\end{equation}
where the constraints $C_{i}$ are as in Example 2. For SLP we use the single anisotropy constraint and for the Hilbertian projection method we use the full set of isotropy constraints. For these examples we use $\gamma_{\mathrm{max}}=0.05$ so that the optimiser does not evolve the designs too quickly. It should be noted that for the case of only volume constraints the extended constraint shape sensitivities have a nullity of one (in $\boldsymbol{\theta}_1$) and zero (in $\boldsymbol{\theta}_2$) owing to the structure of the shape derivatives for the volume constraints. For the case of volume and isotropy constraints the extended constraint shape sensitivities have a nullity of three (in $\boldsymbol{\theta}_1$) and two (in $\boldsymbol{\theta}_2$) due to the underlying symmetry of the shape derivatives for the volume and isotropy constraints. Our method handles these with no special treatment.

We initialise with two overlapping level set functions that give a starting structure completely comprised of the less stiff material ($\Omega_3$) and the void phase. Regions of stiffer material ($\Omega_2$) are readily generated during the optimisation via independent evolution of the two level set functions. We use a starting structure of four equally spaced holes for the optimisation problem without isotropy constraints and nine equally spaced holes for the problem with isotropy constraints.

\begin{figure*}[p]
    \centering
    \begin{minipage}{.2\linewidth}
    \begin{subfigure}{1.1\linewidth}
        \includegraphics[width=\textwidth]{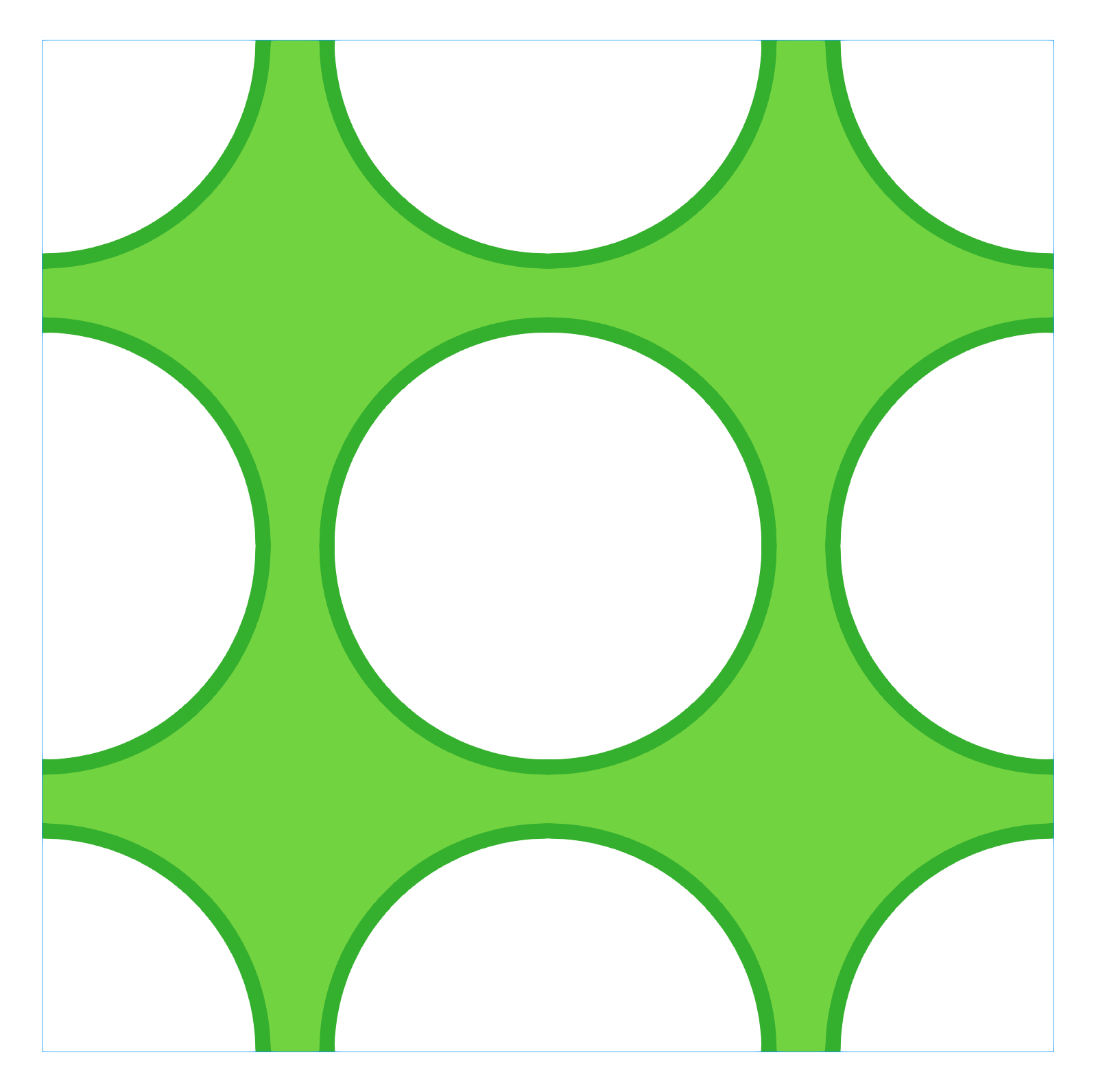}
        \caption{}
        \label{fig:Fig 7a}
    \end{subfigure}
    \end{minipage}
    \Huge{$\mathbf{\rightarrow}$}
    \begin{minipage}{.72\linewidth}
    \centering
    \begin{subfigure}{0.33\linewidth}
        \includegraphics[width=\textwidth]{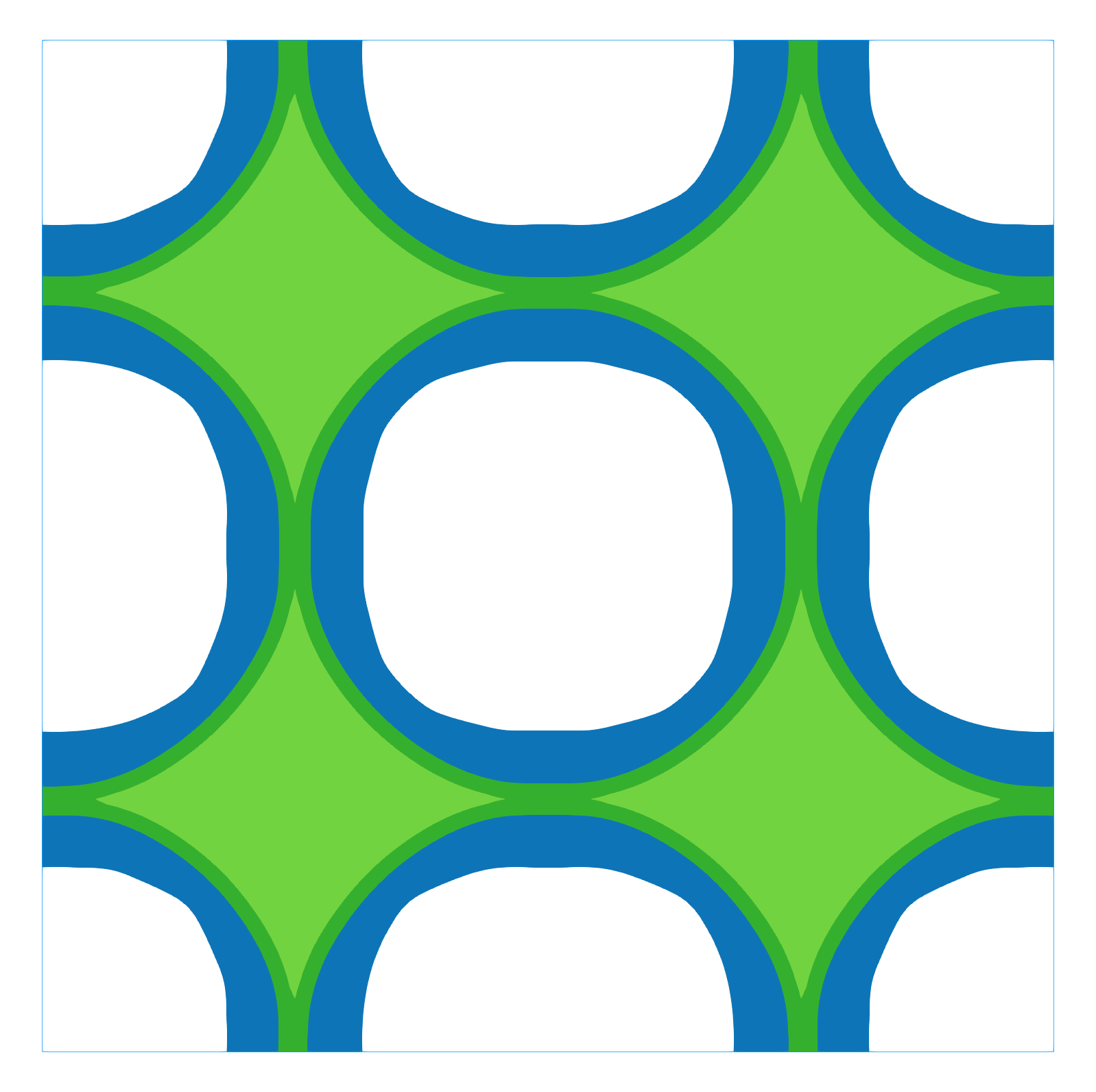}
        \caption{}
        \label{fig:Fig 7b}
    \end{subfigure}\hspace{2em}
    \begin{subfigure}{0.33\linewidth}
        \includegraphics[width=\textwidth]{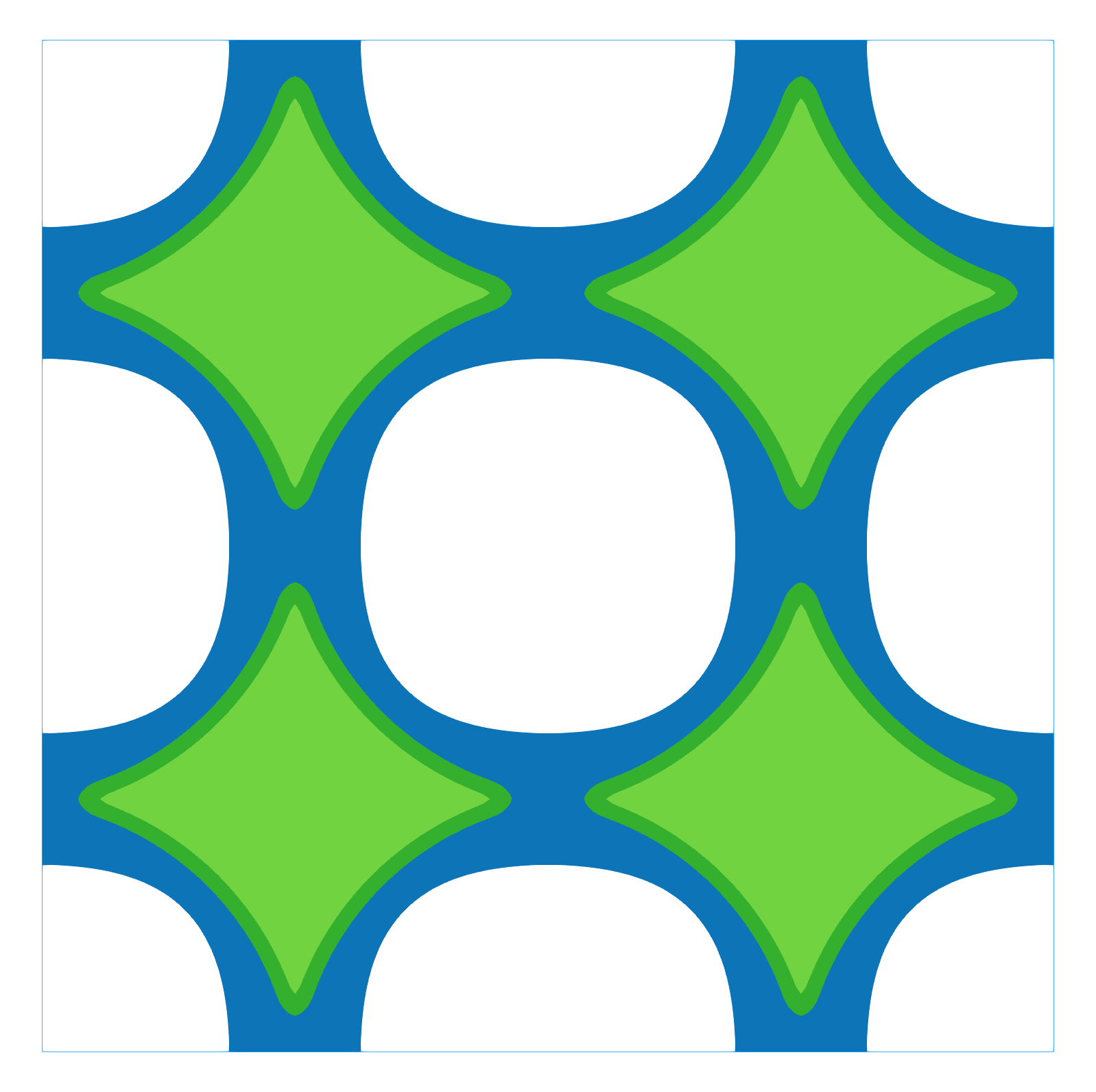}
        \caption{}
        \label{fig:Fig 7c}
    \end{subfigure}
    \begin{subfigure}{0.48\linewidth}
        \includegraphics[width=\textwidth]{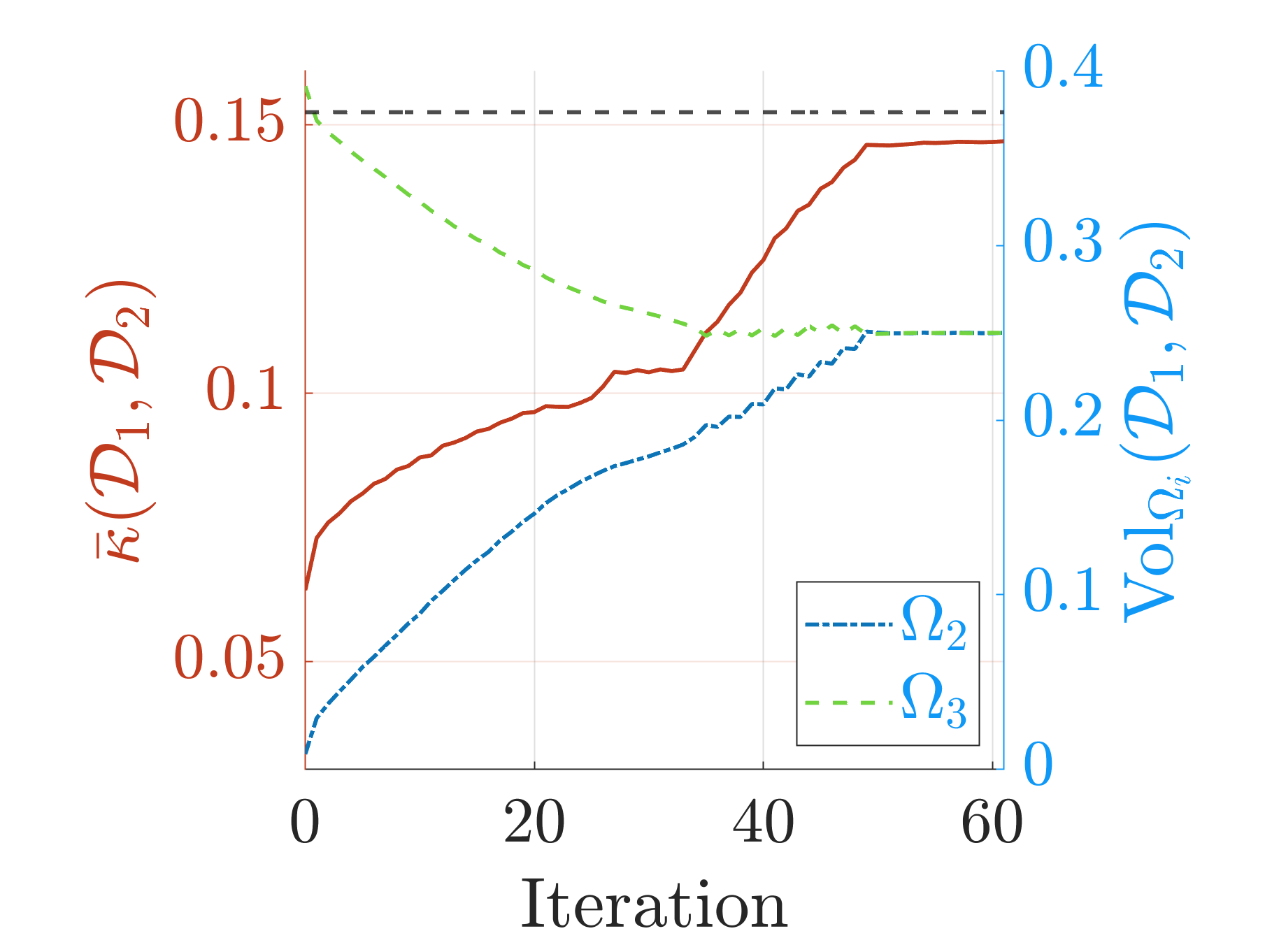}
        \caption{}
        \label{fig:Fig 7d}
    \end{subfigure}
    \begin{subfigure}{0.48\linewidth}
        \includegraphics[width=\textwidth]{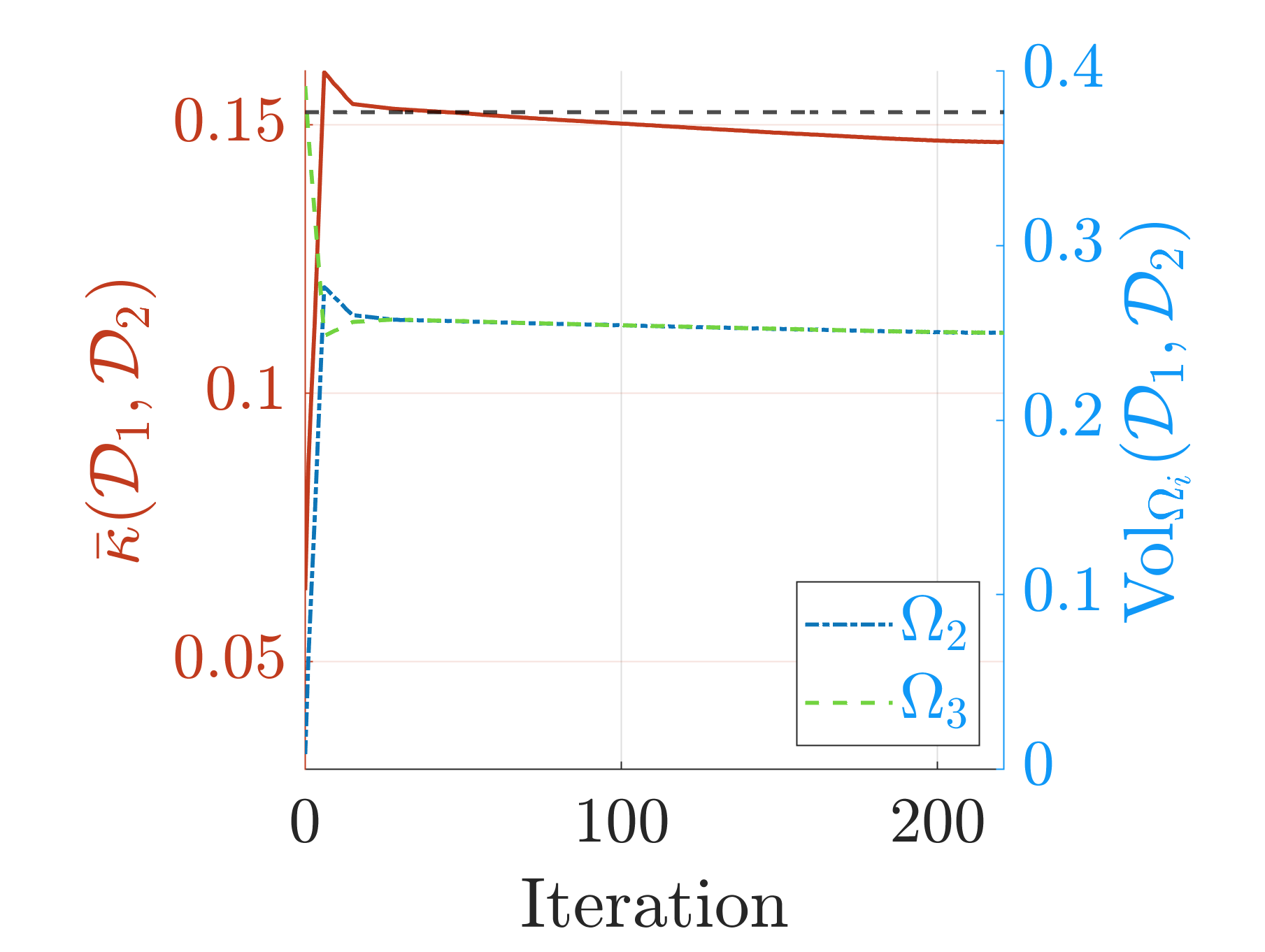}
        \caption{}
        \label{fig:Fig 7e}
    \end{subfigure}
    \end{minipage}
    \caption{Optimisation results for Example 4: maximum bulk modulus multi-phase materials without isotropy constraints. For the starting structure (a), (b) and (c) show the final structures for the Hilbertian projection method and SLP respectively, while (d) and (e) show the respective iteration histories. The Hashin-Shtrikman-Walpole upper bound for the bulk modulus is given by the dashed black line. Note that the volume constraint for each phase is 0.25.}
    \label{fig:Fig 7}

    \centering
    \begin{minipage}{.2\linewidth}
    \begin{subfigure}{1.1\linewidth}
        \includegraphics[width=\textwidth]{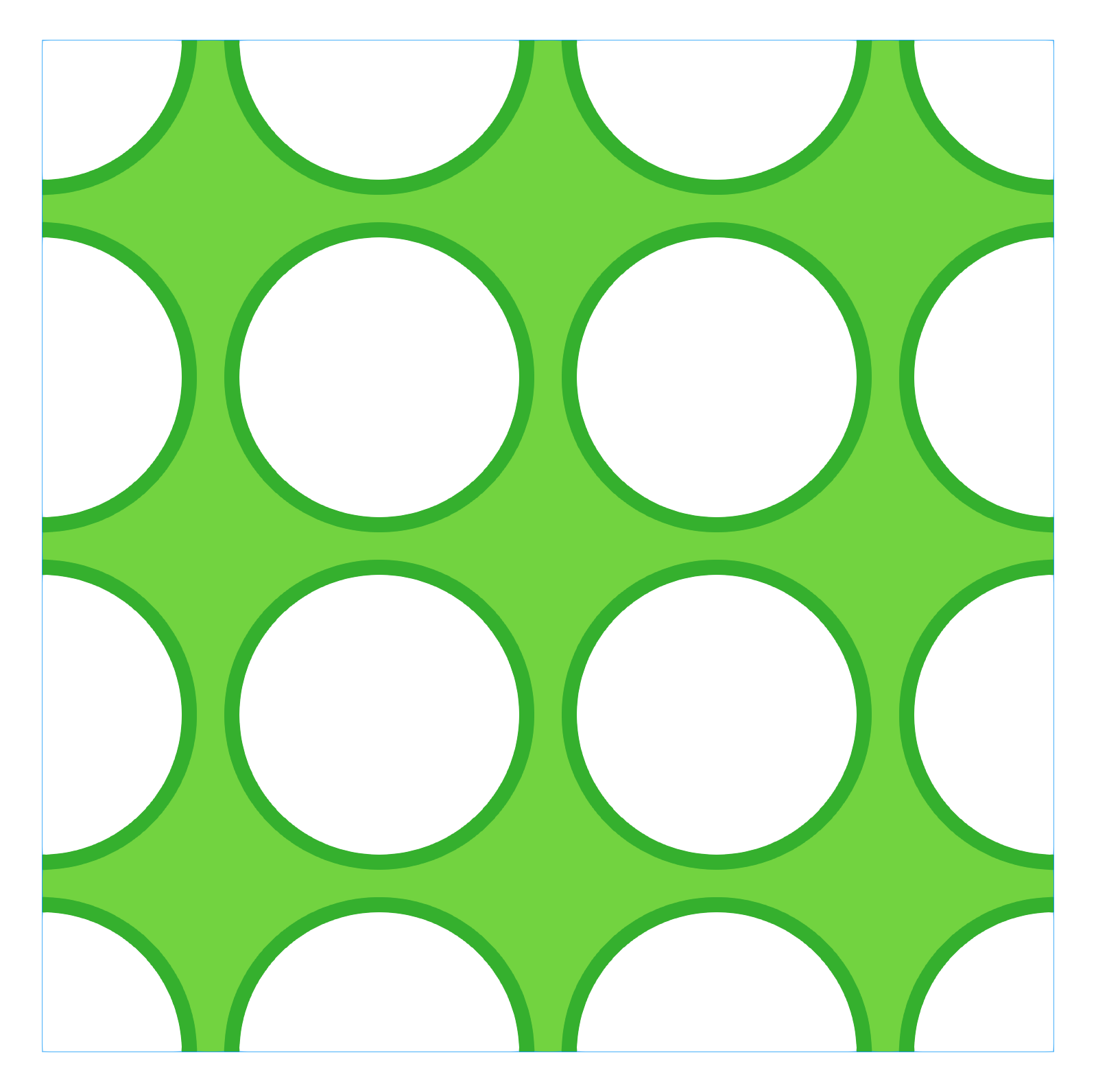}
        \caption{}
        \label{fig:Fig 8a}
    \end{subfigure}
    \end{minipage}
    \Huge{$\mathbf{\rightarrow}$}
    \begin{minipage}{.72\linewidth}
    \centering
    \begin{subfigure}{0.33\linewidth}
        \includegraphics[width=\textwidth]{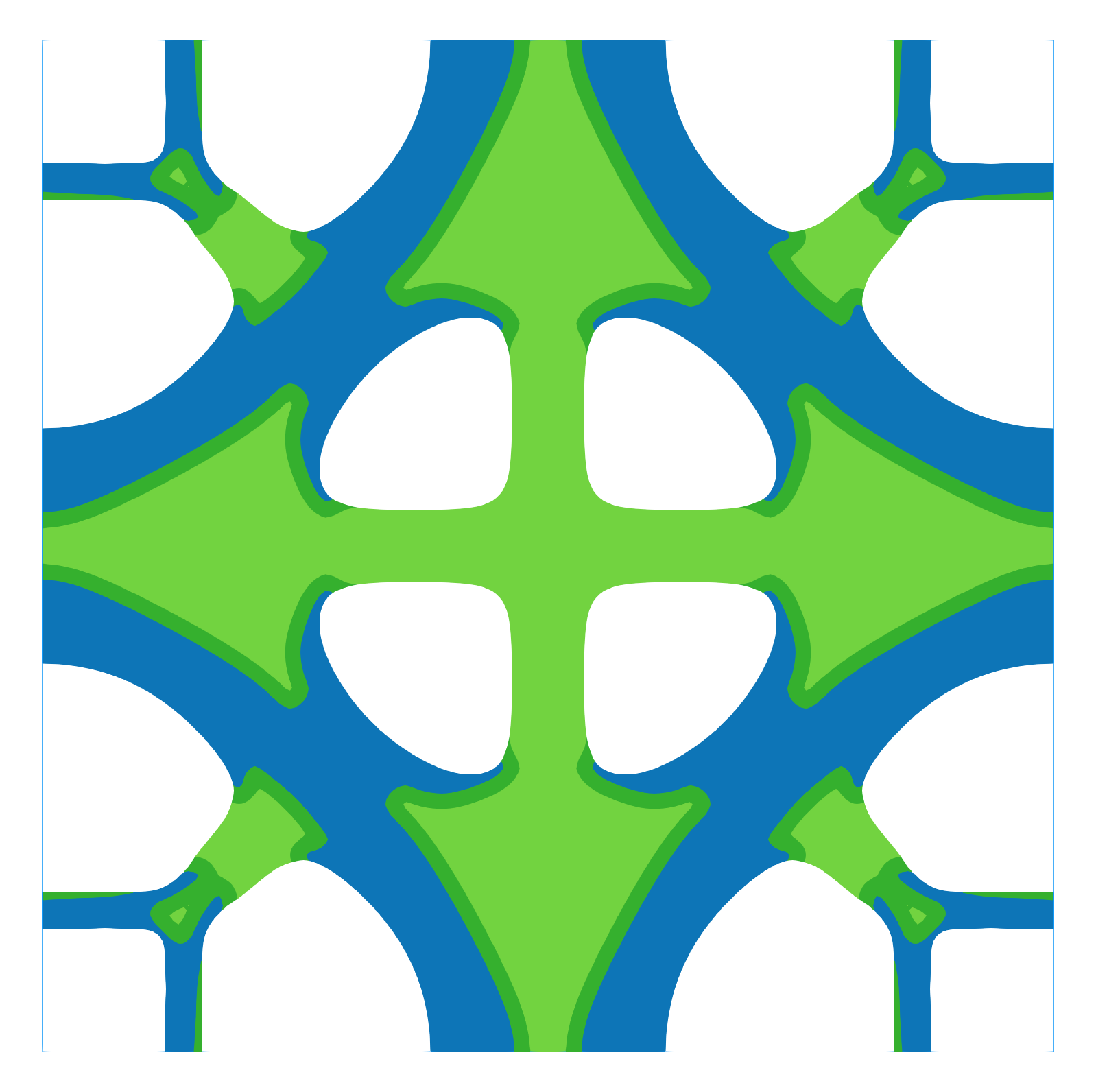}
        \caption{}
        \label{fig:Fig 8b}
    \end{subfigure}\hspace{2em}
    \begin{subfigure}{0.33\linewidth}
        \includegraphics[width=\textwidth]{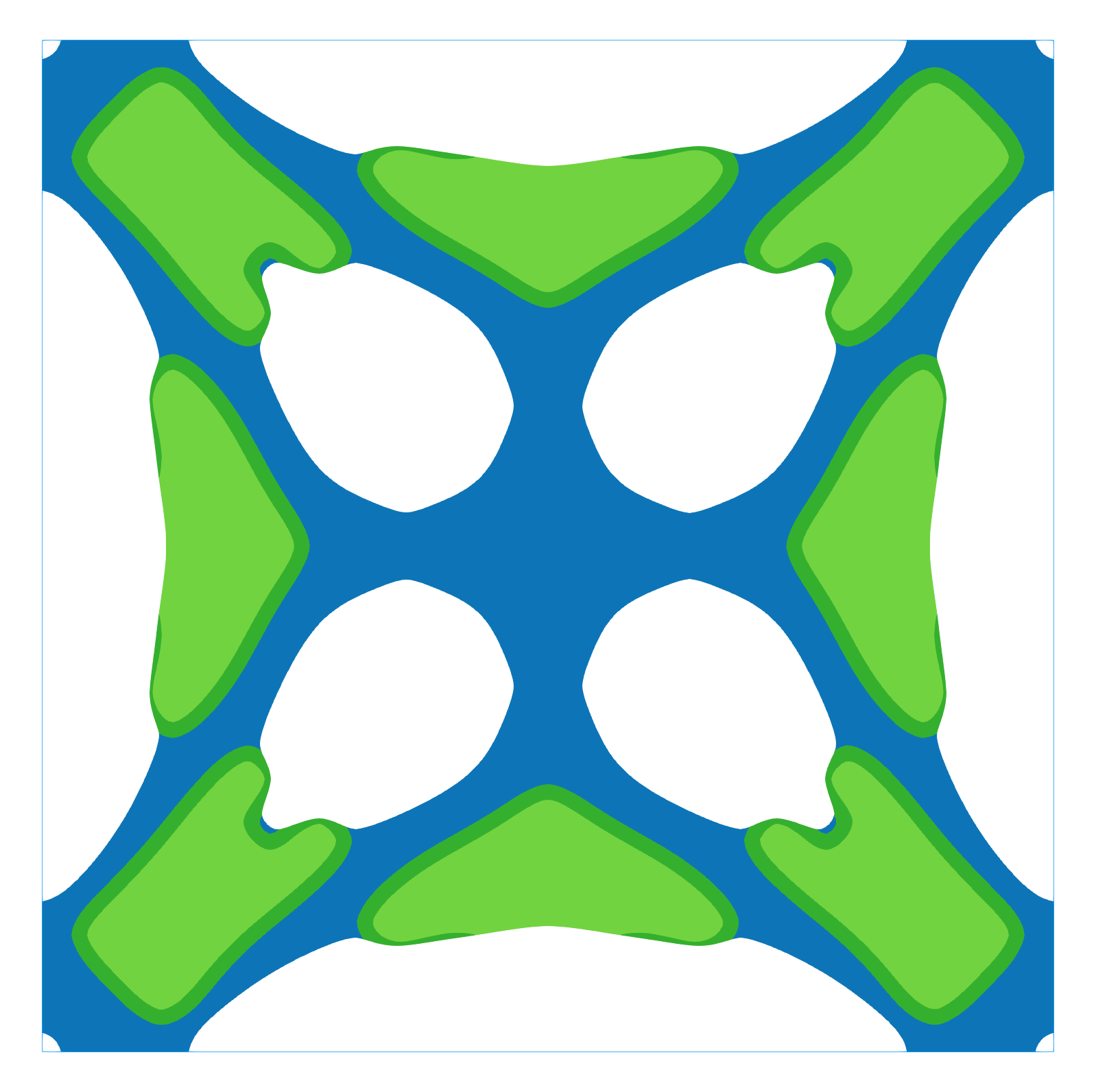}
        \caption{}
        \label{fig:Fig 8c}
    \end{subfigure}
    \begin{subfigure}{0.48\linewidth}
        \includegraphics[width=\textwidth]{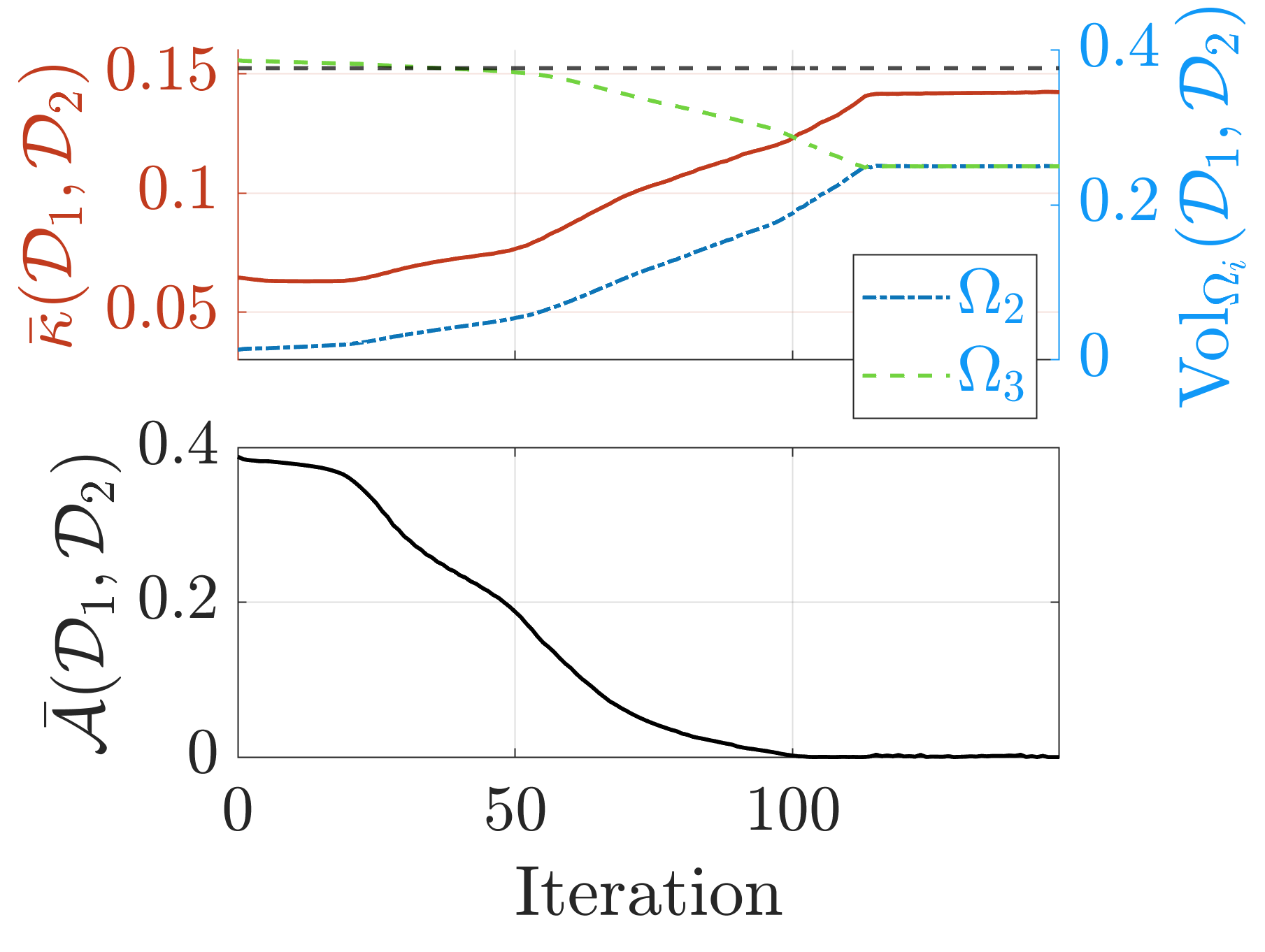}
        \caption{}
        \label{fig:Fig 8d}
    \end{subfigure}
    \begin{subfigure}{0.48\linewidth}
        \includegraphics[width=\textwidth]{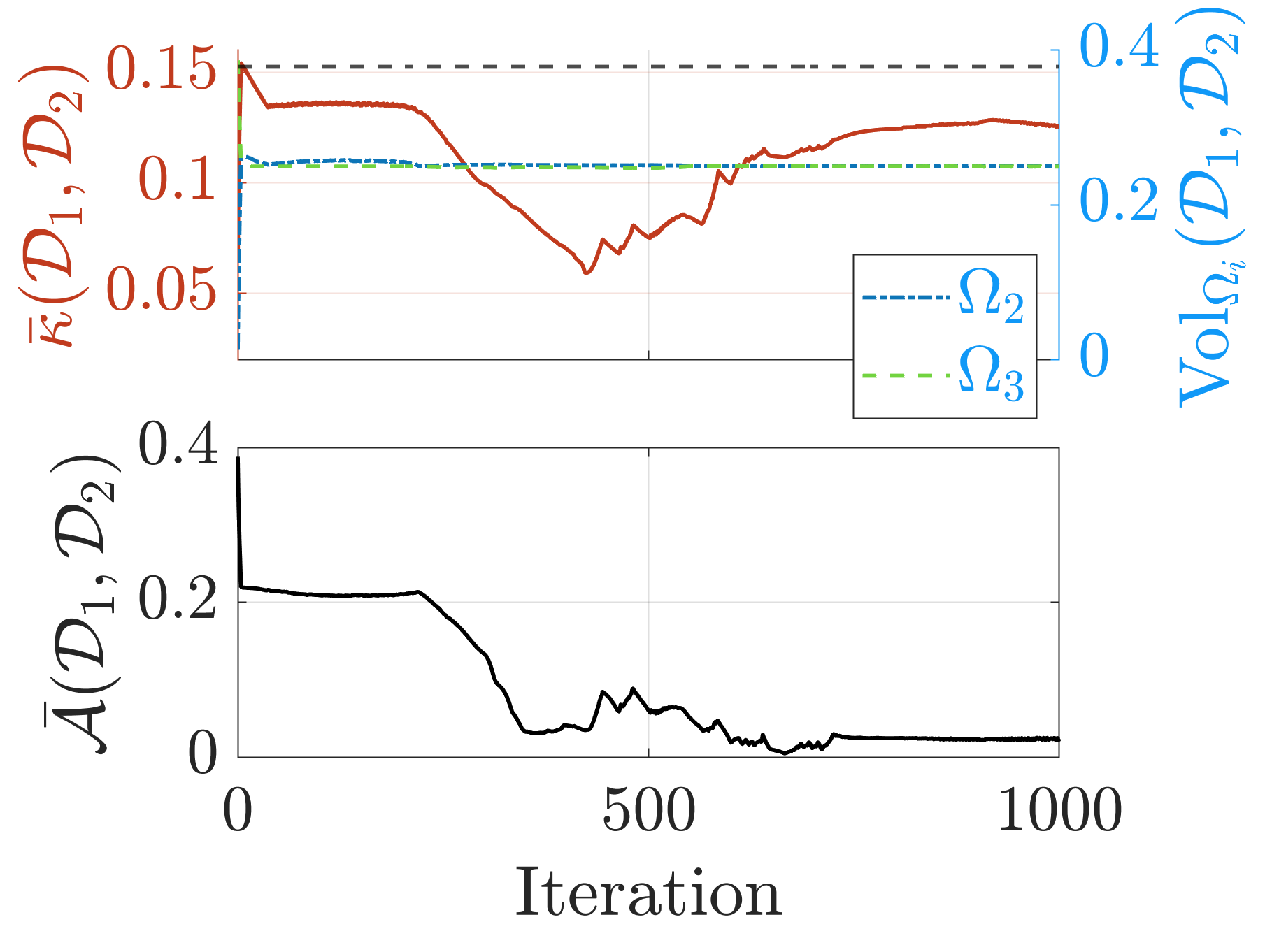}
        \caption{}
        \label{fig:Fig 8e}
    \end{subfigure}
    \end{minipage}
    \caption{Optimisation results for Example 4: maximum bulk modulus multi-phase materials with isotropy constraints. For the starting structure (a), (b) and (c) show the final structures for the Hilbertian projection method and SLP respectively, while (d) and (e) show the respective iteration histories. The Hashin-Shtrikman-Walpole upper bound for the bulk modulus is given by the dashed black line. Note that the volume constraint for each phase is 0.25.}
    \label{fig:Fig 8}
\end{figure*}

Figures \ref{fig:Fig 7} and \ref{fig:Fig 8} show the optimisation results and history without and with isotropy constraints, respectively. We denote the stiff and less stiff material phase by blue and green, respectively, while the smooth interface is given by the dark green overlap. We include the Hashin-Shtrikman-Walpole (HSW) upper bound \citep{WALPOLE1966151} for this problem as the black dashed line in these figures. Table \ref{tab:summary eg. 3} shows a summary of the results.

\begin{table*}[t]
    \centering
    \caption{Summary of optimisation results for Example 4: two-dimensional maximum bulk modulus multi-phase materials. The asterisk denotes failure to converge.}
    \label{tab:summary eg. 3}
    \setlength{\tabcolsep}{0.375em} 
    \begin{tabular}{c|c|c|c|c|c|c|c|c}
       Method & Iso. Const. Type & Fig. & $\bar{\kappa}$ & HSW bound & $\bar{\mathcal{A}}$ & $\operatorname{Vol}_{\Omega_2}$ & $\operatorname{Vol}_{\Omega_3}$ & Iters. \\\hline
        Proj. & N.A. & \ref{fig:Fig 7a}/\ref{fig:Fig 7d} & 0.1469 &   0.1524 & N.A. & 0.2501 & 0.2500 & 61\\
        SLP & N.A. & \ref{fig:Fig 7b}/\ref{fig:Fig 7e} & 0.1468  &  0.1524 & N.A. & 0.2500  &  0.2501 & 221\\
        Proj. & Full set & \ref{fig:Fig 8b}/\ref{fig:Fig 8d} & 0.1424  &  0.1524 & 0.0001  &  0.2500 &   0.2500 & 148\\
        SLP & Aniso. Meas. & \ref{fig:Fig 8c}/\ref{fig:Fig 8e} & 0.1254 &   0.1524 & 0.0236  &  0.2506  &  0.2497 &1000*
    \end{tabular}
\end{table*}

For the case of no isotropy constraints, both the Hilbertian projection method and SLP converge to roughly 96\% of the HSW bound while the resulting structures (Fig. \ref{fig:Fig 7b} and \ref{fig:Fig 7c}) are geometrically similar apart from the thin interface that presents in the results for the Hilbertian projection method. The difference between our results and the theoretical upper bound is likely due to the use of the {approximate} formula for the shape derivative. Indeed, results from \cite{10.1051/cocv/2013076_2014} showed that the {approximate} formula yields slightly less optimal results than the {true} or {Jacobian-free} counterpart. The iteration history for the Hilbertian projection method (Fig. \ref{fig:Fig 7d}) is fairly smooth while the history for SLP (Fig. \ref{fig:Fig 7e}) moves rapidly at the beginning of the optimisation to satisfy the volume constraints and increase the objective. As previously noted, this is likely due to the trust region constraints that, in this case, need to be chosen to be more conservative.

With the addition of isotropy constraints, we again find that the Hilbertian projection method works well without significant parameter tuning. In 148 iterations the optimisation algorithm is able to satisfy all constraints and obtain a final objective value that is 93.44\% of the HS bound. In contrast, SLP does not converge within 1000 iterations.

\section{Conclusion}\label{sec: concl}
In this paper we have presented a Hilbertian extension of the projection method for constrained level set-based topology optimisation. At its core, the method relies on the Hilbertian extension-regularisation method in which a set of identification problems are solved over a Hilbert space $H$ on $D$ with inner product $\langle\cdot,\cdot\rangle_{H}$. This procedure naturally extends shape sensitivities onto the bounding domain $D$ and enriches them with the regularity of $H$. For a constrained optimisation problem the projection method framework aims for the best first-order improvement of the objective in addition to first-order improvement of the constraints. These requirements for the projection method may then be reposed in the Hilbertian framework in terms of $H$ and $\langle\cdot,\cdot\rangle_{H}$. We satisfy these reposed requirements by defining the normal velocity of the level set function as a linear combination of the orthogonal projection operator applied to the extended objective shape sensitivity and basis functions for the extended constraint shape sensitivities. Owing to the Hilbertian extension-regularisation of shape sensitivities, the chosen normal velocity is already extended onto the bounding domain $D$ and endowed with the regularity of $H$.

To demonstrate the Hilbertian projection method for constrained level set-based topology optimisation we have solved several example microstructure optimisation problems with multiple constraints. We showed that the Hilbertian projection method successfully handled all of these optimisation problems with little-to-no tuning of the parameter $\alpha_\mathrm{min}$ that controls the balance between improving the objective and constraints. The Hilbertian projection method also naturally handles linearly dependent constraint shape sensitivities. Such linear dependencies often appear in microstructure optimisation and multi-phase optimisation problems.

We found that our method performs well when compared to a Hilbertian sequential linear programming (SLP) method \citep{10.1016/bs.hna.2020.10.004_978-0-444-64305-6_2021}. For problems only involving volume constraint(s), both methods converged to appropriate optimised microstructures. However, SLP did not successfully solve some of the more complex example optimisation problems. These results demonstrate the capacity of the Hilbertian projection method and likely other projection/null space methods \cite[e.g.,][]{10.3934/dcdsb.2019249,10.1051/cocv/2020015_2020} for solving constrained level set-based topology optimisation problems.

Applying multiple constraints is challenging in level set-based topology optimisation owing to the reliance on implicitly defined domains. Alongside other recent work \citep{10.3934/dcdsb.2019249,10.1051/cocv/2020015_2020, 10.1016/bs.hna.2020.10.004_978-0-444-64305-6_2021}, our proposed Hilbertian projection method makes significant progress towards improving the capacity of conventional level set-based methods for constrained topology optimisation. Furthermore, due to its generality, the Hilbertian projection method is not confined to microstructure optimisation. It may be applied to any topology optimisation problem to be solved in a level-set framework, including macroscopic or multi-physics problems. Inequality constraints could likely be incorporated into the method using slack variables. In addition, the method is not confined to Eulerian level set methods and can be readily applied to Lagrangian or body-fitted level set methods \cite[e.g.,][]{10.1016/j.cma.2014.08.028_2014}. These extensions could be considered in future work. 

\backmatter

\bmhead{Acknowledgments}
This work was supported by the Australian Research Council through the Discovery Grant scheme (DP220102759). Computational resources used in this work were provided by the eResearch Office, Queensland University of Technology. The first author is supported by a QUT Postgraduate Research Award and a Supervisor Top-Up Scholarship. The authors would also like to thank the anonymous reviewers for their constructive comments that have resulted in improvements to the manuscript.

\bmhead{Conflict of interest}
The authors have no competing interests to declare that are relevant to the content of this article.

\bmhead{Replication of Results}
As this work is a part of a new project, we do not provide the source code. However, we do provide Algorithm \ref{alg:hpm} to help readers reproduce results. In addition, we provide all parameter values along with an iteration history for all problems. Interested readers can contact the authors for further information. 





\begin{appendices}

\section{Proof of Lemma 2}\label{secA1}
We proceed via C\'ea's method \citep{10.1051/m2an/1986200303711_1986}: suppose we define the Lagrangian $\mathcal{L}$ to be
\begin{align}
    \mathcal{L}(\Omega,\tilde{\boldsymbol{u}}^{(ij)},\tilde{\boldsymbol{u}}^{(kl)})&\nonumber\\*&\hspace{-2cm}=\int_\Omega{C_{pqrs}({\varepsilon}_{pq}(\tilde{\boldsymbol{u}}^{(ij)})+\bar{\varepsilon}_{pq}^{(ij)})\bar{\varepsilon}_{rs}^{(kl)}}~\mathrm{d}{\Omega}\nonumber\\*&\hspace{-2cm}+\int_\Omega C_{pqrs}(\varepsilon_{pq}(\tilde{\boldsymbol{u}}^{(ij)})+\bar{\varepsilon}_{pq}^{(ij)}){\varepsilon}_{rs}(\tilde{\boldsymbol{u}}^{(kl)})~\mathrm{d}{\Omega}.
\end{align}
We do not include auxiliary fields as it turns out that the problem is self-adjoint.

Under a variation $\varr{\boldsymbol{u}}^{(ij)}$ of $\tilde{\boldsymbol{u}}^{(ij)}$ with $ij\neq kl$, the corresponding variation of $\mathcal{L}$ is
\begin{align*}
    &\frac{\delta\mathcal{L}}{\varr{\boldsymbol{u}}^{(ij)}}=\int_\Omega{C_{pqrs}\varr{u}_{p,q}^{(ij)}\bar{\varepsilon}_{rs}^{(kl)}}~\mathrm{d}{\Omega}\\&\hspace{1.5cm}+\int_\Omega{C_{pqrs}\varr{u}_{p,q}^{(ij)}{\varepsilon}_{rs}(\tilde{\boldsymbol{u}}^{(kl)})}~\mathrm{d}{\Omega}\\&~=\int_\Omega{C_{pqrs}(\bar{\varepsilon}_{rs}^{(kl)}+{\varepsilon}_{rs}(\tilde{\boldsymbol{u}}^{(kl)}))\varr{u}_{p,q}^{(ij)}}~\mathrm{d}{\Omega}\\&~=\int_\Omega{\sigma_{pq}^{(kl)}\varr{u}_{p,q}^{(ij)}}~\mathrm{d}{\Omega}\\&~=-\int_\Omega{\sigma_{pq,q}^{(kl)}\varr{u}_{p}^{(ij)}}~\mathrm{d}{\Omega}+\int_{\partial\Omega}{\sigma_{pq}^{(kl)}n_q\varr{u}_{p}^{(ij)}}~\mathrm{d}{\Gamma}
\end{align*}%
where symmetry of material coefficients has been used along with integration by parts. Requiring $\mathcal{L}$ to be stationary gives the state equations for stress under loading $\bar{\varepsilon}_{rs}^{kl}$. In particular, allowing arbitrary $\varr{u}_p^{(ij)}$ within $\Omega$ gives $\sigma_{pq,q}^{(kl)}=0$ in $\Omega$ and allowing arbitrary $\varr{u}_p^{(ij)}$ on $\partial\Omega$ gives $\sigma_{pq}^{(kl)}n_q=0$ on $\partial\Omega$.

Next, under a variation $\varr{\boldsymbol{u}}^{(kl)}$ of $\tilde{\boldsymbol{u}}^{(kl)}$ with $ij\neq kl$, the corresponding variation of $\mathcal{L}$ is
\begin{align*}
    &\frac{\delta\mathcal{L}}{\varr{\boldsymbol{u}}^{(kl)}}=\int_\Omega{C_{pqrs}({\varepsilon}_{pq}(\tilde{\boldsymbol{u}}^{(ij)})+\bar{\varepsilon}_{pq}^{(ij)})\varr{u}_{r,s}^{(kl)}}~\mathrm{d}{\Omega}\\
    &~=\int_\Omega{\sigma_{rs}^{(ij)}\varr{u}_{r,s}^{(kl)}}~\mathrm{d}{\Omega}\\
    &~=-\int_\Omega{\sigma_{rs,s}^{(ij)}\varr{u}_{r}^{(kl)}}~\mathrm{d}{\Omega}+\int_{\partial\Omega}{\sigma_{rs}^{(ij)}n_s\varr{u}_{r}^{(kl)}}~\mathrm{d}{\Gamma}
\end{align*}%
where symmetry of material coefficients has been used along with integration by parts. Requiring $\mathcal{L}$ to be stationary gives the state equations for stress under loading $\bar{\varepsilon}_{pq}^{(ij)}$. It should be noted that when $ij=kl$, we can apply the product rule which results in the state equations for stress under a constant macroscopic strain $\bar{\varepsilon}_{pq}^{(ij)}$.

Together, the above results give the state equations for the fields $\tilde{\boldsymbol{u}}^{(ij)}$ and $\tilde{\boldsymbol{u}}^{(kl)}$, as required. 

We also require that the Lagrangian $\mathcal{L}$ equals the objective at the solution to the state equations. Indeed, at the solution to the state equations we obtain
\begin{equation}
    \mathcal{L}(\Omega) = \bar{C}_{ijkl}(\Omega),
\end{equation}
as required.

The shape derivative of $\mathcal{L}$ at fixed $\tilde{\boldsymbol{u}}^{(ij)}$ and $\tilde{\boldsymbol{u}}^{(kl)}$ can then be calculated using Lemma \ref{lemma: omega int scalar} to be
\begin{equation}
\begin{aligned}
    C^{\prime}(\Omega)(\boldsymbol{\theta})&=\mathcal{L}^{\prime}(\Omega)(\boldsymbol{\theta})\rvert_{\tilde{\boldsymbol{u}}^{(ij)},\tilde{\boldsymbol{u}}^{(kl)}}\\
    &= \int_{\partial\Omega} C_{pqrs}({\varepsilon}_{pq}(\tilde{\boldsymbol{u}}^{(ij)})+\bar{\varepsilon}_{pq}^{(ij)})\\&\quad\times({\varepsilon}_{rs}(\tilde{\boldsymbol{u}}^{(kl)})+\bar{\varepsilon}_{rs}^{(kl)})~\boldsymbol{\theta}\cdot\boldsymbol{n}~\mathrm{d}{\Gamma}.
\end{aligned}
\end{equation}

\section{Proof of Lemma 3}\label{secA2}
Similarly to Appendix \ref{secA1}, suppose we define the Lagrangian $\mathcal{L}$ to be
\begin{align}
    \mathcal{L}(\mathcal{D}_1,\mathcal{D}_2,\tilde{\boldsymbol{u}}^{(ij)},\tilde{\boldsymbol{u}}^{(kl)})&\nonumber\\*&\hspace{-2.5cm}=\int_D{C_{pqrs}({\varepsilon}_{pq}(\tilde{\boldsymbol{u}}^{(ij)})+\bar{\varepsilon}_{pq}^{(ij)})\bar{\varepsilon}_{rs}^{(kl)}}~\mathrm{d}{\Omega}\nonumber\\*&\hspace{-2.5cm}+\int_D C_{pqrs}(\varepsilon_{pq}(\tilde{\boldsymbol{u}}^{(ij)})+\bar{\varepsilon}_{pq}^{(ij)}){\varepsilon}_{rs}(\tilde{\boldsymbol{u}}^{(kl)})~\mathrm{d}{\Omega}.
\end{align}
where $C_{pqrs}=C_{pqrs}(d_{\mathcal{D}_1},d_{\mathcal{D}_2})$. As previously, stationarity of $\mathcal{L}$ under variations $\varr{\boldsymbol{u}}^{(ij)}$ and ${\varr{\boldsymbol{u}}^{(kl)}}$ retrieve the state equations and at the solution to the state equations $\mathcal{L}(\Omega) = \bar{C}_{ijkl}(\Omega)$. The given Lagrangian therefore satisfies the requirements for C\'ea's method.

Using differentiability of the signed distance function $d_{\mathcal{D}_i}$ \cite[Lemma 2.4, Proposition 2.5,][]{10.1051/cocv/2013076_2014} and the chain rule, we have
\begin{equation}
\begin{aligned}
    &\bar{C}_{ijkl}^{\prime}(\mathcal{D}_1,\mathcal{D}_2)(\boldsymbol{\theta}_1)\\&~=\int_{D}d_{\mathcal{D}_1}^\prime(\boldsymbol{\theta}_1)\frac{\partial C_{pqrs}}{\partial d_{\mathcal{D}_1}}({\varepsilon}_{pq}(\tilde{\boldsymbol{u}}^{(ij)})+\bar{\varepsilon}_{pq}^{(ij)})\\&\hspace{2.6cm}\times({\varepsilon}_{rs}(\tilde{\boldsymbol{u}}^{(kl)})+\bar{\varepsilon}_{rs}^{(kl)})~\mathrm{d}{\Omega}
\end{aligned}
\end{equation}
where $d_{\mathcal{D}_i}^\prime$ is the shape derivative of the signed distance function for $\boldsymbol{x}\in D\setminus\Sigma$ given by 
\begin{equation}
    d^\prime_{\mathcal{D}_1}(\boldsymbol{x})=-\boldsymbol{\theta}(p_{\partial\mathcal{D}_1}(\boldsymbol{x}))\cdot\boldsymbol{n}(p_{\partial\mathcal{D}_1}(\boldsymbol{x}))
\end{equation}
where $p_{\partial\mathcal{D}_1}(\boldsymbol{x})$ is the projection of a point $x$ onto the boundary $\partial\mathcal{D}_1$ and $\Sigma$ is the set of points in the skeleton of $\partial\mathcal{D}_1$ \cite[Definition 2.3,][]{10.1051/cocv/2013076_2014}.

Using this and the {Jacobian-free} coarea formula \cite[Corollary 2.13, Equation 2.15,][]{10.1051/cocv/2013076_2014} results in
\begin{equation}\label{eqn: jacobfree}
\begin{aligned}
    &\bar{C}_{ijkl}^{\prime}(\mathcal{D}_1,\mathcal{D}_2)(\boldsymbol{\theta}_1)\\&~=-\int_{\partial\mathcal{D}_1}  \boldsymbol{\theta}_1\cdot\boldsymbol{n}\int_{\mathrm{ray}_{\partial\mathcal{D}_1}(\boldsymbol{x})\cap D} H^\prime_\eta(d_{\mathcal{D}_1})\frac{\partial C_{pqrs}}{\partial g}\\&~~~\times({\varepsilon}_{pq}(\tilde{\boldsymbol{u}}^{(ij)})+\bar{\varepsilon}_{pq}^{(ij)})({\varepsilon}_{rs}(\tilde{\boldsymbol{u}}^{(kl)})+\bar{\varepsilon}_{rs}^{(kl)})~\mathrm{d}{z}~\mathrm{d}{\Gamma}
\end{aligned}
\end{equation}
where $g(x)=H_\eta(x)$. 

Finally, the support of $H^\prime_\eta(x)$ is $\lvert x\rvert<2\eta$, so the integral over $\mathrm{ray}_{\partial\mathcal{D}_1}(\boldsymbol{x})\cap D$ is restricted to a tubular region about $\partial\mathcal{D}_1$ \citep{10.1051/cocv/2013076_2014}. Therefore, for small $\eta$, we may assume that 
\begin{align}
    \varepsilon_{pq}(\tilde{\boldsymbol{u}}^{(ij)})(\boldsymbol{z})&\approx\varepsilon_{pq}(\tilde{\boldsymbol{u}}^{(ij)})(\boldsymbol{y})
    \intertext{and}
    d_{\mathcal{D}_2}(\boldsymbol{z})&\approx d_{\mathcal{D}_2}(\boldsymbol{y})
\end{align}
where $\boldsymbol{z}\in \mathrm{ray}_{\partial\mathcal{D}_1}(\boldsymbol{x})\cap D$ and $\boldsymbol{y}\in\partial\mathcal{D}_1$. In addition, the derivative $\frac{\partial C_{pqrs}}{\partial g}(H_\eta(d_{\mathcal{D}_1}))$ is independent of $d_{\mathcal{D}_1}$ by Equation \ref{eqn: multiphase interp}. Equation \ref{eqn: jacobfree} can therefore be written as
\begin{equation}
\begin{aligned}
    &\bar{C}_{ijkl}^{\prime}(\mathcal{D}_1,\mathcal{D}_2)(\boldsymbol{\theta}_1)\\&~\approx-\int_{\partial\mathcal{D}_1}  \boldsymbol{\theta}_1\cdot\boldsymbol{n}~\frac{\partial C_{pqrs}}{\partial g}\left({\varepsilon}_{pq}(\tilde{\boldsymbol{u}}^{(ij)})+\bar{\varepsilon}_{pq}^{(ij)}\right)\\&\quad\quad\quad\quad\times\left({\varepsilon}_{rs}(\tilde{\boldsymbol{u}}^{(kl)})+\bar{\varepsilon}_{rs}^{(kl)}\right)\\&\quad\quad\quad\quad\times\left[\int_{\mathrm{ray}_{\partial\mathcal{D}_1}(\boldsymbol{x})\cap D} H^\prime_\eta(d_{\mathcal{D}_1})~\mathrm{d}{z}\right]~\mathrm{d}{\Gamma}
\end{aligned}
\end{equation}
Finally, it can be shown using elementary vector calculus that
\begin{equation}
    \int_{\mathrm{ray}_{\partial\mathcal{D}_1}(\boldsymbol{x})\cap D} H^\prime_\eta(d_{\mathcal{D}_1}(z))~\mathrm{d}{z} = 1,
\end{equation}
which completes this portion of the proof.

For Equation \ref{eqn: volume 1}, \cite[Corollary 2.8,][]{10.1051/cocv/2013076_2014} and the chain rule gives
\begin{equation}
\begin{aligned}
    &\operatorname{Vol}^{\prime}_{\Omega_1}(\mathcal{D}_1,\mathcal{D}_2)(\boldsymbol{\theta}_1)\\&~=\int_{D} d_{\mathcal{D}_1}^\prime(\boldsymbol{\theta}_1) H^\prime_{\eta}(d_{\mathcal{D}_1}) H_{\eta}(d_{\mathcal{D}_2})~\mathrm{d}\Omega.
\end{aligned}
\end{equation}
As previously, shape differentiability of $d_{\mathcal{D}_i}$ along with the {Jacobian-free} coarea formula gives
\begin{equation}
\begin{aligned}
    &\operatorname{Vol}^{\prime}_{\Omega_1}(\mathcal{D}_1,\mathcal{D}_2)(\boldsymbol{\theta}_1)=\int_{\partial\mathcal{D}_1}  \boldsymbol{\theta}_1\cdot\boldsymbol{n}\\&~\times\left[\int_{\mathrm{ray}_{\partial\mathcal{D}_1}(\boldsymbol{x})\cap D} H^\prime_{\eta}(d_{\mathcal{D}_1}) H_{\eta}(d_{\mathcal{D}_2})~\mathrm{d}s\right]~\mathrm{d}\Gamma.
\end{aligned}
\end{equation}
The prior approximations then give the result
\begin{equation}
    \operatorname{Vol}^{\prime}_{\Omega_1}(\mathcal{D}_1,\mathcal{D}_2)(\boldsymbol{\theta}_1)\approx\int_{\partial\mathcal{D}_1}  \boldsymbol{\theta}_1\cdot\boldsymbol{n}~H_{\eta}(d_{\mathcal{D}_2})\mathrm{d}\Gamma.
\end{equation}
which concludes the proof.

\end{appendices}


\bibliography{main.bib}

\end{document}